\documentclass[review]{elsarticle}

\usepackage{bm,amssymb,amsmath,mathrsfs}
\usepackage{amsthm}
\usepackage{lineno}
\usepackage[usenames]{color}

\usepackage{dcolumn}

\newcommand{\bq}{\begin{equation}}
\newcommand{\eq}{\end{equation}}
\newcommand{\bc}{\begin{center}}
\newcommand{\ec}{\end{center}}
\newcommand{\bit}{\begin{itemize}}
\newcommand{\eit}{\end{itemize}}
\newcommand{\ben}{\begin{enumerate}}
\newcommand{\een}{\end{enumerate}}
\newcommand{\pprime}{{\prime\prime}}

\theoremstyle{plain}
\newtheorem{theorem}{Theorem}[section]
\newtheorem*{theorem*}{Theorem}
\newtheorem{proposition}[theorem]{Proposition}
\newtheorem{lemma}[theorem]{Lemma}
\newtheorem{corollary}[theorem]{Corollary}
\newtheorem{remark}[theorem]{Remark}
\newtheorem{definition}[theorem]{Definition}

\renewcommand\thefigure{\arabic{figure}}

\begin{document}

\journal{Expositiones Mathematicae}

\begin{frontmatter}

\title{Multiparameter Fuss--Catalan numbers with application to algebraic equations}

\author[cc]{S.~R.~Mane}
\ead{srmane001@gmail.com}

\address[cc]{Convergent Computing Inc., P.~O.~Box 561, Shoreham, NY 11786, USA}

\begin{abstract}
We present an exposition on the Fuss--Catalan numbers, which are a generalization of the well known Catalan numbers.
The literature on the subject is scattered
(especially for the case of multiple independent parameters, as will be explained in the text),
with overlapping definitions by different authors and duplication of proofs.
This paper collects the main theorems and identities, with a consistent notation.
Contact is made with the works of numerous authors, including the early works of Lambert and Euler.
We demonstrate the application of the formalism to solve algebraic equations by infinite series. 
Our main result in this context is a new necessary and sufficient formula for the domain of absolute convergence 
of the series solutions of algebraic equations, which corrects and extends previous work in the field.
Some historical material is placed in an Appendix.
  
\end{abstract}

\vskip 0.25in
\end{frontmatter}

\setcounter{equation}{0}
\section{\label{sec:intro} Introduction}
We employ the standard notation 
$\mathbb{C}$ for the complex numbers,
$\mathbb{R}$ for the reals and
$\mathbb{N}$ for the natural numbers $\{0,1,2,\dots\}$.
The Catalan numbers are defined, for $t \in \mathbb{N}$, as
\bq
\label{eq:def_cat}
C_t = \frac{1}{t+1}\binom{2t}{t} = \frac{(2t)!}{(t+1)!t!} \,.
\eq
(It is more usual to write $C_n$ instead of $C_t$,
but there are too many other meanings for $n$ later in this paper,
so to avoid confusion I shall employ $t$ not $n$.)
Catalan numbers have been claimed to be the most ubiquitions numbers in combinatorics,
second only to the binomial coefficients themselves,
e.g.~see the text by Stanley \cite{Stanley_book}.
It is also shown in \cite{Stanley_book} that Catalan numbers 
are the solutions to numerous counting problems.
For example, Euler showed that $C_t$ gives the number of triangulations of a convex $(t+2)$-gon.
(See \cite{Stanley_book} for an extensive historical description,
including quotes from the correspondence of Euler and other authors.)
A generalization of the Catalan numbers,
known as the {\em Fuss--Catalan numbers}, 
are the principal objects of interest in this paper.
(They are named after Nicolas Fuss and Eug{\`e}ne Charles Catalan;
see the text by Graham et al.~\cite{GrahamKnuthPatashnik} for a historical discussion.)
First let $m,t\in\mathbb{N}$ and define
\bq
\label{eq:fc_def_int}
A_t(m) = \frac{1}{(m-1)t+1}\binom{mt}{t} \,.
\eq
The Catalan numbers are the special case where $m=2$.
Then $A_t(m+1,1)$ counts the number of dissections of a convex $(mt+2)$-gon into regions that are $(m+2)$-gons
\cite[exercise A14]{Stanley_book}.
The term `dissection' means the diagonals joining the vertices of the $(mt+2)$-gon,
to form the $(m+2)$-gon, do not intersect in their interiors.
See \cite{Stanley_book} for details.
However, our interest extends beyond combinatorics.
We require a definition not restricted to integers.
We define the Fuss--Catalan numbers, for $\mu,r\in\mathbb{C}$ and $t\in\mathbb{N}$,
as $A_0(\mu,r) := 1$ and for $t\ge1$ via
\bq
\label{eq:fc_def_intro}
A_t(\mu,r) := \frac{r}{t!}\prod_{j=1}^{t-1}(t\mu+r-j) \,.
\eq
The above expression is well-defined for all $\mu,r\in\mathbb{C}$. 
We can employ the Gamma function to write
\bq
\label{eq:fc_def_gamma}
A_t(\mu,r) = r\, \frac{\Gamma(t\mu+r)}{\Gamma(t+1)\Gamma(t(\mu-1)+r+1)} \,.
\eq
However, this expression contains potential $0/0$ problems if the arguments of the Gamma functions equal zero or a negative integer.
We shall employ eq.~\eqref{eq:fc_def_intro} in this paper.
There are other equivalent definitions of the Fuss--Catalan numbers;
for example the text by Graham et al.~\cite{GrahamKnuthPatashnik} employs generalized binomial coefficients.
All of the applications in this paper will in fact treat only $\mu,r\in\mathbb{R}$.
Note that eq.~\eqref{eq:fc_def_int} is the special case $\mu=m$ and $r=1$.

Concomitant with the Fuss--Catalan numbers is their generating function,
and in fact we shall mostly work with the generating function below
(here $z\in\mathbb{C}$)
\bq
\label{eq:b_genfcn_intro}
B_\mu(r;z) = \sum_{t=0}^\infty A_t(\mu,r)z^t \,.
\eq
It is proved in \cite{GrahamKnuthPatashnik} that $B_\mu(r;z)$ has the remarkable property $B_\mu(1;z)^r = B_\mu(r;z)$.
Also, and very importantly, $B_\mu(1;z)$ satisfies the following equation for $f(z)$ 
(again, see \cite{GrahamKnuthPatashnik})
\bq
\label{eq:f_gkp_intro}
f = 1 + zf^\mu \,.
\eq
Variations of this equation were solved, using power series, 
by Lambert \cite{Lambert_1758,Lambert_1770} and Euler \cite{Euler_1779}.
In both cases, their solutions are now known to be Fuss--Catalan series; this will be shown below.
(I shall use the term `Fuss--Catalan series' as a shorthand for 
`power series whose coefficients are Fuss--Catalan numbers.')
It is then very natural to extend eq.~\eqref{eq:f_gkp_intro}
to functions of multiple $k>1$ complex variables 
\bq
\label{eq:f_mult_intro}
f = 1 + z_1f^{\mu_1} +\cdots +z_kf^{\mu_k} \,.
\eq
Here $z_1,\dots,z_k \in\mathbb{C}$ 
and also $\mu_1,\dots,\mu_k \in\mathbb{C}$.
Analogous to eq.~\eqref{eq:f_gkp_intro},
the solution of eq.~\eqref{eq:f_mult_intro} 
is also given by a generating function, a multinomial power series in $z_1,\dots,z_k$,
where the series coefficients are `multiparameter Fuss--Catalan numbers.'

This brings us to the heart of this paper.
The multiparameter Fuss--Catalan numbers will be defined below.
However, it turns out that the literature on the multiparameter Fuss--Catalan numbers is scattered.
As can be seen from above, there are two broad threads, 
i.e.~combinatorics and the theory of several complex variables.
Different authors have published overlapping (not always equivalent) definitions, with duplication of theorems and proofs.
It is the purpose of this paper to collect together the literature on the multiparameter Fuss--Catalan numbers,
with a consistent notation and references to the various theorems and proofs by diverse authors.
In particular, consider the general algebraic equation of degree $n$,
with $x\in\mathbb{C}$ and complex coefficients $a_0,\dots,a_n$
\bq
\label{eq:alg_eq_n_intro}
a_0 +a_1 x + \cdots +a_n x^n = 0 \,.
\eq
It is known that eq.~\eqref{eq:alg_eq_n_intro}
can be solved by expressing $x$ in a multivariate series 
(more accurately, a Laurent--Puiseux series) in the coefficients $a_0,\dots,a_n$. 
This can be accomplished by an application of the Lagrange Inversion Theorem;
indeed Lagrange himself did so as a demonstration of his theorem,
for the special case of the trinomial \cite{Lagrange_works}.
Clearly, eq.~\eqref{eq:alg_eq_n_intro} can be cast in the form of
eq.~\eqref{eq:f_mult_intro}, and the solution is a (multiparameter) Fuss--Catalan series.
This highlights two broad themes in this paper:
the literature on combinatorics treats integer-valued parameters,
whereas that on complex variables treats algebraic equations (polynomials),
but both are subsumed into a general framework of Fuss--Catalan series.
Note that the exponents $\mu_1,\dots,\mu_k$ in eq.~\eqref{eq:f_mult_intro}
are arbitrary real (or in principle complex) numbers, 
and are {\em not} restricted to be integers (or rational numbers).
Some authors, such as Mohanty \cite{Mohanty_1966}, recognize this fact, but most do not.
Mohanty's work will be important below.
Contact will also be made with the works of numerous other authors such as 
Euler, Lagrange and Lambert (mentioned above),
Klein (solution of the quintic),
Gould,
Mellin and
Raney, to name a few.
Mellin \cite{Mellin} employed his eponymous transform to solve eq.~\eqref{eq:alg_eq_n_intro};
it will be shown below that his solution is a Fuss--Catalan series.
Ramanujan \cite{Ramanujan_Pt1} also published briefly on the subject,
the equation and his series solution will be cited below; it is of course a Fuss--Catalan series.
Significantly, Ramanujan derived a bound for the radius of convergence of his series (many other authors did not).
We shall derive two bounds for the absolute convergence 
of the series solution of
eq.~\eqref{eq:f_mult_intro}.
The first is necessary but in general not sufficient
and the second is sufficient but in general not necessary.
A necessary and sufficient bound for absolute convergence 
is not known at this time. 
However, for the special case of an {\em algebraic equation}, 
we shall present a new necessary and sufficient bound for the absolute convergence of the series solution
of eq.~\eqref{eq:alg_eq_n_intro}
in Sec.~\ref{sec:conv3}.
The new bound is based on earlier work by
Passare and Tsikh \cite{Passare_Tsikh}.
Some counterexamples to their results will be displayed below;
this indicates the need for a more careful treatment of the problem.

On a more personal level,
in a recent paper \cite{DM3}, 
Dilworth and this author derived the analytical solution for the 
probability mass function of the geometric distribution of order $k$ \cite{Feller}.
The roots of the associated recurrence relation were obtained as series in Fuss--Catalan numbers.
It was recognized that Fuss--Catalan series are a potentially powerful tool to solve related problems,
and in a follow-up paper \cite{DM4},
they were applied to solve additional problems for success runs of Bernoulli trials.
The title of \cite{DM4} was deliberately worded 
``Applications of Fuss--Catalan Numbers to Success Runs of Bernoulli Trials.''
This paper will not treat problems of probability and statistics,
but it is this author's personal belief that (multiparameter) Fuss--Catalan series offer great promise
to solve problems in numerous subfields of mathematics.
This motivates the desire to collect the literature on the subject in one place,
with a consistent notation and to assemble together the various duplicated theorems and proofs.

The structure of this paper is as follows. 
The basic definitions of Fuss--Catalan numbers,
their generating functions and relevant theorems are presented in
Sec.~\ref{sec:defs}.
Bounds for the absolute convergence of Fuss--Catalan series
are derived in Sec.~\ref{sec:conv1}.
The application to algebraic equations
is presented in Sec.~\ref{sec:algeq1}. 
The quintic is sufficiently important that it is placed
in a separate section in Sec.~\ref{sec:quintic}.
The trinomial equation is also sufficiently important that it is placed
in a separate section in Sec.~\ref{sec:tri}.
The domain of absolute convergence for the solutions of
algebraic equations by infinite series is discussed in 
Sec.~\ref{sec:conv2},
where it is shown that a new, more careful treatment is required,
and a new necessary and sufficient bound is presented in 
Sec.~\ref{sec:conv3}.
A sample nontrivial application of the new bound is presented
in Sec.~\ref{sec:principal_brioschi_quintic}, 
for the principal and Brioschi quintics.
Some material, including historical material, 
is relegated to \ref{app:app}.
In \ref{app:sturmfels}, 
contact is made with the work of Sturmfels \cite{Sturmfels}
on the solutions of algebraic equations via so-called
$\mathscr{A}$-hypergeometric series.

A few disclaimers and words of caution follow.
First, it is important to note that there are complex roots in many of the series,
hence branch cuts are required to obtain well-defined expressions.
Overall, this detail is not clearly (or explicitly) addressed in the literature, but it is important.
The claimed series `solution' of eq.~\eqref{eq:alg_eq_n_intro} 
may be erroneous (or meaningless) if an appropriate branch cut is not specified.
The series may converge, but not to the root of the original equation.
The subject of branch cuts will be discussed below.

Next, no claim is made here that the use of a series to solve 
eq.~\eqref{eq:f_mult_intro}
is a computationally efficient algorithm,
nor that a series solution of the algebraic equation eq.~\eqref{eq:alg_eq_n_intro}
converges rapidly to a root of the polynomial. 
McClintock {\em did} make such a claim \cite{McClintock},
but in 1895 modern digital computers and the concomitant numerical algorithms did not exist.
Indeed, a power series will {\em not} converge rapidly close to its circle of convergence.
Nevertheless, an analytical expression can indicate properties of a function not evident from a purely numerical solution.
For example,
no alternative analytical expression is known, at the present time, 
for the probability mass function of the geometric distribution of order $k$ \cite{DM3}.

Finally, this paper is {\em not} intended to be an encyclopedia.
There is a vast literature on the solution of algebraic equations by infinite series,
as well as on combinatorics using Catalan and Fuss--Catalan numbers.
Any omissions are inadvertent and not deliberate.
For example, the text by 
Appell and Kamp{\'e} de F{\'e}riet
\cite{Appell_Kampe_de_Feriet}
derives solutions of algebraic equations using
generalized hypergeometric functions.
The general sextic equation can be solved using
Kamp{\'e} de F{\'e}riet functions.
It is beyond the scope of this paper to discuss such functions.
The paper by Kamber \cite{Kamber}
contains interesting material 
on the coefficients of certain inverse power series,
but is also beyond the scope of this paper.

\setcounter{equation}{0}
\section{\label{sec:defs} Basic definitions and theorems}
For ease of reference, some of the equations displayed in the introduction will be repeated below.
The Fuss--Catalan numbers are defined, for $\mu,r\in\mathbb{C}$ and $t\in\mathbb{N}$,
as $A_0(\mu,r) := 1$ and for $t\ge1$ via
\bq
\label{eq:fc_1def}
A_t(\mu,r) := \frac{r}{t!}\prod_{j=1}^{t-1}(t\mu+r-j) \,.
\eq
As stated in the introduction,
all of the applications in this paper will treat $\mu,r\in\mathbb{R}$.
The above numbers are also known as {\em Raney numbers},
at least when $\mu$ and $r$ are nonnegative integers, 
in which case $A_t(\mu,r)$ is itself a nonnegative integer.
Raney's work \cite{Raney} will be cited below.
The generating function of the Fuss--Catalan numbers is (where $z\in\mathbb{C}$)
\bq
\label{eq:b_gkp}
B_\mu(r;z) = \sum_{t=0}^\infty A_t(\mu,r)z^t \,.
\eq
The following results are known:
\begin{theorem}
\label{thm:thm_1}
\begin{enumerate}[(a)]
\item
The generating function $B_\mu(1;z)$ satisfies the following equation for $f(z)$ %
\bq
\label{eq:f_gkp}
f = 1 + zf^\mu \,.
\eq
\item
The generating function $B_\mu(r;z)$ also has the property 
\bq
\label{eq:pow_1}
B_\mu(1;z)^r = B_\mu(r;z) \,.
\eq
Let $s\in\mathbb{C}$ and 
using $B_\mu(1;z)^{r+s} = B_\mu(1;z)^r B_\mu(1;z)^s$, eq.~\eqref{eq:pow_1} is equivalent to the statement
\bq
\label{eq:pow_rs1}
B_\mu(r+s;z) = B_\mu(r;z) B_\mu(s;z) \,.
\eq
\item
The Fuss--Catalan numbers satisfy the following convolution identity 
\bq
\label{eq:fc_1conv}
A_t(\mu, s+r) = \sum_{u=0}^t A_u(\mu,r) A_{t-u}(\mu,s) \,.
\eq
\item
The Fuss--Catalan numbers satisfy the recurrence relation
(there are other equivalent ways to express the recurrence)
\bq
\label{eq:fc_1rec}
A_t(\mu, r+1) = A_t(\mu, r) + A_{t-1}(\mu, r+\mu) \,.
\eq
\end{enumerate}
\end{theorem}
Note that Theorems \ref{thm:thm_1}(b) and (c) are equivalent.
Write out eq.~\eqref{eq:pow_rs1} in full, then
\bq
\sum_{t=0}^\infty A_t(\mu,r+s)z^t = 
\biggl(\sum_{t^\prime=0}^\infty A_{t^\prime}(\mu,r)z^{t^\prime} \biggr)
\biggl(\sum_{t^\pprime=0}^\infty A_{t^\pprime}(\mu,s)z^{t^\pprime} \biggr) \,.
\eq
Selecting a particular value of $t$ on the left hand side and equating terms,
we obtain a sum of terms $t^\prime+t^\pprime = t$ on the right-hand side and eq.~\eqref{eq:fc_1conv} follows.
Reversing the steps proves the converse.
Also, using eqs.~\eqref{eq:b_gkp} and \eqref{eq:pow_1}, eq.~\eqref{eq:fc_1rec} can easily be employed to show that 
\bq
\label{eq:f_gkp_r}
B_\mu(1;z)^{r+1} = B_\mu(1;z)^r + zB_\mu(1;z)^{r+\mu} \,.
\eq
This is simply eq.~\eqref{eq:f_gkp} multiplied through by $B_\mu(1;z)^r$.

To generalize to $k\ge1$ multiple parameters,
we employ a vector notation and 
introduce the $k$-tuples
$\bm{t}=(t_1,\dots,t_k)\in\mathbb{N}^k$,
$\bm{\mu}=(\mu_1,\dots,\mu_k)\in\mathbb{C}^k$ and
$\bm{z}=(z_1,\dots,z_k)\in\mathbb{C}^k$
(and recall $r,s\in\mathbb{C}$).
Also define $|\bm{t}| = |t_1| +\cdots +|t_k|$.
For brevity we shall frequently write $t = |\bm{t}|$ below.
Also, for $|\bm{t}|>0$, define the `unit vector'
$\hat{\bm{t}} = \bm{t}/|\bm{t}|$.
We also define the zero vector $\bm{0} = (0,\dots,0)$ and the `basis vectors'
$\bm{e}_j = (0,\dots,0,1,0,\dots,0) = (\delta_{1j},\dots,\delta_{kj})$.
\begin{definition}[multiparameter Fuss--Catalan numbers]
We define the multiparameter Fuss--Catalan numbers $\mathscr{A}_{\bm{t}}(\bm{\mu},r)$ via
$\mathscr{A}_{\bm{0}}(\bm{\mu},r) := 1$ for $\bm{t}=\bm{0}$ and for $|\bm{t}|>0$ via
\bq
\label{eq:fc_kdef1}
\mathscr{A}_{\bm{t}}(\bm{\mu},r) := \frac{r}{t_1!\cdots t_k!}\prod_{j=1}^{|\bm{t}|-1}(\bm{t}\cdot\bm{\mu}+r-j) \,.
\eq
If $k=1$ this reduces to the single-parameter definition eq.~\eqref{eq:fc_1def}.
Equivalently, for all $|\bm{t}|\ge0$,
\bq
\label{eq:fc_kdef2}
\mathscr{A}_{\bm{t}}(\bm{\mu},r) = \binom{t}{t_1,\dots,t_k}\,A_t(\hat{\bm{t}}\cdot\bm{\mu},r) \,.
\eq
Note that a $0/0$ indeterminate expression for $\hat{\bm{t}}$ does not arise in eq.~\eqref{eq:fc_kdef2}
because of the definition $A_0(\cdot) :=1$.
\end{definition}
\begin{definition}[multiparameter generating function]
The multiparameter Fuss--Catalan generating function is defined as
\bq
\label{eq:fc_kgen1}
\mathcal{B}(\bm{\mu};r;\bm{z}) := \sum_{\bm{t}\in\mathbb{N}^k} \mathscr{A}_{\bm{t}}(\bm{\mu},r)\,z_1^{t_1}\cdots z_k^{t_k} \,.
\eq
Technically, the above expression is not well-defined because the answer can depend on the order of summation.
In all the applications in this paper, we collect the terms in level sets in $t=|\bm{t}|$
\bq
\label{eq:fc_kgen2}
\mathcal{B}(\bm{\mu};r;\bm{z}) = \sum_{t=0}^\infty \sum_{t_1+\cdots+t_k=t} \binom{t}{t_1,\dots,t_k} A_t(\hat{\bm{t}}\cdot\bm{\mu},r)
\,z_1^{t_1}\cdots z_k^{t_k} \,.
\eq
However, to obtain rigorous results, we must specify a domain of {\em absolute convergence}.
Then the answer will not depend on the order of summation.
The topic of absolute convergence will be discussed below.
\end{definition}
\begin{theorem}
\label{thm:thm_k}
\begin{enumerate}[(a)]
\item
The generating function $\mathcal{B}(\bm{\mu};1;\bm{z})$ satisfies the following equation for $f(\bm{z})$ 
\bq
\label{eq:f_mult}
f = 1 + z_1f^{\mu_1} +\cdots +z_kf^{\mu_k} \,.
\eq
\item
Analogous to eq.~\eqref{eq:pow_1}, the generating function $\mathcal{B}(\bm{\mu};r;\bm{z})$ has the property 
\bq
\label{eq:pow_k}
\mathcal{B}(\bm{\mu};1;\bm{z})^r = \mathcal{B}(\bm{\mu};r;\bm{z}) \,.
\eq
Analogous to eq.~\eqref{eq:pow_rs1}, it follows that
\bq
\label{eq:pow_rsk}
\mathcal{B}(\bm{\mu};r+s;\bm{z}) = 
\mathcal{B}(\bm{\mu};r;\bm{z}) 
\mathcal{B}(\bm{\mu};s;\bm{z}) \,.
\eq
\item
The multiparameter convolution identity analogous to eq.~\eqref{eq:fc_1conv} is
(the allowed values of $\bm{u}$ are obvious)
\bq
\label{eq:fc_kconv}
\mathscr{A}_{\bm{t}}(\bm{\mu},r+s) = 
\sum_{\bm{u}\in\mathbb{N}^k}
\mathscr{A}_{\bm{u}}(\bm{\mu},r) \mathscr{A}_{\bm{t}-\bm{u}}(\bm{\mu},s) \,.
\eq
\item
Analogous to eq.~\eqref{eq:fc_1rec}, the multiparameter recurrence is
(again, there are other equivalent ways to express the recurrence)
\bq
\label{eq:fc_krec}
\mathscr{A}_{\bm{t}}(\bm{\mu}, r+1) = \mathscr{A}_{\bm{t}}(\bm{\mu}, r) 
+\sum_{j=1}^k \mathscr{A}_{\bm{t}-\bm{e}_j}(\bm{\mu}, r+\mu_j) \,.
\eq
\end{enumerate}
\end{theorem}
Theorems \ref{thm:thm_k}(b) and (c) are equivalent; the proof follows the same steps as for the case $k=1$.
Also, using eqs.~\eqref{eq:fc_kgen1} (or \eqref{eq:fc_kgen2}) and \eqref{eq:pow_k}, eq.~\eqref{eq:fc_krec} yields
\bq
\label{eq:f_mult_r}
\mathcal{B}(\bm{\mu};1;\bm{z})^{r+1} = 
\mathcal{B}(\bm{\mu};1;\bm{z})^r 
+\sum_{j=1}^k z_j \mathcal{B}(\bm{\mu};1;\bm{z})^{r+\mu_j} \,. 
\eq
This is eq.~\eqref{eq:f_mult} multiplied through by $\mathcal{B}(\bm{\mu};1;\bm{z})^r$.
A search of the literature revealed that all of the results in Theorem \ref{thm:thm_k} have already been proved.
Unfortunately, the proofs are scattered (and rediscovered) in the literature.
Unlike the single-parameter case, where the relations
are explicitly stated as properties of Fuss--Catalan numbers
(see Theorem \ref{thm:thm_1}),
for $k>1$ there is a variety of notations and not all authors mention Fuss and Catalan.
(This should {\em not} be misinterpreted as a criticism; see the comment at the beginning of \ref{app:app}.)
For the multiparameter case, the most comprehensive references I have found 
were by Raney \cite{Raney}, Chu \cite{Chu} and Mohanty \cite{Mohanty_1966}.
I summarize their works in turn.
Raney \cite{Raney} presented proofs of all the results in Theorem \ref{thm:thm_k}.
Raney's expression is as follows.
Let $a_1,a_2,\dots$ be an infinite sequence of natural numbers of which at most a finite number of terms are different from zero.
Then define
$m = n + \sum_{i=1}^\infty ia_i$
and
$a_0 = n + \sum_{i=1}^\infty (i-1)a_i$.
Raney defined the multinomial coefficient 
\bq
M(a_0,a_1,a_2,\dots) = \frac{(a_0+a_1+a_2+\cdots)!}{a_0!a_1!a_2!\cdots} \,.
\eq
Then \cite[Theorem 2.2]{Raney} states
\bq
mL(n;a_1,a_2,\dots) = nM(a_0,a_1,a_2,\dots) \,.
\eq
Let us reexpress this in our notation.
We know only finitely many of the $a_i$ are nonzero.
Suppose there are $k$ nonzero $a_i$ are they are indexed by the set $(\mu_1,\dots,\mu_k)$.
Also define $t_j = a_{\mu_j}$ and replace $n$ by $r$, then
$m = r + \bm{t}\cdot\bm{\mu}$ and $a_0 = r + \bm{t}\cdot\bm{\mu} - |\bm{t}| = m - |\bm{t}|$.
Then
\bq
\begin{split}
L(n;a_1,a_2,\dots) 
&= \frac{n}{m}\,\frac{(a_0+a_1+a_2+\cdots)!}{a_0!a_1!a_2!\cdots}
\\
&= \frac{r}{r + \bm{t}\cdot\bm{\mu}}\,
\frac{(r + \bm{t}\cdot\bm{\mu})!}
{(r + \bm{t}\cdot\bm{\mu} - |\bm{t}|)! t_1!\cdots t_k!}
\\
&= \frac{r}{t_1!\cdots t_k!}\,
\prod_{j=1}^{|\bm{t}|-1} (\bm{t}\cdot\bm{\mu}+r-j) 
\\
&= \mathscr{A}_{\bm{t}}(\bm{\mu},r) \,.
\end{split}
\eq
Notice that Raney's expression for $L$ is a {\em solution}, not a {\em definition}.
Raney posed and solved many combinatorial problems in \cite{Raney}.
Then \cite[Theorems 2.3, 2,4, 4.1]{Raney} 
yield respectively eqs.~\eqref{eq:fc_kconv}, \eqref{eq:fc_krec} and \eqref{eq:pow_k}.
Also \cite[eqs.~(6.1) and (6.2)]{Raney} yield eq.~\eqref{eq:f_mult}.
Note that Raney \cite{Raney} took the $\mu_j$ (in my notation) to be integers;
this is common also in the derivations by other authors (see below).
However, it is straightforward to generalize from integer to complex-valued parameters.
The relevant steps are given by Graham et al.~\cite{GrahamKnuthPatashnik}
(for the single-parameter case $k=1$, but the same reasoning works also for multiple parameters $k>1$).
Chu \cite{Chu} also published a proof of the solution of eq.~\eqref{eq:f_mult} (citing Raney \cite{Raney}).
Chu remarked that eq.~\eqref{eq:f_mult} can also be derived using 
the multi-variable version of the Lagrange inversion formula \cite{Good}.
(Numerous authors have stated that eq.~\eqref{eq:f_mult} can be derived using Lagrange inversion.
Raney gave an example of Lagrange inversion in \cite{Raney}.)
Chu treated only integer-valued parameters.
Chu defined `higher Catalan numbers' and `generalized Catalan numbers' as follows.
Chu employed vectors $\vec{v}$ and $\bm{n}$, which are $k$-tuples of integers.
The `higher Catalan numbers' are \cite[eq.~(1)]{Chu}
\bq
\label{eq:chu_1}
C_k(n) = \frac{1}{nk+1}\binom{nk+1}{n} \,.
\eq
This is equivalent to $A_n(k,1)$ in eq.~\eqref{eq:fc_1def}
The `generalized Catalan numbers' are \cite[eq.~(2)]{Chu}
\bq
\label{eq:chu_2}
C_{\vec{v}}(\vec{n}) = \frac{1}{\sum_{i=1}^k n_iv_i+1}\binom{\sum_{i=1}^k n_iv_i+1}{n_1,n_2,\dots,n_k,1+\sum_{i=1}^k n_i(v_i-1)} \,.
\eq
This is equivalent to $\mathscr{A}_{\bm{n}}(\bm{v},1)$ in eq.~\eqref{eq:fc_kdef1}.
Beware of the slightly inconsistent use of $k$ by Chu,
as quoted in eqs.~\eqref{eq:chu_1} and \eqref{eq:chu_2}.
Curiously, Chu \cite{Chu} restricted his definitions only to $r=1$,
even though unnamed expressions with $r>1$ appear in his paper (Chu wrote $t$ for what I call $r$).
Mohanty \cite{Mohanty_1966} explicitly treated complex-valued parameters.
He defined the multinomial coefficient as follows \cite[eq.~(3)]{Mohanty_1966}
\bq
\binom{x}{j_1,\dots,j_k} = \frac{x(x-1)\cdots(x-\sum_{j=1}^k j +1)}{\prod_{i=1}^k j_i!} \,.
\eq
Then Mohanty defined (without assigning a name) \cite[eq.~(4)]{Mohanty_1966}
\bq
\label{eq:fc_mohanty}
A(a;b_1,\dots,b_k;n_1,\dots,n_k) = \frac{a}{a+\sum_{i=1}^k b_in_i}\binom{a+\sum_{i=1}^k b_in_i}{n_1,\dots,n_k} \,.
\eq
Here $a,b_1,\dots,b_k \in\mathbb{C}$ are all complex.
This is equivalent to $\mathscr{A}_{\bm{n}}(\bm{b},a)$ in eq.~\eqref{eq:fc_kdef1}.
Mohanty proved several multiparameter convolution identities in \cite{Mohanty_1966},
in particular eq.~\eqref{eq:fc_kconv}.
(Additional results are given in \ref{app:app} below.)
Mohanty defined a generating function \cite[unnumbered before eq.~(13)]{Mohanty_1966} 
and proved that it satisfies the following equation for $z$ \cite[eq.~(22)]{Mohanty_1966} 
\bq
\label{eq:mohanty_eq22}
sz^b +tz^d -z +1 = 0 \,.
\eq
Here all of $b$, $d$, $s$ and $t$ are complex.
Note that Mohanty displayed explicit derivations for the case $k=2$ 
and pointed out that the extension to more parameters merely requires additional bookkeeping,
hence eq.~\eqref{eq:mohanty_eq22} generalizes to $k>2$ parameters and is effectively eq.~\eqref{eq:f_mult}.
Similarly \cite[eq.~(25)]{Mohanty_1966} yields the identity eq.~\eqref{eq:pow_rsk}.
Strehl \cite{Strehl} also gave a proof of the solution of eq.~\eqref{eq:f_mult},
where \cite[eq.(21)]{Strehl} is an algebraic equation with complex coefficients.
In fact \cite[eq.(21)]{Strehl} is the equation solved by Mellin \cite{Mellin}
and displayed in eq.~\eqref{eq:poly_mellin_nach} below.
Strehl also provides some historical background, citing both Chu \cite{Chu} and Raney \cite{Raney}.
More recently, eq.~\eqref{eq:f_mult} was solved by Schuetz and Whieldon \cite[Theorem 4.2]{Schuetz_Whieldon},
who treated integer valued coefficients and exponents only.
The series coefficients were identified as Fuss--Catalan numbers,
multiplied by multinomial coefficients (see eq.~\eqref{eq:fc_kdef2});
this is the only reference I have found to mention Fuss--Catalan explicitly for the multiparameter case
(but see Chu \cite{Chu} above).
Banderier and Drmota \cite{Banderier_Drmota} 
also derived a series solution for an algebraic equation
where \cite[Theorem 3.3]{Banderier_Drmota} is termed the 
`Flajolet-Soria formula for coefficients of an algebraic function.'
See \cite[eq.~(3.3)]{Banderier_Drmota}.

\begin{remark}
The exponents $\mu_j$ in eq.~\eqref{eq:f_mult} need not be distinct, 
although from a practical viewpoint it may be pointless if they are not.
Consider the extreme case where they are all equal $\mu_1=\cdots=\mu_k=\mu$ so $\hat{\bm{t}}\cdot\bm{\mu}=\mu$. 
Then eq.~\eqref{eq:f_mult} simplifies to
\bq
\label{eq:f_all_equal}
f = 1 + (z_1+\cdots+z_k)f^\mu \,.
\eq
This is simply eq.~\eqref{eq:f_gkp} with $z = \sum_{j=1}^k z_j$.
Then in eq.~\eqref{eq:f_mult}, $A_t(\mu,1)$ does not depend on the individual $t_j$ so
\bq
\begin{split}
f &= \sum_{t=0}^\infty A_t(\mu,1) \biggl[ \sum_{t_1+\cdots+t_k=t} \binom{t}{t_1,\dots,t_k}  z_1^{t_1}\cdots z_k^{t_k} \biggr]
\\
&= \sum_{t=0}^\infty A_t(\mu,1) (z_1+\cdots+z_k)^t \,.
\end{split}
\eq
This is precisely the Fuss--Catalan series which is the known solution of eq.~\eqref{eq:f_all_equal}.
\end{remark}

\begin{remark}[branch cuts]
If some of the $\mu_j$ in eq.~\eqref{eq:f_mult} are nonintegers, a branch cut is required in the complex plane.
The classic example is the square root $\mu_j=\frac12$ and $f^{1/2}$.
A specific sheet of the complex plane must be selected, to render equations such as eq.~\eqref{eq:f_mult} well defined
(although, as pointed out above, the {\em domain} of absolute convergence will not depend on branch cuts).
In all of the numerical work reported in this paper, the branch cut was placed along the positive real axis,
so $0 \le \arg(z_j) < 2\pi$ for $j=1,\dots,k$ and similarly for $f$ and all other complex variables to appear below.
This is necessary to obtain meaningful sums for the various series in this paper.
Mellin \cite{Mellin} placed the branch cut along the negative real axis.
The essential fact is that a branch cut is required; one must make a specific choice and adhere to it consistently.
\end{remark}

Up now now, we have treated only powers of $r$, i.e.~$\mathcal{B}(\bm{\mu};r,\bm{z})=\mathcal{B}(\bm{\mu};1,\bm{z})^r$.
Let us briefly treat symmetries involving the tuple $\bm{\mu}$.
\begin{theorem}
\label{thm:thm_mu}
Fix $k\ge1$ and let $\bm{\mu}=(\mu_1,\dots,\mu_k)$ and $\bm{z}=(z_1,\dots,z_k)$.
Define the tuple $\bm{1} = (1,\dots,1)$.
The sum of two tuples and the multiplication of a tuple by a scalar are defined in the obvious way.
Hence $\bm{1}-\bm{\mu} = (1-\mu_1,\dots,1-\mu_k)$ and $-\bm{z}=(-z_1,\dots,-z_k)$.
Then the generating function $\mathcal{B}(\bm{\mu};r;\bm{z})$ satisfies the relation
\bq
\label{eq:b_symm_mu}
\mathcal{B}(\bm{\mu}; r; \bm{z}) = \mathcal{B}(\bm{1}-\bm{\mu}; -r; -\bm{z}) \,.
\eq
\end{theorem}
\begin{proof}
The proof consists of fixing a tuple $\bm{t}$ and evaluating the finite products in 
$\mathscr{A}_{\bm{t}}(\bm{\mu},r)$ and $\mathscr{A}_{\bm{t}}(\bm{1}-\bm{\mu},-r)$ and keeping track of the minus signs.
\end{proof}
\begin{remark}\label{rem:b_symm_mu_scalar}
In the scalar case, eq.~\eqref{eq:b_symm_mu} simplifies to
\bq
\label{eq:b_symm_mu_scalar}
\mathcal{B}(\mu;r;z) = \mathcal{B}(1-\mu;-r;-z) \,.
\eq
Combined with the result $\mathcal{B}(\mu;r;z) = \mathcal{B}(\mu;1;z)^r$,
this means that in the scalar case, we need only compute $\mathcal{B}(\mu;1;z)$ for $\mu\ge0$.
\end{remark}

\setcounter{equation}{0}
\section{\label{sec:conv1} Domain of convergence}
In general, the convergence of an infinite series depends on the order of summation.
In this paper, we take `convergence' to mean exclusively {\em absolute convergence}.
In that case, the answer does not depend on the order of summation.
In general, the series in eq.~\eqref{eq:fc_kgen2} has a finite domain of absolute convergence.
We present two sets of conditions for the series in eq.~\eqref{eq:fc_kgen2} to converge absolutely.
The first is necessary, but in general not sufficient,
and the second is sufficient, but in general not necessary.
A more detailed analysis for the special case of algebraic equations will be presented in 
Secs.~\ref{sec:conv2} and \ref{sec:conv3}.
The derivations below assume the $\mu_j$ are real and are ordered $\mu_1 \le \mu_2 \le \cdots \le \mu_k$.
We begin with the following lemma for the asymptotic value of the Fuss--Catalan numbers.
\begin{lemma}
Asymptotically for $t\gg1$ and real $\mu,r$, 
\bq
\label{eq:fc_asymp}
A_t(\mu,r) \sim \frac{r}{\sqrt{2\pi}\,t^{3/2}}\, 
\frac{|\mu|^{r-\frac12}}{|1-\mu|^{r+\frac12}} \,
(|\mu|^\mu |1-\mu|^{1-\mu})^t \,.
\eq
The above is an application of Stirling's formula and the proof is omitted.
We require $\mu\ne0$ and $\mu \ne 1$ to justify the intermediate steps in the derivation.
To determine the radius of convergence using d'Alembert's ratio test, note that asymptotically
\bq
\label{eq:fc_ratio_test}
\frac{A_t(\mu,r)}{A_{t-1}(\mu,r)} \sim |\mu|^\mu |1-\mu|^{1-\mu} \,.
\eq
\end{lemma}

\begin{proposition}[necessary, not sufficient]
\label{prop:nec1}
For the series in eq.~\eqref{eq:fc_kgen2} to converge absolutely, it is necessary that
\bq
\label{eq:conv_nec_j}
|z_j| \le |z_j|_{\max} \equiv \frac{1}{|\mu_j|^{\mu_j} |1-\mu_j|^{1-\mu_j}} \qquad (j=1,\dots,k) \,.
\eq
Then all points of the following form lie in the domain of convergence
\bq
\label{eq:conv_nec_j_vertex}
\tilde{\bm{z}}_j = (0,\dots,0,|z_j| = |z_j|_{\max}, 0,\dots,0) \qquad (j=1,\dots,k) \,.
\eq
\end{proposition}
\begin{proof}
Fix a value of $j$ and set all the other $z_{j^\prime}$ to zero, where $j^\prime \ne j$.
Then the series in eq.~\eqref{eq:fc_kgen2} reduces to a sum in powers of single variable $z_j$.
Then eq.~\eqref{eq:conv_nec_j} follows from eq.~\eqref{eq:fc_ratio_test}
and d'Alembert's ratio test.
From the asymptotic form of the Fuss--Catalan numbers in eq.~\eqref{eq:fc_asymp},
the series converges also on its circle of convergence, justifying the `$\le$' in 
eq.~\eqref{eq:conv_nec_j}.
Then eq.~\eqref{eq:conv_nec_j_vertex} follows immediately.
\end{proof}

\begin{proposition}[sufficient, not necessary]
\label{prop:suff1}
The series in eq.~\eqref{eq:fc_kgen2} converges absolutely if 
\bq
\label{eq:radconv}
\sum_{j=1}^k |z_j| \le \min\biggl( \frac{1}{|\mu_1|^{\mu_1} |1-\mu_1|^{1-\mu_1}} \,,
\frac{1}{|\mu_k|^{\mu_k} |1-\mu_k|^{1-\mu_k}} \biggr) \,.
\eq
The above condition is sufficient, but in general not necessary.
\end{proposition}

\begin{proof}
We employ eq.~\eqref{eq:fc_kgen2} and eq.~\eqref{eq:f_mult}.
Let us define $\alpha = \sum_{j=1}^k |z_j|$ and $p_j = |z_j|/\alpha$, for $j=1,\dots,k$.
Then $0 \le p_j \le 1$ and $\sum_{j=1}^k p_j = 1$.
Then from eq.~\eqref{eq:f_mult}
\bq
|f| \le \sum_{t=0}^\infty \alpha^t \biggl\{ \sum_{t_1+\cdots+t_k=t} 
|A_t(\hat{\bm{t}}\cdot\bm{\mu},1)| \binom{t}{t_1,\dots,t_k} p_1^{t_1}\cdots p_k^{t_k} \biggr\} \,.
\eq
Let us suppose that $|A_t(\hat{\bm{t}}\cdot\bm{\mu},1)|$ is majorized by setting $\hat{\bm{t}}\cdot\bm{\mu}=\mu_*$,
where $\mu_*$ does not depend on the $t_j$.
(This will be discussed in more detail below.)
Actually, to establish convergence of the series, 
it is sufficient if $|A_t(\hat{\bm{t}}\cdot\bm{\mu},1)| < |A_t(\mu_*,1)|$ only asymptotically, say for $t \ge T$.
Then
\bq
\begin{split}
|f| &\le \textrm{const}
+\sum_{t=T}^\infty |A_t(\mu_*,1)| \,\alpha^t
\biggl\{ \sum_{t_1+\cdots+t_k=t} \binom{t}{t_1,\dots,t_k} p_1^{t_1}\dots p_k^{t_k} \biggr\}
\\
&= \textrm{const} + \sum_{t=T}^\infty |A_t(\mu_*,1)| \,\alpha^t \,.
\end{split}
\eq
Using eq.~\eqref{eq:fc_ratio_test} and d'Alembert's ratio test,
we obtain the following sufficient (but not always necessary) condition for convergence:
\bq
\label{eq:radconv_mu*}
\sum_{j=1}^k |z_j| \le \frac{1}{|\mu_*|^{\mu_*} |1-\mu_*|^{1-\mu_*}} \,.
\eq
The essential step to complete the proof is to specify the value of $\mu_*$.
Since the $\mu_j$ are ordered, $\mu_1 \le \hat{\bm{t}}\cdot\bm{\mu} \le \mu_k$.
Now the graph of $|\mu_*|^{\mu_*} |1-\mu_*|^{1-\mu_*}$ 
attains a minimum at $\mu_*=\frac12$ (and is symmetric around $\mu_*=\frac12$) 
and increases monotonically in either direction away from the minimum.
Hence the value of $|\mu_*|^{\mu_*} |1-\mu_*|^{1-\mu_*}$
is maximized by setting $\mu_* = \mu_1$ or $\mu_*=\mu_k$.
Either value will do if they are equidistant from $\frac12$.
This proves eq.~\eqref{eq:radconv}.
Admittedly, this may not be an {\em optimal} criterion: it is sufficient, but may not be necessary.
Numerical tests indicate that the domain of convergence using the above value of $\mu_*$ can be very conservative.
\end{proof}

\begin{corollary}[trinomial]
\label{eq:cor_k1}
For the special case $k=1$, there is only one summand, and so $\mu_*=\mu_1=\mu$ and we may write $z_1=z$.
Then the criterion for absolute convergence is necessary and sufficient.
The series in eq.~\eqref{eq:b_gkp} converges if and only if 
\bq
\label{eq:radconv1}
|z| \le \frac{1}{|\mu|^\mu |1-\mu|^{1-\mu}} \,.
\eq
The proof is immediate from taking Propositions \ref{prop:nec1} and \ref{prop:suff1} together.
From Proposition \ref{prop:nec1}, the series converges everywhere on its circle of convergence.
This corollary will be important below.
\end{corollary}

\begin{remark}
It is clear that the conditions in eq.~\eqref{eq:conv_nec_j} are individually necessary,
but, even taken together, they are not sufficient to guarantee absolute convergence of the full sum in
eq.~\eqref{eq:fc_kgen2}.
Hence an upper bound for the measure of the domain of absolute convergence,
for $(|z_1|,\dots,|z_k|) \in \mathbb{R}^k$, is given by the finite product
\bq
\label{eq:measure_upper}
\mu \le \prod_{j=1}^k \frac{1}{|\mu_j|^{\mu_j} |1-\mu_j|^{1-\mu_j}} < \infty \,.
\eq
The use of $\mu$ on the left hand side to denote measure should not be confused with other uses of $\mu$ in this paper.
The true domain of absolute convergence is a set of smaller measure.
This justifies the claim at the beginning of this section that 
the series in eq.~\eqref{eq:fc_kgen2} has a `finite domain' of absolute convergence, 
i.e.~finite measure.
Similarly, using the sufficient condition in eq.~\eqref{eq:radconv}
and $\mu_*$ from eq.~\eqref{eq:radconv_mu*}, a lower bound for the measure of the domain of absolute convergence is
\bq
\label{eq:measure_lower}
\mu \ge \frac{1}{k!} \biggl(\frac{1}{|\mu_*|^{\mu_*} |1-\mu_*|^{1-\mu_*}}\biggr)^k > 0 \,. 
\eq
The measure is positive: for sufficiently small $|z_j|$, $j=1,\dots,k$,
the series in eq.~\eqref{eq:fc_kgen2} converges in an open neighborhood of the origin for $\bm{z}\in\mathbb{C}^k$.
Of course this latter fact could be deduced directly using eq.~\eqref{eq:f_mult},
but eq.~\eqref{eq:measure_lower} supplies a quantitative lower bound.
For the special case of a trinomial, where $k=1$,
the two bounds in eqs.~\eqref{eq:measure_upper} and \eqref{eq:measure_lower} coincide.
\end{remark}

\begin{remark}
A complete derivation of a necessary and sufficient condition 
for the absolute convergence of the series in eq.~\eqref{eq:fc_kgen2} 
has not yet been discovered.
However, the situation is different for an {\em algebraic equation}.
As stated in the introduction,
the necessary and sufficient bound for the convergence of the series solution
of eq.~\eqref{eq:alg_eq_n_intro} will be presented in Sec.~\ref{sec:conv3}.
\end{remark}

\setcounter{equation}{0}
\section{\label{sec:algeq1} Algebraic equations}
\subsection{Preliminary remarks}
We now treat some applications of the above formalism.
In this paper we shall treat algebraic equations, i.e.~series solutions for roots of polynomials. 
Consider the general algebraic equation of degree $n$ with $x\in\mathbb{C}$ 
and complex coefficients $a_0,\dots,a_n$
\bq
\label{eq:algeq1}
0 = a_0 +a_1x + \cdots + a_nx^n \,.
\eq
We require $a_0\ne0$ else we factor out a root $x=0$.
We also require $a_n\ne0$. 
We begin with an obvious, but necessary, caveat.
It is possible that some or all of the remaining $a_j$ could vanish.
To avoid cluttering the presentation, it is to be understood that in all of the multinomial sums below,
the sums extend only over the nonzero $a_j$.
We now note two elementary transformations of eq.~\eqref{eq:algeq1},
which do not affect the fundamental properties of the roots.
First, we can multiply all the coefficients by a constant $\lambda \ne 0$.
This does not change the roots of eq.~\eqref{eq:algeq1}.
Next, we can replace $x$ by $\mu y$, where $\mu\ne0$.
The roots for $x$ are simply those for $y$, multiplied by $\mu$.
The resulting equation is
$\sum_{j=0}^n a_j\lambda\mu^j y^j = 0$.
Define $b_j = a_j\lambda\mu^j$.
We can select two integers $p$ and $q$ such that $0 \le p < q \le n$
and find values for $\lambda$ and $\mu$ such that $b_p=b_q=1$, yielding
\bq
\label{eq:algeq2}
0 = b_0 +b_1y + \cdots +y^p +\cdots +y^q +\cdots + b_ny^n \,.
\eq
Clearly both $a_p$ and $a_q$ must be nonzero to do this.
It is easily derived that $\mu = (a_p/a_q)^{1/(q-p)}$ and 
\bq
b_j = \frac{a_j}{a_p^{(q-j)/(q-p)} a_q^{(j-p)/(q-p)}} \,.
\eq
Technically, $b_j$ depends on $p$ and $q$ also, but we consider this to be understood below.
Clearly a branch cut is required to derive the above expressions.
There are actually $q-p$ solutions for $\mu$ and $b_j$, indexed by the $q-p$ roots of unity $1^{1/(q-p)}$
(actually $(-1)^{1/(q-p)}$, we shall see this below).
For brevity we define the set $\mathscr{N}_{npq} = \{0,\dots,n\} \setminus \{p,q\}$.
We divide eq.~\eqref{eq:algeq2} through by $y^p$ and rearrange terms to obtain the equation 
\bq
\label{eq:algeq3}
-y^{q-p} = 1 + \sum_{j\in \mathscr{N}_{npq}} b_j y^{j-p} \,.
\eq
Note that if we set all the $b_j$ to zero, the equation reduces to $y^{q-p} = -1$.
The solution is any of the radicals $y = (-1)^{1/(q-p)}$.
By the implicit function theorem, an absolutely convergent solution for $y$ exists
for sufficienly small amplitudes of the $|b_j|$.
Hence the domain of absolute convergence of the series solution of eq.~\eqref{eq:algeq3} is nonempty,
expressing $y$ in a power series in the $b_j$.
(We already know this from Sec.~\ref{sec:conv1}.)
Now set $\zeta = -y^{q-p}$, so $y = (-1)^{1/(q-p)}\zeta^{1/(q-p)}$.
We append a subscript $\ell$ on $x$, $y$ and $\zeta$ to index the $q-p$ choices of radicals $(-1)^{1/(q-p)}$.
Employing a branch cut along the positive real axis, they are
$e^{i\pi(2\ell+1)/(q-p)}$, where $\ell=0,\dots,q-p-1$.
Set $\mu_j = (j-p)/(q-p)$, then $b_j = a_j/(a_p^{1-\mu_j}a_q^{\mu_j})$ and $\zeta_\ell$ satisfies
\bq
\label{eq:poly_root1}
\zeta_\ell = 1 + \sum_{j\in\mathscr{N}_{npq}} e^{i\pi(2\ell+1)\mu_j} \frac{a_j}{a_p^{1-\mu_j}a_q^{\mu_j}}\,\zeta_\ell^{\mu_j} \,.
\eq
This has the form of eq.~\eqref{eq:f_mult}, with $k=\textrm{Card}(\mathscr{N}_{npq})$ parameters (the $b_j$).
The expressions for $t$ and $\bm{t}\cdot\bm{\mu}$ are, in this case,
\bq
t = \sum_{j\in\mathscr{N}_{npq}} t_j \,,\qquad
\bm{t}\cdot\bm{\mu} = \sum_{j\in\mathscr{N}_{npq}} t_j\mu_j \,.
\eq
The solution for $\zeta_\ell$ is given by eq.~\eqref{eq:fc_kgen2}
and $x_\ell$ is obtained from $\zeta_\ell$ via
\bq
\label{eq:x_zeta_pq}
x_\ell = e^{i\pi\frac{2\ell+1}{q-p}} \Bigl(\frac{a_p}{a_q}\Bigr)^{1/(q-p)} \zeta_\ell^{1/(q-p)} \,.
\eq
It is conventional to solve for the $r^{th}$ powers of the roots.
From Theorem \ref{thm:thm_k}(a) and (b)
and eq.~\eqref{eq:x_zeta_pq}, we obtain
\bq
\label{eq:pivot_pq}
\begin{split}
x_\ell^r &= e^{i\pi\frac{(2\ell+1)r}{q-p}} \Bigl(\frac{a_p}{a_q}\Bigr)^{r/(q-p)} \,
\sum_{\bm{t}\in\mathbb{N}^{n-1}} \mathscr{A}_{\bm{t}}\Bigl(\bm{\mu}, \frac{r}{q-p}\Bigr)\, 
e^{i\pi (2\ell+1)\bm{t}\cdot\bm{\mu}} \Bigl(\prod_{j\in\mathscr{N}_{npq}} b_j^{t_j}\Bigr) 
\\
&= e^{i\pi\frac{(2\ell+1)r}{q-p}} \Bigl(\frac{a_p}{a_q}\Bigr)^{1/(q-p)} \,
\sum_{\bm{t}\in\mathbb{N}^{n-1}} \mathscr{A}_{\bm{t}}\Bigl(\bm{\mu}, \frac{r}{q-p}\Bigr)\, 
\frac{e^{i\pi (2\ell+1)\bm{t}\cdot\bm{\mu}}}{a_p^{t-\bm{t}\cdot\bm{\mu}} a_q^{\bm{t}\cdot\bm{\mu}}} 
\Bigl(\prod_{j\in\mathscr{N}_{npq}} a_j^{t_j}\Bigr) \,.
\end{split}
\eq
In the first line of eq.~\eqref{eq:pivot_pq},
$x_\ell^r$ is a sum over products of positive integral powers of the $b_j$, i.e.~a power series.
In the second line, $a_p$ and $a_q$ appear with fractional (and possibly negative) exponents, 
whereas the other $a_j$ appear with positive integral powers.
Hence in general the series solution for the roots 
is a Laurent--Puiseux series in the coefficients of the  polynomial,
i.e.~eq.~\eqref{eq:algeq1}.
This is of course a known fact, not connected with Fuss--Catalan numbers.

Note that eq.~\eqref{eq:algeq1} has $n$ roots, counting multiplicities,
but eq.~\eqref{eq:pivot_pq} yields $q-p$ roots.
If $p=0$ and $q=n$, so $q-p=n$, then eq.~\eqref{eq:pivot_pq} yields expressions for all the $n$
roots of eq.~\eqref{eq:algeq1}.
If $q-p<n$ then we require multiple series to obtain all the $n$ roots of eq.~\eqref{eq:algeq1}.
It is simplest to explain with an example.
Choose $p=0$ and $q=1$, this yields only one root.
Next choose $p=1$ and $q=n$, this yields $n-1$ roots.
One must try different selections for $p$ and $q$ to verify that expressions 
for all the roots of eq.~\eqref{eq:algeq1} have been found. 
We shall see this in connection with the trinomial below.

It is also possible to choose $q<p$.
Doing so yields the same set of roots of eq.~\eqref{eq:algeq1} 
obtained by interchanging $p$ and $q$.
This can be seen with some elementary transformations and relabelling of indices.
The details are left to the reader.
Hence without loss of generality we may assume $p<q$.

For all the nonvanishing $a_j$, from eq.~\eqref{eq:conv_nec_j}, 
for absolute convergence we require (necessary, not sufficient)
\bq
\label{eq:alg_b_nec}
|b_j| = \frac{|a_j|}{|a_p|^{1-\mu_j}|a_q|^{\mu_j}} \le \frac{1}{|\mu_j|^{\mu_j} |1-\mu_j|^{1-\mu_j}} \,.
\eq
Let the lowest and highest indices of the nonzero $a_j$ in $\mathscr{N}_{npq}$ be $j_{\min}$ and $j_{\max}$, respectively.
From eq.~\eqref{eq:radconv}, the sufficient (but not necessary) criterion for absolute convergence is
\bq
\label{eq:alg_b_suff}
\begin{split}
\sum_{j\in\mathscr{N}_{npq}} |b_j| = 
\sum_{j\in\mathscr{N}_{npq}} \frac{|a_j|}{|a_p|^{1-\mu_j}|a_q|^{\mu_j}} &\le 
\min\biggl( \frac{1}{|\mu_{j_{\min}}|^{\mu_{j_{\min}}} |1-\mu_{j_{\min}}|^{1-\mu_{j_{\min}}}} \,,
\\
&\qquad\qquad\qquad
\frac{1}{|\mu_{j_{\max}}|^{\mu_{j_{\max}}} |1-\mu_{j_{\max}}|^{1-\mu_{j_{\max}}}} \biggr) \,.
\end{split}
\eq
As noted in Sec.~\ref{sec:conv1}, the domain of absolute convergence depends only on the amplitudes $|a_j|$
and hence the domain is the same for all the choices of radicals for $(-1)^{1/(q-p)}$, i.e.~all the values of $\ell$.

\subsection{Comment on McClintock's series}
McClintock in 1895 published a paper \cite{McClintock} 
deriving series expressions for all the roots of a polynomial of arbitrary degree.
He began with the illustrative example $x^6 = -1 -x$.
McClintock obtained \cite[eq.~(1)]{McClintock} 
\bq
x = \omega -\omega^2 a -\frac32 \omega^3a^2 -\frac83 \omega^4 a^3 - \cdots 
\eq
where $a=-\frac16$ and ``$\omega$ is any one of sixth-roots of $-1$.''
This is essentially exactly the procedure I employed above:
in eq.~\eqref{eq:x_zeta_pq} I took an $n^{th}$ root of the highest power $x^n$
and solved for $\zeta_\ell$ in eq.~\eqref{eq:poly_root1}.
Note that $e^{i\pi(2\ell+1)/n}$ is an $n^{th}$ root of $-1$. 
McClintock examined many other polynomials, on a case by case basis;
the analysis in the previous subsection gives a general expression for the roots of an arbitrary polynomial.
McClintock's solutions are multiparameter Fuss--Catalan series.
McClintock noted that his series had finite radii of convergence
but did not derive a general expression for the radius of convergence.
McClintock also noted the use of the Lagrange inversion theorem in his derivations.

\subsection{Comment on Mellin's solution}
Mellin derived a series solution for the following algebraic equation
\cite{Mellin}
\bq
\label{eq:poly_mellin_nach}
z^n +x_1z^{n_1} +x_2z^{n_2} +\cdots +x_pz^{n_p} -1 = 0 \,.
\eq
Here $n > n_s \ge 1$, $s=1,\dots,p$ (see \cite{Mellin}).
Hence all the coefficients $x_s$ in eq.~\eqref{eq:poly_mellin_nach} are nonzero by definition.
Mellin derived a series solution for the 
`Hauptl{\"o}sung' or principal root,
which is the unique branch which equals 1
for $x_1=\cdots=x_p=0$,
and where $\alpha$ is a positive number \cite{Mellin}
\bq
z^\alpha = 1 +\alpha\sum_{k=1}^\infty \frac{(-1)^k}{n^k}
\sum_{\nu_1+\cdots+\nu_p=k}
\frac{\prod_{\mu=1}^{k-1} (\alpha+n_1\nu_1+\cdots n_p\nu_p -n\mu)}
{\Gamma(\nu_1+1)\Gamma(\nu_2+1)\cdots\Gamma(\nu_p+1)}\,
x_1^{\nu_1}\cdots x_p^{\nu_p} \,.
\eq
It is easily verified that this equals the Fuss--Catalan series
with $k=p$, $r=\alpha/n$, 
$\mu_j=n_j/n$, $t_j=\nu_j$ and $z_j = -x_j$, for $j=1,\dots,p$, so
\bq
z^\alpha = \mathcal{B}\Bigl(
\Bigl(\frac{n_1}{n},\dots,\frac{n_p}{n}\Bigr);\frac{\alpha}{n};
(-x_1,\dots,-x_p)\Bigr) \,.
\eq
In fact $\alpha$ is not constrained to be positive.
Mellin also specified the following bound for the domain of convergence of his series;
it is clearly sufficient but not always necessary
\cite{Mellin}
\bq
\label{eq:radconv_mellin}
|x_1|, \dots, |x_p| < 
\frac{1}{p} \min\biggl( \frac{1}{|\mu_1|^{\mu_1} |1-\mu_1|^{1-\mu_1}} \,, \frac{1}{|\mu_p|^{\mu_p} |1-\mu_p|^{1-\mu_p}} \biggr) \,.
\eq
For ease of comparison with my work,
I have written $\mu_1$ and $\mu_p$ on the right hand side.
This is a more conservative bound and is superseded by eq.~\eqref{eq:radconv}
or eq.~\eqref{eq:alg_b_suff}.

\subsection{Comment on Birkeland's series}\label{sec:Birkeland}
Birkeland published papers on the solutions of algebraic equations using hypergeometric series
\cite{Birkeland_trinom,Birkeland_quintic,Birkeland_algebraic_1920},
culminating in his 1927 paper \cite{Birkeland_algebraic_1927}.
We study the latter paper (which largely subsumes his earlier work).
Birkeland treated the general algebraic equation with complex coefficients 
\cite[eq.~(1)]{Birkeland_algebraic_1927} 
\bq
\label{eq:birk_eq1}
a_0x^n + a_ax^{n-1} +\cdots +a_{n-1}x +a_n = 0 \,.
\eq
Hence his indexing is the opposite of that in eq.~\eqref{eq:algeq1}.
Birkeland also selected two integers, with $p>q$, such that $0 \le q < p \le n$.
He then obtained the scaled equation
\cite[eq.~($1^\prime$)]{Birkeland_algebraic_1927}
(see his paper for the definitions of $z$, $l_i$ and $m_i$)
\bq
\label{eq:birk_eq1p}
z^p = z^q + l_1 z^{m_1} + l_2 z^{m_2} +\cdots + l_s z^{m_s} \,. 
\eq
This is very similar to eq.~\eqref{eq:algeq3}.
Birkeland then derived a series solution of eq.~\eqref{eq:birk_eq1p}.
First define $\varepsilon$ as a primitive root of unity satisfying the equation $x^{p-q}=1$.
Then Birkeland obtained for the root $z_j$
(raised to the power $\gamma$), where $j=1,2,\dots,p-q$
\cite[eq.~(5)]{Birkeland_algebraic_1927}
\bq
\label{eq:birk_eq5}
z_j^\gamma = \varepsilon^{j\gamma} \biggl[\,
1 + \frac{\gamma}{p-q}\sum_{\alpha_1,\dots,\alpha_s=0}^\infty 
\varepsilon^{jv} \,\frac{(\tau,r-1)}{\alpha_1!\alpha_2!\dots\alpha_s!}\, l_1^{\alpha_1}l_2^{\alpha_2}\dots l_s^{\alpha_s}
\,\biggr] \,.
\eq
({\em N.B.:} I changed $i$ to $j$ in Birkeland's equation to avoid confusion with $i=\sqrt{-1}$.)
With elementary changes of notation, 
eq.~\eqref{eq:birk_eq5} is equivalent to the first line of eq.~\eqref{eq:pivot_pq}.
See in particular \cite[eq.~(4)]{Birkeland_algebraic_1927}
for his definitions of $\tau$ and $v$.
Birkeland did not recognize his series coefficients as Fuss--Catalan numbers.
Birkeland then expressed the series solution in
eq.~\eqref{eq:birk_eq5}
in terms of sums of hypergeometric series.
It would take us too far afield to discuss hypergeometric series in this paper.
Birkeland derived the same convergence criteria as in Sec.~\ref{sec:conv1}.
He derived the necessary (but not sufficient) bound
\cite[unnumbered, \S 2]{Birkeland_algebraic_1927}
\bq
|\zeta_1| < 1 \,, \quad
|\zeta_2| < 1 \,, \quad \dots, \quad
|\zeta_s| < 1 \,.
\eq
This matches eq.~\eqref{eq:alg_b_nec},
after working through the details of his notation:
his $\zeta_j$ equals my $b_j |\mu_j|^{\mu_j} |1-\mu_j|^{1-\mu_j}$,
with obvious allowances for differences in his indexing.
Birkeland did not recognize that one can write `$\le$' instead of strict inequalities `$<$' in the bound.
As for the sufficient (but not necessary) bound, Birkeland obtained
\cite[eq.~12]{Birkeland_algebraic_1927}
\bq
\label{eq:birk_eq12}
|l_1| +|l_2| +\cdots +|l_s| < \frac{p-q}{m+p-q}\frac{1}{\displaystyle \Bigl(1+\frac{p-q}{m}\Bigr)^{\frac{m}{p-q}}} \,.
\eq
This is clearly similar to eq.~\eqref{eq:alg_b_suff}.
From eq.~\eqref{eq:birk_eq1p}, 
Birkeland's $l_j$ are my $b_j$, and on the right hand side of eq.~\eqref{eq:birk_eq12},
he wrote out the bound explicitly in terms of integers $m$, $p$ and $q$.
Contrary to Mellin \cite{Mellin}
(see eq.~\eqref{eq:radconv_mellin}) and myself (see eq.~\eqref{eq:alg_b_suff}),
Birkeland did not write a `min' of two possible choices for the best bound.
Here Birkeland made an error of algebra:
Birkeland defined $m$ in eq.~\eqref{eq:birk_eq12} as the value of $m_\nu$ which maximizes the value of $|m_\nu-p|$.
Quoting from \cite{Birkeland_algebraic_1927},
``Wir wollen mit $m$ die gr{\"o}{\ss}te der Zahlen $|m_\nu-p|$ \dots''
However, as was seen in Sec.~\ref{sec:conv1} for the parameter $\mu_*$,
we must choose $\mu_*$ to be the value of $\mu_j$ which maximizes the value of $|\mu_j-\frac12|$.
Working through Birkeland's notation, we must choose $m$ to maximize the value of $|m_\nu -(p+q)/2|$.
Birkeland then applied his formalism to derive the solution of the trinomial,
which he had also treated in an earlier paper \cite{Birkeland_trinom}.
The trinomial is sufficiently important that it will be studied in a section of its own in Sec.~\ref{sec:tri}.

\subsection{Comment on Lewis's series}
In 1939, Lewis \cite{Lewis}
published a paper on the solution of algebraic equations by infinite series. 
He treated the trinomial, then the quadrinomial and finally general multinomial equations.
We discuss only the general case here.
Lewis treated the general algebraic equation with complex coefficients 
\cite[eq.~(39)]{Lewis} 
\bq
\label{eq:lewis_eq39}
a_nz^n -a_kz^k -a_gz^g - \cdots -a_bz^b -a_0 = 0 \,.
\eq
(The above corrects a misprint in \cite{Lewis}.)
The notation suggests that all the coefficients are nonzero.
Lewis treated only the case we denoted above by $p=0$ and $q=n$.
He wrote
\cite[eq.~(40)]{Lewis} 
\bq
\label{eq:lewis_eq40}
z^n = a_0/a_n + (1/a_n)(a_kz^k +a_gz^g +\cdots +a_bz^b) \,.
\eq
Lewis employed Lagrange inversion to derive his solution.
Following Lewis, we write $a_0/a_n = re^{i\theta}$ 
and the $n$ roots of $a_nz^n - a_0 = 0$ are denoted by 
$\alpha_h = r^{1/n}e^{i(2h\pi+\theta)/n}$, where $h=1,\dots,n$.
The solution for the root $z_h$ is given as 
\cite[eq.~(41)]{Lewis} 
\bq
\label{eq:lewis_eq41}
z_h = \sum_{p,q,\dots,v=0}^\infty \frac{a_k^p a_g^q \dots a_b^v}{p!q!\dots v! (a_0n)^{p+q+\cdots+v}}
\binom{1 +pk +\cdots +vb -n}{p+q+\cdots+v-1} \,\alpha_h^{1 +pk +\cdots +vb} 
\qquad (h=1,\dots,n) \,.
\eq
Lewis did not derive a series for powers of the roots.
With some effort, eq.~\eqref{eq:lewis_eq41} can be equated to the solution in eq.~\eqref{eq:pivot_pq}.
Lewis derived the following sufficient condition for absolute convergence
\cite[eq.~(42)]{Lewis} 
\bq
\label{eq:lewis_eq42}
\biggl\{ \frac{|a_k\alpha_h^k| +\cdots +|a_b\alpha_h^b|}{|a_0|}\biggr\}^n 
\le \frac{n^n}{k^k(n-k)^{n-k}} \,.
\eq
Unlike Mellin \cite{Mellin} and Birkeland \cite{Birkeland_algebraic_1927},
Lewis \cite{Lewis} recognized that equality `$\le$' is permitted in the bound.
However, like Birkeland, Lewis failed to recognize that the bound on the right hand side 
is given by the minimum of multiple possibilities,
and the expression he derived is not always the correct choice.

\subsection{Comment on Raney's series}
Raney \cite{Raney}
employed his formalism to demonstrate the use of Lagrange inversion for a power series
$\bar{z} = \sum_{n=0}^\infty a_n\bar{x}^n$ 
\cite[eqs.~(5,4), (5,7) and (5.8)]{Raney}.
Then in \cite[Sec.~6]{Raney},
Raney applied his formalism to derive series solutions for
algebraic equations.
As mentioned earlier,
\cite[eqs.~(6.1) and (6.2)]{Raney} yield eq.~\eqref{eq:f_mult}.
Raney took the $\mu_j$ to be integers;
this yields an algebraic equation.
Raney displayed the example of the trinomial
$\bar{w} = 1 +\bar{x}\bar{w}^n$
\cite[eq.~(6.3)]{Raney}, 
with the series solution
\cite[eq.~(6.4)]{Raney}
\bq
\bar{w} = \sum_{k=0}^\infty \frac{1}{1+(n-1)k} \binom{nk}{k} \bar{x}^k \,.
\eq
This matches the series coefficients in eq.~\eqref{eq:fc_def_int},
replacing $m$ by $n$ and $t$ by $k$.
Raney did not discuss questions of convergence.
Unlike the other authors cited earlier in this section,
Raney took the coefficients in his equations to be elements
in a commutative ring, not just complex numbers.

\setcounter{equation}{0}
\section{\label{sec:quintic} Quintic}
The quintic is sufficiently important that it is placed in a separate section.
It is known that by means of a Tschirnhaus transformation, a general quintic may be brought to the Bring-Jerrard normal form
\bq
\label{eq:bj}
x^5 - x + \gamma = 0 \,.
\eq
This algebraic equation (of degree $n=5$) lends itself naturally to a solution using Fuss--Catalan series.
Using the formalism in Sec.~\ref{sec:quintic}, we set $p=0$ and $q=5$.
From eq.~\eqref{eq:x_zeta_pq}, 
$x_\ell = e^{i\pi(2\ell+1)/5} \gamma^{1/5} \zeta_\ell^{1/5}$ with $\ell=0,\dots,4$.
Then $\zeta_\ell$ satisfies the equation
\bq
\zeta_\ell = 1 - \frac{e^{i\pi(2\ell+1)/5}}{\gamma^{4/5}}\,\zeta_\ell^{1/5} \,.
\eq
There is only one summand, so $k=1$ and $\mu=1/5$. 
The roots $x_\ell$ are given by
\bq
\label{eq:bj_soln1}
\begin{split}
x_\ell = e^{i\pi(2\ell+1)/5} \gamma^{1/5} \mathcal{B}\Bigl(\frac15;\, \frac15;\, -\frac{e^{i\pi(2\ell+1)/5}}{\gamma^{4/5}}\Bigr) \,.
\end{split}
\eq
From Corollary \ref{eq:cor_k1}, the condition for convergence is necessary and sufficient:
\bq
\label{eq:bj_conv1}
\frac{1}{|\gamma|^{4/5}} \le \frac{1}{(\frac15)^{1/5}(\frac45)^{4/5}} = \frac{5}{4^{4/5}} \,,\qquad
|\gamma| \ge \frac{4}{5^{5/4}} \simeq 0.534992 \,.
\eq
We can say more. 
What if the value of $\gamma$ does not satisfy the above bound?
There are alternative series we can derive.
Rewrite eq.~\eqref{eq:bj} as $x = \gamma + x^5$ and set $\zeta=x/\gamma$. 
This corresponds to $p=0$ and $q=1$.
Then $\zeta$ satisfies
$\zeta = 1 +\gamma^4 \zeta^5$.
Once again $k=1$ and now $z_1=\gamma^4$ and $\mu=5$. 
There is only one root and it is
\bq
\label{eq:bj_soln2}
x = \gamma\zeta = \gamma \mathcal{B}(5;1;\gamma^4) \,.
\eq
{\em Comment: eq.~\eqref{eq:bj_soln2} is equivalent to the solution published by Eisenstein 
\cite{Eisenstein_1844} (see also \cite{Stillwell}, in English).
Eisenstein expressed the Bring-Jerrard normal form differently and solved the equation $x^5+x+\gamma=0$.}  
The necessary and sufficient condition for convergence is 
\bq
\label{eq:bj_conv2}
|\gamma|^4 \le \frac{1}{5^5 4^{-4}} = \frac{4^4}{5^5} \,,\qquad
|\gamma| \le \frac{4}{5^{5/4}} \,.
\eq
This is the inverse of the condition in eq.~\eqref{eq:bj_conv1}.
However, we have found only one root.
We obtain the other four roots as follows. 
We divide eq.~\eqref{eq:bj} through by $x$ to obtain
$x^4 = 1 - \gamma/x$.
This corresponds to $p=1$ and $q=5$.
Set $\zeta=x^4$ or $x = e^{i2\pi\ell/4} \zeta^{1/4}$, for $l=0,\dots,3$, so $\zeta_\ell$ satisfies
\bq
\zeta_\ell = 1 -\gamma e^{-i\pi\ell/2} \zeta_\ell^{-1/4} \,.
\eq
Once again $k=1$ and now $z_1= -\gamma e^{-i\pi\ell/2}$ and $\mu=-\frac14$. 
There are four roots, indexed by $\ell=0,\dots,3$
\bq
\label{eq:bj_soln3}
\begin{split}
x_\ell = e^{i\pi\ell/2} \mathcal{B}\Bigl(-\frac14;\, \frac14;\, -e^{-i\pi\ell/2} \gamma\Bigr) \,.
\end{split}
\eq
The necessary and sufficient condition for convergence is 
\bq
|\gamma| \le \frac{1}{(\frac14)^{1/5}(\frac54)^{5/4}} = \frac{4}{5^{5/4}} \,.
\eq
This is the same as eq.~\eqref{eq:bj_conv2}.
As noted earlier for the case $k=1$, all the series converge on their respective circles of convergence.

The series in eqs.~\eqref{eq:bj_soln2} and \eqref{eq:bj_soln3} do not yield the same root.
If we set $\gamma=0$, eq.~\eqref{eq:bj} reduces to $x(x^4-1)=0$.
One root is zero and the others are the fourth roots of unity.
The roots in the series in eqs.~\eqref{eq:bj_soln2} and \eqref{eq:bj_soln3} 
lie on the branches which respectively approach zero and the fourth roots of unity.
Hence the two series, taken together, yield all five roots of eq.~\eqref{eq:bj}.

There are multiple ways to find the roots of a polynomial using Fuss--Catalan series.
The series in eq.~\eqref{eq:bj_soln1} converges for $|\gamma| \ge 4/5^{5/4}$
whereas those in eqs.~\eqref{eq:bj_soln2} and \eqref{eq:bj_soln3}
converge for $|\gamma| \le 4/5^{5/4}$.
In all cases one obtains convergent solutions for all the five roots of eq.~\eqref{eq:bj},
thence the general quintic.

The Bring-Jerrard normal form has been solved using hypergeometric functions,
e.g.~see \cite{Nash}.
Klein's solution of the quintic \cite{Klein}
also employed hypergeometric series.
The series have finite radii of convergence (actually the same radius for all the series).
Analytic continuation is required to treat all coefficients of the general quintic.
Using Fuss--Catalan series,
the series in eqs.~\eqref{eq:bj_soln2} and \eqref{eq:bj_soln3} 
are the explicit analytic continuations of 
the series in eq.~\eqref{eq:bj_soln1}
across the `boundary' $|\gamma| = 4/5^{5/4}$.
Together they cover the whole parameter space, i.e.~all values of $\gamma$ in eq.~\eqref{eq:bj}.
The solution of the Bring-Jerrard normal form using Fuss--Catalan series 
is arguably `cleaner' than that using hypergeometric series.

\setcounter{equation}{0}
\section{Trinomial}\label{sec:tri} 
\subsection{General solution}
The trinomial is also sufficiently important that it is placed in a separate section.
The Bring-Jerrard normal form of the quintic and 
Lambert's trinomial, to be discussed in Sec.~\ref{sec:lamb},
are particular examples of the general trinomial equation 
\bq
\label{eq:tri}
x^{m+n} + ax^n + b = 0 \,.
\eq
The general solution of the trinomial, 
for arbitrary values of the coefficients,
was derived by
Birkeland (1920,1927) \cite{Birkeland_trinom,Birkeland_algebraic_1927}, 
Lewis (1935) \cite{Lewis} and 
Eagle (1939) \cite{Eagle_trinomial}
(who employed McClintock's \cite{McClintock} formalism). 
The general solution of the trinomial was also derived in 1908 by
P.~A.~Lambert \cite{PA_Lambert_1908},
not to be confused with Johann Lambert.
P.~A.~Lambert in fact also presented a series solution for the general algebraic equation,
but his analysis contained some technical errors and was not discussed in 
Sec.~\ref{sec:algeq1}.
The derivation below follows Eagle (1939) \cite{Eagle_trinomial},
and eq.~\eqref{eq:tri} is taken from his paper.
We have already noted that the solutions are Fuss--Catalan series.
All four authors cited above derived correct expressions for the radii of convergence of their series.
The derivation below may be considered as an independent validation of their results.

As was seen in Sec.~\ref{sec:quintic}, there are three series.
To systematize the derivation, to give a more panoramic overview of the results, we proceed as follows.
Here $\ell$ takes values as appropriate to index the roots of unity.
\begin{itemize}
\item
Set $p=0$ and $q=m+n$ and 
$x_\ell = e^{i\pi(2\ell+1)/(m+n)} b^{1/(m+n)} \zeta_\ell^{1/(m+n)}$ then $\zeta_\ell$ satisfies
\bq
\zeta_\ell = 1+ e^{i\pi(2\ell+1)n/(m+n)} \frac{a}{b^{m/(m+n)}} \zeta_\ell^{n/(m+n)} \,.
\eq
Then $\mu = n/(m+n)$. This series yields $m+n$ roots.
\item 
Set $p=0$ and $q=n$ and 
$x_\ell=e^{i\pi(2\ell+1)/n} (b\zeta_\ell/a)^{1/n}$ then $\zeta_\ell$ satisfies 
\bq
\zeta_\ell = 1 +e^{i\pi (2\ell+1)(m+n)/n} \frac{b^{m/n}}{a^{(m+n)/n}} \zeta_\ell^{(m+n)/n} \,.
\eq
Then $\mu = (m+n)/n$. This series yields $n$ roots.
\item 
Set $p=n$ and $q=m+n$ and 
divide eq.~\eqref{eq:tri} through by $x^n$. 
Set $x_\ell=e^{i\pi(2\ell+1)/m} (a\zeta_\ell)^{1/m}$ then $\zeta_\ell$ satisfies 
\bq
\label{eq:trinom_3}
\zeta_\ell = 1 +e^{-i\pi (2\ell+1)n/m} \frac{b}{a^{(m+n)/m}} \zeta_\ell^{-n/m} \,.
\eq
Then $\mu = -n/m$. This series yields $m$ roots.
\end{itemize}
The respective series solutions are 
\begin{subequations}
\begin{align}
\label{eq:trinom_soln_1}
x_\ell &= e^{i\pi (2\ell+1)/(m+n)} b^{1/(m+n)} 
\mathcal{B}\Bigl(\frac{n}{m+n};\,\frac{1}{m+n};\, e^{i\pi\frac{(2\ell+1)n}{m+n}} \frac{a}{b^{m/(m+n)}} \Bigr) 
&\qquad& (0 \le \ell \le m+n-1) \,,
\\
\label{eq:trinom_soln_2}
x_\ell &= e^{i\pi(2\ell+1)/n} \frac{b^{1/n}}{a^{1/n}} 
\mathcal{B}\Bigl(\frac{m+n}{n};\,\frac{1}{n};\, e^{i\pi (2\ell+1)(m+n)/n} \frac{b^{m/n}}{a^{(m+n)/n}} \Bigr) 
&\qquad& (0 \le \ell \le n-1) \,,
\\
\label{eq:trinom_soln_3}
x_\ell &= e^{i\pi(2\ell+1)/m} a^{1/m}
\mathcal{B}\Bigl(-\frac{n}{m};\,\frac{1}{m};\, e^{-i\pi (2\ell+1)n/m} \frac{b}{a^{(m+n)/m}} \Bigr) 
&\qquad& (0 \le \ell \le m-1) \,.
\end{align}
\end{subequations}
The respective domains of convergence are as follows.
\begin{subequations}
\begin{align}
\frac{|a|}{|b|^{m/(m+n)}} &\le \frac{(m+n)}{n^{n/(m+n)}m^{m/(m+n)}} \,, &\qquad 
\frac{|b|^m}{|a|^{m+n}} &\ge \frac{m^mn^n}{(m+n)^{m+n}} \,,
\\
\frac{|b^{m/n}|}{|a|^{(m+n)/n}} &\le \frac{n}{(m+n)m^{-m/n}} \,, &\qquad 
\frac{|b|^m}{|a|^{m+n}} &\le \frac{m^mn^n}{(m+n)^{m+n}} \,,
\\
\frac{|b|}{|a|^{(m+n)/m}} &\le \frac{m}{n^{-n/m}(m+n)^{(m+n)/m}} \,, &\qquad 
\frac{|b|^m}{|a|^{m+n}} &\le \frac{m^mn^n}{(m+n)^{m+n}} \,.
\end{align}
\end{subequations}
If we set $b\to0$ then $x^n(x^m+a)\to0$. 
Hence $n$ roots approach zero and $m$ roots approach the respective $m^{th}$ roots of $-a$.
The roots of the second and third series respectively lie on the branches which approach zero
and the $m^{th}$ roots of $-a$ as $b\to0$.
The second and third series have the same domain of convergence 
and hence together yield all the $m+n$ roots of eq.~\eqref{eq:tri}.
The above set of three series are those found by Eagle \cite{Eagle_trinomial} and together yield
all the roots of the general trinomial for arbitrary values of the coefficients.
They are equivalent to the solutions derived by
P.~A.~Lambert \cite{PA_Lambert_1908},
Birkeland \cite{Birkeland_trinom,Birkeland_algebraic_1927} and
Lewis \cite{Lewis}.

Consider also the following.
Divide eq.~\eqref{eq:tri} through by $x^n$ as above, 
but now set $x_\ell=e^{i\pi(2\ell+1)/n} (a\zeta_\ell/b)^{-1/n}$. 
This corresponds to setting $p=n$ and $q=0$, i.e.~$q<p$.
Then $\zeta_\ell$ satisfies
\bq
\label{eq:trinom_4}
\zeta_\ell = 1 + e^{i\pi(2\ell+1)m/n} \frac{b^{m/n}}{a^{(m+n)/n}} \zeta_\ell^{-m/n} \,.
\eq
Compare this to eq.~\eqref{eq:trinom_3}
Now $\mu = -m/n$. This series yields $n$ roots.
It converges if and only if
\bq
\frac{|b|^{m/n}}{|a|^{(m+n)/n}} \le \frac{n}{m^{-m/n}(m+n)^{(m+n)/n}} \,, \qquad\qquad 
\frac{|b|^m}{|a|^{m+n}} \le \frac{m^mn^n}{(m+n)^{m+n}} \,.
\eq
The solution is
\bq
\label{eq:trinom_soln_4}
x_\ell = e^{i\pi(2\ell+1)/n} \frac{b^{1/n}}{a^{1/n}} 
\mathcal{B}\Bigl(-\frac{m}{n};\,-\frac{1}{n};\, e^{i\pi (2\ell+1)m/n} \frac{b^{m/n}}{a^{(m+n)/n}} \Bigr) 
\qquad (0 \le \ell \le n-1) \,.
\eq
Hence this series yields the same $n$ roots as the second series above, for which $\mu=(m+n)/n$,
with the same domain of convergence.
It is therefore the same series as in
eq.~\eqref{eq:trinom_soln_2}.
Even the permutations of the roots are identical, 
because the first terms of both series are
$e^{i\pi(2\ell+1)/n} (b/a)^{1/n}$. 
It was remarked in Sec.~\ref{sec:algeq1} that choosing $q<p$ yields the same solutions
as the series obtained by interchanging $p$ and $q$.

\begin{itemize}
\item
Johann Lambert \cite{Lambert_1758} solved the equation $x^m +px = q$ in 1758. 
Corless et al.~\cite{W_fcn_Corless_etal} stated
``In 1758, Lambert solved the trinomial equation $x = q + x^m$ 
by giving a series development for $x$ in powers of $q$.''
I found that the equation $x = q + x^m$ appears in Lambert's 1770 paper \cite[\S 8]{Lambert_1770}, 
where Lambert stated 
``auquel on peut toujous donner la forme plus simple''
(``to which we can always provide the simpler form'') 
and where Lambert derived a series for the $n^{th}$ power $x^n$.
Lambert's solutions are Fuss--Catalan series,
for the branch which approaches zero when $q\to0$.
Lambert's series solution for $x^n$, 
where $x = q + x^m$, may be written as \cite[\S 8]{Lambert_1770}
\bq
\label{eq:lamb_soln}
x^n = q\mathcal{B}(m;n;q^{m-1}) \,.
\eq
Lambert did not specify the radius of convergence of his series.
It converges if and only if
\bq
|q| \le \frac{m-1}{m^{m/(m-1)}} \,.
\eq

\item
Ramanujan also solved the trinomial via a series.
The equation he treated was \cite[first quarterly report, 1.6 (iv), eq.~(1.15)]{Ramanujan_Pt1}
\bq
\label{eq:Ramanujan_trinom}
aqx^p + x^q = 1 \,.
\eq
Ramanujan derived the following solution for any power $n$
\cite[first quarterly report, 1.6 (iv), eq.~(1.16)]{Ramanujan_Pt1}
\bq
\label{eq:Ramanujan_soln}
x^n = \frac{n}{q} \sum_{k=0}^\infty \frac{\Gamma(\{n+pk\}/q)(-qa)^k}{\Gamma(\{n+pk\}/q -k+1)k!} \,.
\eq
This is the branch which approaches unity for $a\to0$.
The above expression is stated in \cite{Ramanujan_Pt1} to be valid for all real numbers $n$, $p$, $q$
and for complex $a$ satisfying
\bq
\label{eq:Ramanujan_conv}
|a| \le |p|^{-p/q} |p-q|^{(p-q)/q} \,.
\eq
Let us verify Ramanujan's solution.
The expression in eq.~\eqref{eq:Ramanujan_soln} is tricky if $n=0$.
The first term in the sum is actually unity
\bq
\begin{split}
x^n &= \frac{(n/q)\Gamma(n/q)}{\Gamma(n/q+1)} 
+ \frac{n}{q} \sum_{k=1}^\infty \frac{(-qa)^k}{k!} \prod_{u=1}^{k-1} (kp/q +n/q - u)
\\
&= 1+ \frac{n}{q} \sum_{k=1}^\infty \frac{(-qa)^k}{k!} \prod_{u=1}^{k-1} (kp/q +n/q - u) \,.
\end{split}
\eq
We need to perform the cancellations before setting $n=0$ on the right hand side.
To solve eq.~\eqref{eq:Ramanujan_trinom} using a Fuss--Catalan series (for the branch treated by Ramanujan), 
put $\zeta=x^q$, so $x=\zeta^{1/q}$, then
$\zeta = 1 - aq\zeta^{p/q}$.
Hence $\mu=p/q$ and $z = -qa$ in eq.~\eqref{eq:f_gkp}.
The solution is (using $k$ as a summation variable)
\bq
\begin{split}
x^n = \zeta^{n/q} &= \sum_{k=0}^\infty A_k(p/q,n/q)(-qa)^k 
\\
&= 1 + \frac{n}{q} \sum_{k=1}^\infty \frac{(-qa)^k}{k!} \prod_{u=1}^{k-1} (kp/q +n/q - u) \,.
\end{split}
\eq
This equals the expression in eq.~\eqref{eq:Ramanujan_soln}.
The series converges if and only if
\bq
\begin{split}
|a| &\le \frac{1}{|q|}\,\frac{1}{|p/q|^{p/q} |1-p/q|^{1-p/q}} \,,
\\
&= |p|^{-p/q} |p-q|^{(p-q)/q} \,.
\end{split}
\eq
This confirms the bound in eq.~\eqref{eq:Ramanujan_conv}.
\end{itemize}

\subsection{\label{sec:lamb} Lambert and Euler trinomial equations}
At stated above, in 1758 Lambert \cite{Lambert_1758} gave a series solution for the trinomial equation $x^m+px=q$
and later in 1770, Lambert \cite[\S 8]{Lambert_1770} revisited the equation in the form
\bq
\label{eq:lambert_tri}
x = q+x^m \,.
\eq
The treatment below follows Corless et al.~\cite{W_fcn_Corless_etal}.
In 1779 Euler \cite{Euler_1779} derived the following equation from Lambert's trinomial
(I have changed Euler's `$x$' to `$z$' to avoid confusion as to which equation $x$ satisfies)
\bq
\label{eq:euler_tri}
z^\alpha - z^\beta = (\alpha-\beta)v z^{\alpha+\beta} \,.
\eq
This is obtained from eq.~\eqref{eq:lambert_tri} 
via the substitutions
$x = z^{-\beta}$, $m=\alpha/\beta$ (this corrects a misprint in \cite{W_fcn_Corless_etal}, which stated $m=\alpha\beta$)
and $q=(\alpha-\beta)v$.
Euler's solution of eq.~\eqref{eq:euler_tri}, for $z^n$, was \cite{Euler_1779}
\bq
\label{eq:euler_soln}
\begin{split}
z^n = 1 + nv &+\frac{1}{2!}\,n(n+\alpha+\beta) \,v^2
\\
&+\frac{1}{3!}\,n(n+\alpha+2\beta)(n+2\alpha+\beta) \,v^3
\\
&+\frac{1}{4!}\,n(n+\alpha+3\beta)(n+2\alpha+2\beta)(n+3\alpha+\beta) \,v^4 +\cdots
\end{split}
\eq
Clearly, eq.~\eqref{eq:lambert_tri} can be solved using Fuss--Catalan series.
There are $m$ roots, of which one approaches 0 and $m-1$ approach the roots of unity as $q\to0$.
Lambert's solution is the unique branch which vanishes for $q=0$
and was displayed as a Fuss--Catalan series in eq.~\eqref{eq:lamb_soln}.
Euler's solution in eq.~\eqref{eq:euler_soln} does not vanish for $\alpha=\beta$, 
i.e.~$q=0$, and is the unique solution of eq.~\eqref{eq:lambert_tri} 
which is real (if $q$ is real) and approaches 1 as $q\to0$.
We can derive it as follows.
Divide eq.~\eqref{eq:lambert_tri} through by $x^m$ and set $x=\zeta^{-1/(m-1)}$, then
$\zeta = 1 + q\zeta^{m/(m-1)}$.
Hence $\mu = m/(m-1)$ and the solution is 
\bq
x = \mathcal{B}\Bigl(\frac{m}{m-1};\,-\frac{1}{m-1};\,q\Bigr) \,.
\eq
Put $x=z^{-\beta}$, $m=\alpha/\beta$ and $q=(\alpha-\beta)v$, then $z^n=x^{-n/\beta}$
\bq
\label{eq:lamb_fc}
z^n = \mathcal{B}\Bigl(\frac{\alpha}{\alpha-\beta};\,\frac{n}{\alpha-\beta};\,(\alpha-\beta)v\Bigr) \,.
\eq
Let us verify from eq.~\eqref{eq:euler_soln} that this equals Euler's solution:
\bq
\begin{split}
z^n &= 1 +n \sum_{t=1}^\infty \frac{v^t}{t!}\,\prod_{j=1}^{t-1}(t\alpha+n-j(\alpha-\beta)) 
\\
&= 1 +\frac{n}{\alpha-\beta} \sum_{t=1}^\infty \frac{(\alpha-\beta)^tv^t}{t!}
\,\prod_{j=1}^{t-1}\Bigl(\frac{t\alpha}{\alpha-\beta}+\frac{n}{\alpha-\beta}-j\Bigr) 
\\
&= \sum_{t=0}^\infty A_t\Bigl(\frac{\alpha}{\alpha-\beta},\,\frac{n}{\alpha-\beta}\Bigr)\,(\alpha-\beta)^tv^t 
\\
&= \mathcal{B}\Bigl(\frac{\alpha}{\alpha-\beta};\,\frac{n}{\alpha-\beta};\, (\alpha-\beta)v \Bigr) \,.
\end{split}
\eq
This agrees with eq.~\eqref{eq:lamb_fc}.
The paper by Corless et al.~\cite{W_fcn_Corless_etal} discusses various applications of the Lambert $W$ function,
which is the real solution (for real $x \ge -1/e$) of the equation $W(x)e^{W(x)}=x$,
but that is beyond the scope of this paper.

\setcounter{equation}{0}
\section{\label{sec:conv2} Algebraic equations: convergence of series I}
\subsection{General remarks}
A necessary and sufficient bound for the domain of absolute convergence is available 
for the important special case of
the solutions of algebraic equations by infinite series. 
We begin with some known theorems from the theory of power series in several complex variables.

\begin{definition}[multicircular or Reinhardt domain]
A multi-circular or Reinhardt domain in $\mathbb{C}^k$ has the property that 
for $k$ complex variables $\bm{z}=(z_1,\dots,z_k)$,
if a point $\bm{z}_*$ lies in the domain,
then so does every point $\bm{z}$ such that $|z_j|= |z_{*j}|$ for $j=1,\dots,k$.
A multi-circular domain with the property that if a point 
$\bm{z}_*$ lies in the domain,
then so does the polydisc given by $\{\bm{z} : |z_j| \le |z_{*j}|,\; j=1,\dots,k\}$
is known as a {\em complete Reinhardt domain}.
A polydisc is a Cartesian product of discs, in general with different radii.
\end{definition}
The convergence domain of a power series in multiple variables 
is a union of polydiscs centered at the origin and is a complete Reinhardt domain.
The following is also known. 
Using a vector notation, with coefficients $\bm{c}_{\bm{\alpha}}$
indexed by a $k$-tuple $\bm{\alpha}$,
if both $\sum_{\bm{\alpha}} |\bm{c}_{\bm{\alpha}} \bm{z}^{\bm{\alpha}}|$
and $\sum_{\bm{\alpha}} |\bm{c}_{\bm{\alpha}} \bm{w}^{\bm{\alpha}}|$
converge, then so does
$\sum_{\bm{\alpha}} |\bm{c}_{\bm{\alpha}}| |\bm{z}^{\bm{\alpha}}|^t |\bm{w}^{\bm{\alpha}}|^{1-t}$ for $0 \le t \le 1$.
This property of a Reinhardt domain is called {\em logarithmic convexity}.
Define a map $\textrm{Log}: (\mathbb{C}\setminus\{0\})^k \to \mathbb{R}^k$
where $z_j \mapsto \ln|z_j|$ for $j=1,\dots,k$.
Let the image of the domain of convergence 
$\mathcal{D}$ be $\textrm{Log}(\mathcal{D}) \subset \mathbb{R}^k$.
If $\textrm{Log}(z), \textrm{Log}(w) \in \textrm{Log}(\mathcal{D})$, then also
$t\textrm{Log}(z) +(1-t)\textrm{Log}(w) \in \textrm{Log}(\mathcal{D})$ for $0\le t \le 1$, i.e.
\bq
(t\ln|z_1| +(1-t)\ln|w_1|, \dots, t\ln|z_k| +(1-t)\ln|w_k|) \in \textrm{Log}(\mathcal{D}) \,. 
\eq
A complete Reinhardt domain in $\mathbb{C}^k$ is the domain of absolute convergence 
of a power series if and only if the domain is logarithmically convex.
The power series converges uniformly in every compact subset of the domain $\mathcal{D}$.
Note that logarithmic convexity does {\em not} imply convexity.
For $k=1$, the domain of convergence of a univariate power series is a disc in $\mathbb{C}$ centered on the origin, and is convex.
However, for $k\ge2$ variables, a complete Reinhardt domain in not in general convex.
However, from the foregoing remarks about polydiscs, the following is true.
If a point $\bm{z}_*$ lies in the domain of convergence, then so does every point on the ray joining the origin to $\bm{z}_*$,
i.e.~$\bm{z}=\lambda\bm{z}_*$ for $0 \le \lambda \le 1$.

The above theory is general.
In this section, we are concerned with the domain of absolute convergence of the series in eq.~\eqref{eq:pivot_pq},
which is the solution of eq.~\eqref{eq:algeq1}.
Passare and Tsikh \cite{Passare_Tsikh} claimed to offer a necessary and sufficient bound for absolute convergence in this case.
We summarize their work below.
We also display some counterexamples to their bound, 
and offer a more detailed analysis.
For ease of contact with the formalism in \cite{Passare_Tsikh}, we write
\bq
\label{eq:algeq_pt}
a_0 +a_1x +\cdots +x^p +\cdots +x^q +\cdots +a_nx^n = 0 \,.
\eq
This is effectively eq.~\eqref{eq:algeq1} (or eq.~\eqref{eq:algeq2}) where we have set $a_p=a_q=1$.
This is the equation treated in \cite{Passare_Tsikh}.
The series solution is given by eq.~\eqref{eq:pivot_pq}, with obvious changes of notation.
Passare and Tsikh employed the notation $[p]$ to denote that the index $p$ is excluded from a list of the form
$(\alpha_0,\alpha_1,\dots,[p],\dots,\alpha_n)$.
The solution of eq.~\eqref{eq:algeq_pt} is a power series in the $n-1$ variables 
$(a_0,a_1,\dots,[p],\dots,[q],\dots,a_n)$.
Passare and Tsikh studied the discriminant $\Delta_{pq}(a_0,a_1,\dots,[p],\dots,[q],\dots,a_n)$
which is the discriminant of the polynomial in eq.~\eqref{eq:algeq_pt}.
Then Passare and Tsikh \cite{Passare_Tsikh} claimed that the domain of absolute convergence $\mathcal{D}_{pq}$
of the series solution of eq.~\eqref{eq:algeq_pt}
is a complete Reinhardt domain 
whose boundary is (a segment of) the zero locus $\Delta_{pq}(a_0,a_1,\dots,[p],\dots,[q],\dots,a_n)=0$.
Specifically, for absolute convergence, they derived equations of the form
$\Delta_{pq}(\pm|a_0|,\dots,[p],\dots,[q],\dots,\pm|a_n|) = 0$.
See \cite[Thm.~3]{Passare_Tsikh} for a precise statement of their result.

\subsection{Application to cubics}
Passare and Tsikh employed their formalism to
display the domains of convergence for the series solutions of a cubic \cite[Sec.~5.2]{Passare_Tsikh}.
The general cubic equation with complex coefficients $(a_0,a_1,a_2,a_3)$ is
\bq
\label{eq:gencubic}
a_0 +a_1x +a_2x^2 +a_3x^3 = 0 \,.
\eq
There are six choices for $p$ and $q$, and the respective domains of convergence $\mathcal{D}_{pq}$ were given as follows
\cite[unnumbered before eq.~(17)]{Passare_Tsikh}
\begin{subequations}
\begin{align}
\label{eq:pt_cubic_d01}
\mathcal{D}_{01} &= \{ \Delta_{01}(|a_2|, -|a_3|) < 0 \}   \,,
\\
\label{eq:pt_cubic_d02}
\mathcal{D}_{02}^* &= \{ \Delta_{02}(|a_1|, |a_3|) > 0 \} \cap \{ \Delta_{02}(|a_1|, -|a_3|) < 0 \}  \,,
\\
\label{eq:pt_cubic_d03}
\mathcal{D}_{03} &= \{ \Delta_{03}(-|a_1|, -|a_2|) > 0 \}   \,,
\\
\label{eq:pt_cubic_d12}
\mathcal{D}_{12} &= \{ \Delta_{12}(|a_0|, -|a_3|) < 0 \} \cap \{ \Delta_{12}(-|a_0|, |a_3|) < 0 \}  \,,
\\
\label{eq:pt_cubic_d13}
\mathcal{D}_{13}^* &= \{ \Delta_{13}(|a_0|, |a_2|) > 0 \} \cap \{ \Delta_{13}(|a_0|, -|a_2|) < 0 \}  \,,
\\
\label{eq:pt_cubic_d23}
\mathcal{D}_{23}^* &= \{ \Delta_{23}(|a_0|, -|a_1|) < 0 \}   \,.
\end{align}
\end{subequations}
Three of the above six cases, marked with asterisks, are wrong.
The cases $\mathcal{D}_{02}$ and $\mathcal{D}_{13}$ contain fundamental errors, while $\mathcal{D}_{23}$ can be explained as a misprint.
I have attempted to resolve these issues privately with Tsikh,
but regrettably have not received a reply of scientific substance.
(Passare is deceased.)
We begin with the case $p=0$ and $q=2$. 
The relevant cubic equation is 
\bq
\label{eq:cubic_d02_contra}
1 +a_1x +x^2 +a_3x^3 = 0 \,.
\eq 
Setting $a_3=0$ in eq.~\eqref{eq:pt_cubic_d02} yields the self-contradictory conditions
\bq
\Delta_{02}(|a_1|, 0) > 0 \qquad \textrm{and} \qquad \Delta_{02}(|a_1|, 0) < 0 \,.
\eq
These conditions imply that for $a_3=0$, the series does not converge for any $a_1$, 
and in particular the origin $(a_1,a_3)=(0,0)$ is not in the domain of convergence, 
which is false.
For $a_3=0$, eq.~\eqref{eq:cubic_d02_contra} reduces to the quadratic $1 +a_1x +x^2 =0$
and the series solution converges for $4-|a_1|^2>0$ or $|a_1|<2$.
We now show that the error in eq.~\eqref{eq:pt_cubic_d02} 
is fundamental and cannot be explained as a misprint in \cite{Passare_Tsikh}.
\begin{itemize}
\item
First, the expressions for the discriminants are, for all four $\pm$ sign assignments $(\pm|a_1|,\pm|a_3|)$,
\begin{subequations}
\label{eq:delta02_cubic_pt}
\begin{align}
\Delta_{02}(|a_1|,|a_3|) &= 27|a_3|^2 +4|a_1|^3|a_3| +4 -18|a_1||a_3| -|a_1|^2 \,,
\\
\Delta_{02}(|a_1|,-|a_3|) &= 27|a_3|^2 -4|a_1|^3|a_3| +4 +18|a_1||a_3| -|a_1|^2 \,.
\\
\Delta_{02}(-|a_1|,|a_3|) &= \Delta_{02}(|a_1|,-|a_3|) \,,
\\
\Delta_{02}(-|a_1|,-|a_3|) &= \Delta_{02}(|a_1|,|a_3|) \,.
\end{align}
\end{subequations}
Hence there are only two independent expressions, 
viz.~$\Delta_{02}(|a_1|,|a_3|)$ and $\Delta_{02}(|a_1|,-|a_3|)$.
Hence the problem with eq.~\eqref{eq:pt_cubic_d02}
cannot be explained as a misprint in the assignment of $\pm$ signs for $\pm|a_1|$ and/or $\pm|a_3|$.

\item
Putting $a_1=a_3=0$ yields $\Delta_{02}(0,0) = 4$, i.e.~a positive number.
Let us therefore {\em tentatively} reverse the second inequality in eq.~\eqref{eq:pt_cubic_d02} as follows
\bq
\label{eq:pt_cubic_d02_contra_rev}
\mathcal{D}_{02} \ =?\ \{\Delta_{02}(|a_1|,|a_3|)>0\} \cap \{\Delta_{02}(|a_1|,-|a_3|)>0\} \,.
\eq
Putting $a_1=0$ yields
$\Delta_{02}(0,|a_3|) = \Delta_{02}(0,-|a_3|) = 27|a_3|^2 +4$,
i.e.~both expressions are equal and positive definite,
so eq.~\eqref{eq:pt_cubic_d02_contra_rev} is satisfied for all $a_3$. 
Hence, if eq.~\eqref{eq:pt_cubic_d02_contra_rev} is taken seriously,
it implies that for $a_1=0$,
the series solution of eq.~\eqref{eq:cubic_d02_contra}
converges absolutely for all $a_3$.
We know this is false. 
If $a_1=0$, then eq.~\eqref{eq:cubic_d02_contra} reduces to the following trinomial equation
$1 +x^2 +a_3x^3 = 0$.
We proved in Sec.~\ref{sec:tri}
that for such a situation the series solution converges absolutely for $|a_3|^2 \le 4/27$.
\end{itemize}

Hence there is no assignment of $\pm$ signs for $\pm|a_1|$ and/or $\pm|a_3|$,
nor any reversal of the inequalities in eq.~\eqref{eq:pt_cubic_d02},
which leads to a correct formula for the domain of convergence $\mathcal{D}_{02}$.
The error in eq.~\eqref{eq:pt_cubic_d02} cannot be explained as a misprint in \cite{Passare_Tsikh}.
A more careful treatment is therefore required, and will be presented in the next section.
We shall also deal with the other cases $\mathcal{D}_{13}$ and $\mathcal{D}_{23}$ in Sec.~\ref{sec:conv3}.

\setcounter{equation}{0}
\section{\label{sec:conv3} Algebraic equations: convergence of series II}
\subsection{Revised formalism}
We present a more careful analysis of the problem of the domain of absolute convergence 
of the series solution of an algebraic equation below.
To make the exposition self-contained, we begin from scratch,
although we shall attempt to minimize repetition of material already presented earlier in this paper.
The original polynomial is
\bq
\label{eq:orig_px}
\mathscr{P}(x) = a_0 +a_1 x + \cdots +a_n x^n \,.
\eq
For the purposes of determining domains of convergence,
we assume all the coefficients are nonzero in general.
The specialization to cases such as a trinomial is obvious.
We fix two integers $p$ and $q$ such that $0 \le p < q \le n$
and derive a transformed polynomial,
whose roots are proportional to those of $\mathscr{P}(x)$.
Employing (nonzero) constants $\lambda$ and $\mu$, where $x = \mu y$,
we obtain 
\bq
\mathscr{P}_{pq}(y) \equiv 
\lambda \mathscr{P}(\mu y) =
b_0 +b_1 y +\cdots +y^p +\cdots +y^q +\cdots +b_n y^n \,.
\eq
Here $b_j = a_j/(a_p^{1-\mu_j}a_q^{\mu_j})$,
where $\mu_j = (j-p)/(q-p)$.
Then $b_p=b_q=1$, by construction.
For brevity below, we define the tuples $\bm{a}=(a_0,\dots,a_n)$
and $\bm{b}=(b_0,\dots,[p],\dots,[q],\dots,b_n)$.
Then $\bm{a}$ and $\bm{b}$ contain respectively $n+1$ and $n-1$ components.
We solve for a root $y_m$ where $\mathscr{P}_{pq}(y_m)=0$.
Note that $y_m$ also depends on $p$ and $q$, but we omit this for brevity.
We express $y_m$ as a multivariate power series in the $n-1$ scaled coefficients $b_j$, 
where $j\in\mathscr{N}_{npq}$.
Recall $\mathscr{N}_{npq} = \{0,1,\dots,n\} \setminus \{p,q\}$.
We saw previously that this procedure yields $q-p$ roots,
so $m=0,\dots,q-p-1$.
We know the domain of absolute convergence 
of the resulting power series for $y_m$ includes 
a nonempty open neighborhood of the origin
$\bm{0}_{pq}$, where $b_j=0$ for all $j\in\mathscr{N}_{npq}$.
We also know the domain of absolute convergence 
is a complete Reinhardt domain and depends only on the amplitudes $|b_j|$,
i.e.~the domain is the same for all the $q-p$ roots $y_m$.

Next note that just as the original polynomial $\mathscr{P}(x)$ 
can always be transformed to $\mathscr{P}_{pq}(y)$, 
the same procedure also transforms the discriminant
$\Delta(\bm{a})$ of $\mathscr{P}(x)$ 
to the scaled discriminant $\Delta_{pq}(\bm{b})$ of $\mathscr{P}_{pq}(y)$.
Then $\Delta_{pq}(\bm{b}) = \lambda^{2(n-1)}\mu^{n(n-1)}\Delta(\bm{a})$.
Thus far, the argument is correct.
It is, however, {\em false} to conclude 
that the boundary of the domain of absolute convergence 
is determined by the scaled discriminants given by
$\Delta_{pq}(\pm|b_0|,\dots,[p],\dots,[q],\dots,\pm|b_n|)$,
specifically, by solving for the hypersurfaces given by
$\Delta_{pq}(\pm|b_0|,\dots,[p],\dots,[q],\dots,\pm|b_n|)=0$.
For example, we saw above that this led to erroneous results 
for the domains of convergence for the series solutions of a cubic.

The weak point is that there are additional discriminants, 
which are also required to determine the boundary of the domain of absolute convergence.
To see this, let us review the key steps.
We employ fresh notation to avoid confusion with the above symbols.
For brevity below, define an $(n-1)$-tuple of $\pm$ signs 
\bq
\label{eq:sigma_tuple}
\bm{\sigma} = (\sigma_0,\dots,[p],\dots,[q],\dots,\sigma_n) \,.
\eq
Here $\sigma_j=\pm1$ for $j\in\mathscr{N}_{npq}$.
The dependence of $\bm{\sigma}$ and $\sigma_j$ on $p$ and $q$ is taken as understood.
We also define the set $\Sigma_{pq}$ of all the distinct tuples $\bm{\sigma}$.
Then $\Sigma_{pq}$ has cardinality $2^{n-1}$.
The power series solutions for the roots of the following $2^{n+1}$ 
algebraic equations all have the same domain of absolute convergence
\bq
\sigma_0|b_0| +\sigma_1|b_1|y + \cdots \pm y^p \cdots \pm y^q + \cdots +\sigma_n |b_n| y^n = 0 \,.
\eq
The coefficient of $y^j$ is permitted to be $\pm|b_j|$ only.
We can always divide through by $-1$, if necessary,
so that the coefficient of $y^p$ is unity.
This yields $2^n$ distinct equations.
The discriminant of the associated polynomial is
$\Delta(\sigma_0|b_0|,\cdots,1,\cdots,\pm1,\cdots,\sigma_n|b_n|)$.
Let us now define the following two families (or sets) of discriminants.
We employ the symbol $\Psi$ to avoid confusion with $\Delta_{pq}$ above.
Then, with $1$ in the $p^{th}$ slot and $\pm1$ in the $q^{th}$ slot, we define
\begin{subequations}
\begin{align}
\Psi^+_{pq}(\bm{b},\bm{\sigma}) &= 
\Delta(\sigma_0|b_0|,\cdots,1,\cdots,\phantom{-}1,\cdots,\sigma_n|b_n|) \,,
\\
\Psi^-_{pq}(\bm{b},\bm{\sigma}) &= 
\Delta(\sigma_0|b_0|,\cdots,1,\cdots,-1,\cdots,\sigma_n|b_n|) \,,
\\
\Psi^+_{pq}(\bm{b}) &= \{
\Psi^+_{pq}(\bm{b},\bm{\sigma})\,\vert\, \bm{\sigma}\in\Sigma_{pq} \} \,,
\\
\Psi^-_{pq}(\bm{b}) &= \{
\Psi^-_{pq}(\bm{b},\bm{\sigma})\,\vert\, \bm{\sigma}\in\Sigma_{pq} \} \,.
\end{align}
\end{subequations}
Each set has at most $2^{n-1}$ distinct elements.
The following lemma shows that the family $\Psi_{pq}^-(\bm{b})$ is nontrivial.

\begin{lemma}
If $q-p$ is odd, the sets $\Psi^+_{pq}(\bm{b})$ and $\Psi^-_{pq}(\bm{b})$ are identical.
If $q-p$ is even, the sets $\Psi^+_{pq}(\bm{b})$ and $\Psi^-_{pq}(\bm{b})$ are disjoint.
\end{lemma}
\begin{proof}
Introduce two additional tuples
$\bm{\sigma}^\prime$ where $\sigma_j^\prime = (-1)^j\sigma_j$
and $-\bm{\sigma}^\prime$ 
where obviously the components are $-(-1)^j\sigma_j$.
Then consider the polynomials
\bq
\label{eq:algeq_proof_p0sig}
P_\pm(y,\bm{\sigma}) = \sigma_0|b_0| +\sigma_1|b_1| y 
+\cdots +y^p +\cdots \pm y^q +\cdots +\sigma_n|b_n|y^n \,.
\eq
The only permitted transformations of $P_+$ and $P_-$ are 
to reverse the sign of $y$ and/or to multiply $P_\pm$ by $-1$,
because the coefficient of $y^j$ must be $\pm|b_j|$ only.
By construction, the discriminant of 
$P_+(y,\bm{\sigma})$ is an element of $\Psi^+_{pq}(\bm{b})$
and that of $P_-(y,\bm{\sigma})$ is an element of $\Psi^-_{pq}(\bm{b})$.
First suppose $q-p$ is odd.
If $p$ is even and $q$ is odd, then
\bq
\begin{split}
P_\pm(-y,\bm{\sigma}) &= 
\sigma_0|b_0| -\sigma_1|b_1| y +\cdots +y^p +\cdots \mp y^q +\cdots +(-1)^n\sigma_n|b_n|y^n 
\\
&= P_\mp(y,\bm{\sigma}^\prime) \,.
\end{split}
\eq
If $p$ is odd and $q$ is even, then
\bq
\begin{split}
-P_\pm(-y,\bm{\sigma}) &= 
-\sigma_0|b_0| +\sigma_1|b_1| y +\cdots +y^p +\cdots \mp y^q +\cdots +(-1)^{n+1}\sigma_n|b_n|y^n 
\\
&= P_\mp(y,-\bm{\sigma}^\prime) \,.
\end{split}
\eq
Cycling through all values of $\bm{\sigma}$ shows that the sets 
$\Psi^+_{pq}(\bm{b})$ and $\Psi^-_{pq}(\bm{b})$ are identical, if $q-p$ is odd.
Now suppose $q-p$ is even.
Suppose $p$ and $q$ are both even.
Then
\bq
\begin{split}
P_\pm(-y,\bm{\sigma}) &= 
\sigma_0|b_0| -\sigma_1|b_1| y +\cdots +y^p +\cdots \pm y^q +\cdots +(-1)^n\sigma_n|b_n|y^n 
\\
&= P_\pm(y,\bm{\sigma}^\prime) \,.
\end{split}
\eq
The discriminant of $P_+(-y,\bm{\sigma})$ is an element of $\Psi^+_{pq}(\bm{b})$.
The other transformations $-P_+(y,\bm{\sigma})$ and $-P_+(-y,\bm{\sigma})$
also fail to yield a polynomial with a discriminant which is an element of $\Psi^-_{pq}(\bm{b})$.
Similarly the discriminant of $\pm P_-(\pm y,\bm{\sigma})$ is always an element of $\Psi^-_{pq}(\bm{b})$.
Next suppose $p$ and $q$ are both odd.
Then
\bq
\begin{split}
-P_\pm(-y,\bm{\sigma}) &= 
-\sigma_0|b_0| +\sigma_1|b_1| y +\cdots +y^p +\cdots \pm y^q +\cdots +(-1)^{n+1}\sigma_n|b_n|y^n 
\\
&= P_\pm(y,-\bm{\sigma}^\prime) \,.
\end{split}
\eq
As was the case when $p$ and $q$ were both even,
the discriminant of $\pm P_+(\pm y,\bm{\sigma})$ is always an element of $\Psi^+_{pq}(\bm{b})$
and the discriminant of $\pm P_-(\pm y,\bm{\sigma})$ is always an element of $\Psi^-_{pq}(\bm{b})$.
Hence if $q-p$ is even, 
the sets $\Psi^+_{pq}(\bm{b})$ and $\Psi^-_{pq}(\bm{b})$ are disjoint.
\end{proof}
Armed with this additional information, we return to 
eq.~\eqref{eq:cubic_d02_contra}
and the case $p=0$ and $q=2$.
There are {\em four} distinct discriminants
which can contribute to the boundary of the domain of convergence, viz.
\begin{subequations}
\begin{align}
\Psi^+_{02}(|a_1|,|a_3|) &= 27|a_3|^2 +4|a_1|^3|a_3| +4 -18|a_1||a_3| -|a_1|^2 \,,
\\
\Psi^+_{02}(|a_1|,-|a_3|) &= 27|a_3|^2 -4|a_1|^3|a_3| +4 +18|a_1||a_3| -|a_1|^2 \,,
\\
\Psi^-_{02}(|a_1|,|a_3|) &= 27|a_3|^2 +4|a_1|^3|a_3| -4 +18|a_1||a_3| -|a_1|^2 \,,
\\
\Psi^-_{02}(|a_1|,-|a_3|) &= 27|a_3|^2 -4|a_1|^3|a_3| -4 -18|a_1||a_3| -|a_1|^2 \,.
\end{align}
\end{subequations}
The expressions for
$\Psi^+_{02}(|a_1|,\pm|a_3|)$
are the same as for
$\Delta_{02}(|a_1|,\pm|a_3|)$ in
eq.~\eqref{eq:delta02_cubic_pt}.
Observe that $\Psi^+_{02}(0,0)=4$ and $\Psi^-_{02}(0,0)=-4$.
If we set $|a_1|=0$ then
$\Psi^-_{02}(0,\pm|a_3|) = 27|a_3|^2 -4$,
which yields the correct upper bound for $|a_3|$.
If we set $|a_3|=0$ then
$\Psi^+_{02}(|a_1|,0) = 4 -|a_1|^2$,
which yields the correct upper bound for $|a_1|$.
The correct answer requires {\em both} $\Psi^+_{02}$ and $\Psi^-_{02}$.
Some further algebra yields the correct expression for the domain of convergence to be
$\mathcal{D}_{02} = \{ \Psi^+_{02}(|a_1|,|a_3|) \ge 0 \} \cap
\{ \Psi^-_{02}(|a_1|,|a_3|) \le 0 \}$.
Note that the series converges on the boundary of its domain of convergence.

The corrected expressions for the domains of absolute convergence
for the series solutions for the roots of a cubic as follows.
For clarity, we distinguish between the coefficients $a_0,\dots,a_3$
of the original cubic in eq.~\eqref{eq:gencubic}
and the coefficients $b_j$ in the scaled polynomial
$\mathscr{P}_{pq}(y)$. 
Recall that `$b_j$' depends also on $p$ and $q$ but this is considered to be understood.
The domains of convergence $\mathcal{D}_{pq}$ are given by
\begin{subequations}
\label{eq:pt_cubic_all_me}
\begin{align}
\label{eq:pt_cubic_d01_me}
\mathcal{D}_{01} &= \{ \Psi^+_{01}(|b_2|, -|b_3|) \le 0 \} \,,
\\
\label{eq:pt_cubic_d02_me}
\mathcal{D}_{02} &= \{ \Psi^+_{02}(|b_1|, |b_3|) \ge 0 \} \cap \{ \Psi^-_{02}(|b_1|, |b_3|) \le 0 \}  \,,
\\
\label{eq:pt_cubic_d03_me}
\mathcal{D}_{03} &= \{ \Psi^+_{03}(-|b_1|, -|b_2|) \ge 0 \} \,,
\\
\label{eq:pt_cubic_d12_me}
\mathcal{D}_{12} &= \{ \Psi^+_{12}(|b_0|, -|b_3|) \le 0 \} \cap \{ \Psi^+_{12}(-|b_0|, |b_3|) \le 0 \} \,,
\\
\label{eq:pt_cubic_d13_me}
\mathcal{D}_{13} &= \{ \Psi^+_{13}(|b_0|, |b_2|) \ge 0 \} \cap \{ \Psi^-_{13}(|b_0|, |b_2|) \le 0 \} \,,
\\
\label{eq:pt_cubic_d23_me}
\mathcal{D}_{23} &= \{ \Psi^+_{23}(-|b_0|, |b_1|) \le 0 \} \,.
\end{align}
\end{subequations}
The series converge on the boundaries of their respective domains of convergence.
For the case $\mathcal{D}_{23}$, note that $q-p=1$ is odd,
hence $\Psi^+_{pq}(\bm{b})$ and $\Psi^-_{pq}(\bm{b})$ are identical,
so the solution is expressed using purely $\Psi^+_{23}(\bm{b})$.
We can therefore consider the expression in
eq.~\eqref{eq:pt_cubic_d23} to be a misprint.
For the cases $\mathcal{D}_{01}$ and $\mathcal{D}_{13}$, where $q-p=2$ is even,
the need for $\Psi^-_{pq}(\bm{b})$ is essential.

There is an additional caveat,
which is that the formula for the domain of convergence
cannot always be expressed using purely inequalities.
Consider the quartic equation with no term in $x^3$
\bq
\label{eq:quartic_caveat}
a_0 +a_1 x +x^2 +x^4 = 0 \,.
\eq
We choose $p=2$ and $q=4$ and we have set $a_2=a_4=1$ so that 
the scaled coefficients are simply $b_j=a_j$.
The boundary of the domain of convergence in this case is 
determined solely by $\Psi^-_{24}(-|a_0|,|a_1|)$.
The derivation is omitted.
Then
\bq
\Psi_{24}^-(-|a_0|,|a_1|) = 16|a_0|(1-4|a_0|)^2 +4(1-36|a_0|)|a_1|^2 - 27|a_1|^4 \,.
\eq
\begin{itemize}
\item
The discriminant vanishes at the origin: $\Psi_{24}^-(0,0) = 0$.
The significance of this will be discussed below.
For now we seek nonzero solutions of the equation
$\Psi_{24}^-(-|a_0|,|a_1|) =0$.

\item
Put $a_1=0$, then $\Psi^-_{24}(-|a_0|,0) = 16|a_0|(1-4|a_0|)^2$.
This is (proportional to) a perfect square, which equals zero at $|a_0|=\frac14$.
The necessary upper bound on $|a_0|$ for this problem is known to be $|a_0|\le \frac14$.

\item
Next put $a_0=0$, then $\Psi^-_{24}(0,|a_1|) = (4 -27|a_1|^2)|a_1|^2$.
For nonzero $a_1$, this vanishes at $|a_1|=\sqrt{4/27}$.
The necessary upper bound on $|a_1|$ for this problem is known to be $|a_1|\le \sqrt{4/27}$.

\item
Next let us put $|a_0|=a$ and $|a_1|=\frac12 a$, where $a\in\mathbb{R}_+$.
The graph of $\Psi^-_{24}(-a,\frac12 a)$ is plotted against $a$ in Fig.~\ref{fig:ray}.
As the value of $a$ increases from zero,
initially $\Psi^-_{24}(0,0)=0$ for $a=0$,
then the value of $\Psi^-_{24}(-a,\frac12 a)$ is positive, reaches a maximum, 
then it changes sign and becomes negative, 
reaches a minimum and then becomes positive again and increases to $+\infty$ thereafter.
The value of $\Psi^-_{24}(-a,\frac12 a)$ is thus {\em not} monotonic in $a$.

\item
All of the above facts demonstrate that an unconditional inequality $\Psi^-_{24}(-|a_0|,|a_1|)\ge0$
is insufficient to determine the domain of convergence.
First, the discriminant vanishes at the origin.
We need to exclude the origin as a solution, because we know the domain of convergence has positive measure.
Even after doing so, we require an additional stipulation
``the domain of convergence includes only the component 
which satisfies the inequality and is connected to the origin.''

\item
It is implicit in \cite[Thm.~3]{Passare_Tsikh}
that the domain of convergence includes only the component
connected to the origin.
What is {\em not} clear is that the formula for the domain of convergence 
{\em cannot} always be expressed using only unconditional inequalities on the values of the discriminants.
The stipulation ``the component connected to the origin'' is necessary.

\item
We remark in passing that for this problem, 
the domain of convergence is determined solely by a discriminant of the form $\Psi^-_{pq}$.
A discriminant of the form $\Psi^+_{pq}$,
i.e.~$\Delta_{pq}$ in the formalism in \cite{Passare_Tsikh}, does not appear.
\end{itemize}
One source of the difficulty is that if $a_1=0$ in eq.~\eqref{eq:quartic_caveat},
it becomes an algebraic equation in $x^2$, viz.~$a_0 +x^2 +x^4 = 0$.
For a polynomial of degree $n$ such as in eq.~\eqref{eq:orig_px}, 
the discriminant of $\mathscr{P}(x^m)$, for a positive integer $m$, is given by
\bq
\Delta(\mathscr{P}(x^m)) = (-1)^{nm(m-1)/2} m^{mn} (a_0a_n)^{m-1} \bigl(\Delta(\mathscr{P}(x))\bigr)^m \,.
\eq
Hence if $m$ is even, the discriminant $\Delta(\mathscr{P}(x^m))$
will not change sign as the (absolute values of the) coefficients are varied.
Hence, in general, an unconditional inequality on the value of the discriminant(s) 
is insufficient to determine the domain of convergence.
This feature will occur generically
(or at least, cannot be ruled out)
for a quartic and algebraic equations of all higher degrees,
for example if the coefficients of all the odd powers of $x$ are set to zero.

\subsection{General formula}
We have seen that the formalism in \cite{Passare_Tsikh}
must be augmented by the inclusion of an extra set of discriminants. 
Although this yields the correct result for a cubic,
as in eq.~\eqref{eq:pt_cubic_all_me},
the procedure in \cite{Passare_Tsikh}
becomes tedious for polynomials of high degree,
and we have seen that it is prone to error.
We seek a procedure that yields a single `general formula' valid for arbitrary $n$,
which is simpler to state and to compute, for practical work.
This can be accomplished via the use of hyperplanes and foliations, as will be explained below.
(N.B.~the word `single' was employed informally above; we shall require at least two formulas.)

Still speaking informally, given an algebraic equation of degree $n$ with a coefficient tuple $\bm{a}$
and a choice for $p$ and $q$, hence a scaled tuple $\bm{b}$,
the equations $\Psi^+_{pq}(\bm{b},\bm{\sigma})=0$ and $\Psi^-_{pq}(\bm{b},\bm{\sigma})=0$,
taken over all $\bm{\sigma}\in\Sigma_{pq}$, 
specify a set of hyperplanes in the amplitudes $|b_j|$.
The domain of convergence in $\mathbb{R}_+^{n-1}$ 
is given by the set of hyperplanes closest to the origin and 
which together bound a region which is connected to the origin $\bm{0}_{pq}$.
The domain of convergence for $\bm{b}\in\mathbb{C}^{n-1}$ is the inverse image of the above domain in $\mathbb{R}_+^{n-1}$.
The domain of absolute convergence is clearly unique.
If there were two or more sets of such hyperplanes,
the full domain of absolute convergence would simply be the union of the individual domains.
However, one reason the above discussion is informal is that we saw that the discriminant can vanish at the origin.
Hence to write an equation such as `$\Psi^\pm_{pq}(\bm{b},\bm{\sigma})=0$' is not precise enough for our needs.

We now sharpen the above ideas.
Clearly, the domain of absolute convergence is determined solely by the amplitudes $|b_j|$, $j\in\mathscr{N}_{npq}$.
We previously denoted the doman of absolute convergence by $\mathcal{D}$ 
and introduced its image $\textrm{Log}(\mathcal{D})$.
Here we define a second image via an `amplitude map' $\mathbb{C}^k \to \mathbb{R}_+^k$ 
where $z_j \mapsto |z_j|$ for $j=1,\dots,k$:
\bq
\mathscr{D} = \{ (|z_1|,\dots,|z_k|) \,\vert\, \bm{z} \in\mathcal{D} \} \,.
\eq
From the previous discussion of polydiscs and a complete Reinhardt domain, 
$\bm{z} \in\mathcal{D}$ if and only if $(|z_1|,\dots,|z_k|) \in \mathscr{D}$.
Clearly also $\textrm{Log}(\mathscr{D}) = \textrm{Log}(\mathcal{D})$.
Our interest is the case of an algebraic equation of degree $n$,
so $k=n-1$ and $\bm{z}=\bm{b}$.
We shall derive a formula to determine the domain $\mathscr{D}$ in this case.
Obviously $\bm{0}_{pq}\in\mathscr{D}$.
Recall that for an algebraic equation of degree $n$ and fixed $p,q$,
then $\mu_j = (j-p)/(q-p)$.
From eq.~\eqref{eq:conv_nec_j_vertex}, let us define, for algebraic equations,
\bq 
\hat{b}_j = \frac{1}{|\mu_j|^{\mu_j} |1-\mu_j|^{1-\mu_j}} \,.
\eq
Recall one must have $|b_j| \le \hat{b}_j$ for $j\in\mathscr{N}_{npq}$.
It follows that $\mathscr{D} \subset \mathscr{\hat{D}}$ where the `hypercuboid' is
\bq
\mathscr{\hat{D}} = \biggl\{ (|b_1|,\dots,[p],\dots,[q],\dots,|b_n|) \;\biggl\vert\; |b_j| \le \hat{b}_j,\, j\in\mathscr{N}_{npq} \biggr\} \,.
\eq
The following $n-1$ vertices of the hypercuboid lie in the domain of convergence,
viz.~$(\hat{b}_0,0,\dots,0)$, $(0,\hat{b}_1,0,\dots,0)$, \dots, $(0,\dots,\hat{b}_n)$.
We also know that $\mathscr{D}$ has positive measure and $\mathscr{D} \supset \mathscr{\check{D}}$, where 
\bq
\mathscr{\check{D}} = \biggl\{ (|b_1|,\dots,[p],\dots,[q],\dots,|b_n|) \,\biggl\vert\, 
\sum_{j\in\mathscr{N}_{npq}} |b_j| \le \frac{1}{|\mu_*|^{\mu_*} |1-\mu_*|^{1-\mu_*}} \biggr\} \,.
\eq
Recall eq.~\eqref{eq:radconv_mu*} and the definition of $\mu_*$.
We require the following lemma.
\begin{lemma}
\label{lem:psi_origin}
At the origin, exactly one of the two following mutually exclusive possibilities is true:
(i) $\Psi_{pq}^+(\bm{0}_{pq}) = \Psi_{pq}^-(\bm{0}_{pq}) = 0$, or 
(ii) $\Psi_{pq}^+(\bm{0}_{pq}) = \pm \Psi_{pq}^-(\bm{0}_{pq}) \ne 0$.
(Explicit mention of $\bm{\sigma}$ has been omitted since it is irrelevant at the origin.)
\end{lemma}
\begin{proof}
The values of $\Psi_{pq}^+(\bm{0}_{pq})$ and $\Psi_{pq}^-(\bm{0}_{pq})$ can be nonzero if and only if 
the unscaled discriminant $\Delta(a_0,\dots,a_n)$ contains a term of the form $c_{pq} a_p^\alpha a_q^\beta$
for some coefficient $c_{pq}$ and exponents $\alpha$ and $\beta$.
From the homogeneity properties of the discriminant, 
we must have $\alpha+\beta=2n-2$ and $p\alpha +q\beta = n(n-1)$.
Hence, given $p$ and $q$, then $\alpha$ and $\beta$ are uniquely determined, 
so there is at most one monomial of this form in the discriminant.
We say `at most one' because $\alpha = (n-1)(2q-n)/(q-p)$ and $\beta = (n-1)(n-2p)/(q-p)$
and these values may not be integers.
Even if they are integers, the relevant monomial may not appear in the discriminant.
After scaling, this term (if it exists) maps to $c_{pq}b_p^\alpha b_q^\beta$.
Then at the origin $\bm{)}_{pq}$ we obtain
$\Psi_{pq}^+(\bm{0}_{pq}) = c_{pq}$ and $\Psi_{pq}^-(\bm{0}_{pq}) = (-1)^\beta c_{pq}$.
Hence either (i) holds, if $c_{pq}=0$, or else (ii) holds, with $\Psi_{pq}^+(\bm{0}_{pq}) = c_{pq} = \pm \Psi_{pq}^-(\bm{0}_{pq})$.
\end{proof}

The two cases (i) and (ii) in Lemma \ref{lem:psi_origin} require separate treatments.
In practice, it is convenient to introduce the notion of a `reduced' discriminant.
If $\Psi_{pq}^+(\bm{b},\bm{\sigma})$ contains a common factor, we divide out that common factor.
A common factor in a discriminant clearly cannot contribute to the determination of the domain boundary
in an equation such as $\Psi_{pq}^+(\bm{b},\bm{\sigma})=0$.
We denote the reduced discriminant by $\tilde{\Psi}_{pq}^+(\bm{b},\bm{\sigma})$.
By definition, it does not vanish if any single component $b_j$ in $\bm{b}$ is set to zero.
Next $\Psi_{pq}^-(\bm{b},\bm{\sigma})$ clearly contains the same common factor as $\Psi_{pq}^+(\bm{b},\bm{\sigma})$
because flipping $\pm$ signs in the coefficients of the polynomial does not affect common factors in the discriminant.
Hence by an obvious analogy we define the reduced discriminant $\tilde{\Psi}_{pq}^-(\bm{b},\bm{\sigma})$.
We work with $\tilde{\Psi}_{pq}^+(\bm{b},\bm{\sigma})$ and $\tilde{\Psi}_{pq}^-(\bm{b},\bm{\sigma})$ below.
Note that Lemma \ref{lem:psi_origin} holds true also for the reduced discriminants.

The next key idea is that of {\em foliation}.
For fixed tuples $\bm{b}$ and $\bm{\sigma}$, the level sets of
$\tilde{\Psi}_{pq}^+(\bm{b},\bm{\sigma})$ foliate the parameter space $\mathbb{R}_+^{n-1}$.
The level sets of $\tilde{\Psi}_{pq}^-(\bm{b},\bm{\sigma})$ also foliate $\mathbb{R}_+^{n-1}$.
Hence both families of level sets foliate the domain $\mathscr{\hat{D}}$. 
For our purposes, the foliation is a mapping $\mathbb{R}_+^{n-1} \to \mathbb{R}$,
because both $\tilde{\Psi}_{pq}^+(\bm{b},\bm{\sigma})$ and $\tilde{\Psi}_{pq}^-(\bm{b},\bm{\sigma})$ are real valued.

\subsection{$\tilde{\Psi}_{pq}^\pm(\bm{b},\sigma)$ nonzero at origin}
We begin with the simpler case (ii) in Lemma \ref{lem:psi_origin},
where $\tilde{\Psi}_{pq}^+(\bm{b},\sigma)$ and $\tilde{\Psi}_{pq}^-(\bm{b},\sigma)$ are nonzero at the origin.
First fix the values of $p$ and $q$.
Then for any tuple $\bm{\sigma}$, 
for any $\bm{b}^\prime$ whose image is in $\mathscr{\check{D}}$,
both $\tilde{\Psi}_{pq}^+(\bm{0}_{pq},\bm{\sigma}) \tilde{\Psi}_{pq}^+(\bm{b}^\prime,\bm{\sigma}) > 0$ 
and $\tilde{\Psi}_{pq}^-(\bm{0}_{pq},\bm{\sigma}) \tilde{\Psi}_{pq}^-(\bm{b}^\prime,\bm{\sigma}) > 0$.
To determine the boundary of the domain of convergence, 
we solve for $\bm{b}_*$ where
$\tilde{\Psi}_{pq}^+(\bm{b}_*,\bm{\sigma}) = 0$ or $\tilde{\Psi}_{pq}^-(\bm{b}_*,\bm{\sigma}) = 0$
for any $\sigma \in \Sigma_{pq}$.
This can be encapsulated in a single formula
\bq
\label{eq:bdy_psi0_nonzero}
\prod_{\bm{\sigma}\in\Sigma_{pq}} \tilde{\Psi}_{pq}^+(\bm{b}_*,\bm{\sigma}) \tilde{\Psi}_{pq}^-(\bm{b}_*,\bm{\sigma}) = 0 \,.
\eq
The domain of convergence $\mathscr{D}$ is the set connected to the origin, 
bounded by the hyperplanes which satisfy eq.~\eqref{eq:bdy_psi0_nonzero}.
Although technically there are $2^n$ discriminants in the product in
eq.~\eqref{eq:bdy_psi0_nonzero},
in practice many of them are identical and the number of {\em distinct} discriminants is much fewer.
However, I do not have a definitive estimate of the number of distinct discriminants.
If $q-p$ is odd, we need consider only $\tilde{\Psi}_{pq}^+$ and we can simplify eq.~\eqref{eq:bdy_psi0_nonzero} to
\bq
\label{eq:bdy_psi0_nonzero_odd}
\prod_{\bm{\sigma}\in\Sigma_{pq}} \tilde{\Psi}_{pq}^+(\bm{b}_*,\bm{\sigma}) = 0 \,.
\eq
As an illustrative example, consider the quartic equation
$a_0 +x +x^2 +a_4 x^4 = 0$.
We choose $p=1$ and $q=2$ and we have set $a_1=a_2=1$ so that 
the scaled coefficients are simply $b_j=a_j$.
Because $q-p=1$ is odd, we require only $\Psi_{pq}^+$.
The discriminant has a common factor of $|a_4|$:
\bq
\Psi_{12}^+(|a_0|,|a_4|) = |a_4| \Bigl(256 |a_0|^3 |a_4|^2 -128 |a_0|^2 |a_4| +144 |a_0| |a_4| +16 |a_0| -27 |a_4| -4\Bigr) \,.
\eq
We divide out the common factor $|a_4|$ and obtain the reduced discriminants
\begin{subequations}
\begin{align}
\tilde{\Psi}_{12}^+(|a_0|,|a_4|) &= \phantom{-}256 |a_0|^3 |a_4|^2 -128 |a_0|^2 |a_4| +144 |a_0| |a_4| +16 |a_0| -27 |a_4| -4 \,,
\\
\tilde{\Psi}_{12}^+(|a_0|,-|a_4|) &= \phantom{-}256 |a_0|^3 |a_4|^2 +128 |a_0|^2 |a_4| -144 |a_0| |a_4| +16 |a_0| +27 |a_4| -4 \,,
\\
\tilde{\Psi}_{12}^+(-|a_0|,|a_4|) &= -256 |a_0|^3 |a_4|^2 -128 |a_0|^2 |a_4| -144 |a_0| |a_4| -16 |a_0| -27 |a_4| -4 \,,
\\
\tilde{\Psi}_{12}^+(-|a_0|,-|a_4|) &= -256 |a_0|^3 |a_4|^2 +128 |a_0|^2 |a_4| +144 |a_0| |a_4| -16 |a_0| +27 |a_4| -4 \,.
\end{align}
\end{subequations}
The reduced discriminants all equal $-4$ at the origin.
The necessary bounds for convergence yield $|a_0|\le\frac14$ and $|a_4|\le 4/27$.
Note that the above expressions are quadratics in $|a_4|$.
Thus to solve for $\Psi^+_{12}(\pm|a_0|,\pm|a_4|)=0$, 
we fix a value of $|a_0|$ and solve the resulting quadratic in $|a_4|$.
Note that this procedure will not always yield a real solution for $|a_4|$;
the discriminants which fail to do so do not contribute to the boundary of the domain of convergence.
The discriminant $\tilde{\Psi}_{12}^+(-|a_0|,|a_4|)$ is such a case.
Setting the other discriminants to zero yields valid hyperplanes.
The resulting curves in the $(|a_0|,|a_4|)$ parameter space are displayed in Fig.~\ref{fig:foliation_d12}, for
$\Psi^+_{12}(|a_0|,|a_4|)=0$ (dashed), 
$\Psi^+_{12}(|a_0|,-|a_4|)=0$ (dotdash) and
$\Psi^+_{12}(-|a_0|,-|a_4|)=0$ (solid).
The shaded area indicates the domain $\mathscr{D}_{12}$, 
which is determined by the two hyperplanes given by the level sets
$\Psi^+_{12}(|a_0|,|a_4|)=0$ and $\Psi^+_{12}(-|a_0|,-|a_4|)=0$.
The level set $\Psi^+_{12}(|a_0|,-|a_4|)=0$ does not contribute.
The domain of convergence is therefore
\bq
\mathcal{D}_{12} = \{ \Psi^+_{12}(|a_0|,|a_4|) \le 0 \} \cap \{ \Psi^+_{12}(-|a_0|,-|a_4|) \le 0 \} \,.
\eq
Recall that technically, the domain $\mathcal{D}_{12}$ is the component which satisfies the above conditions
and is connected to the origin.
Observe from the curvature of the upper boundary in Fig.~\ref{fig:foliation_d12}, 
i.e.~the level set $\Psi^+_{12}(-|a_0|,-|a_4|)=0$, that the domain of convergence is {\em not} convex.
A complete Reinhardt domain is logarithmically convex, but is not necessarily convex.

\subsection{$\tilde{\Psi}_{pq}^\pm(\bm{b},\sigma)$ vanishes at origin}
The case (i) in Lemma \ref{lem:psi_origin} is more difficult.
Now $\tilde{\Psi}_{pq}^+(\bm{b},\sigma)$ and $\tilde{\Psi}_{pq}^-(\bm{b},\sigma)$ vanish at the origin, 
hence solving for $\tilde{\Psi}_{pq}^+(\bm{b},\sigma)=0$ or $\tilde{\Psi}_{pq}^-(\bm{b},\sigma)=0$ 
yields the origin as an unwanted solution.
Recall the example of the quartic eq.~\eqref{eq:quartic_caveat}.
Hence we must proceed more carefully.
As always, we first fix the values of $p$ and $q$.
Next, fix a tuple $\bm{\sigma}$.
Then the discriminants will exhibit one of three mutually exclusive properties: 
either $\tilde{\Psi}_{pq}^+(\bm{b},\sigma)$ 
has a local maximum, or a local minimum, or a saddle point at the origin.
The same is true for $\tilde{\Psi}_{pq}^-(\bm{b},\sigma)$.
We can also say `a local extremum or a saddle point' at the origin.
The concept of `local extremum' must be understood carefully,
because it is really a constrained extremization.
It is simplest to illustrate with an example.
Consider a cubic equation $1 + x + b_2x^2 + b_3x^3 = 0$ with $p=0$ and $q=1$. 
Since $q-p=1$ is odd, it suffices to treat $\tilde{\Psi}_{pq}^+(\bm{b},\sigma)$ only.
The expressions for the discriminants in this case are  
(note that they all vanish at the origin)
\begin{subequations}
\begin{align}
\tilde{\Psi}_{01}^+(|b_2|,|b_3|) &= 27|b_3|^2 +4|b_3| +4|b_2|^3 -18|b_2||b_3| -|b_2|^2 \,,
\\
\tilde{\Psi}_{01}^+(|b_2|,-|b_3|) &= 27|b_3|^2 -4|b_3| +4|b_2|^3 +18|b_2||b_3| -|b_2|^2 \,,
\\
\tilde{\Psi}_{01}^+(-|b_2|,|b_3|) &= 27|b_3|^2 +4|b_3| -4|b_2|^3 +18|b_2||b_3| -|b_2|^2 \,,
\\
\tilde{\Psi}_{01}^+(-|b_2|,-|b_3|) &= 27|b_3|^2 -4|b_3| -4|b_2|^3 -18|b_2||b_3| -|b_2|^2 \,.
\end{align}
\end{subequations}
\begin{itemize}
\item
Then $\tilde{\Psi}_{01}^+(|b_2|,-|b_3|)$ has a local maximum at the origin.
Put $|b_2|=0$, then $\tilde{\Psi}_{01}^+(0,-|b_3|) \simeq -4|b_3|$ for sufficiently small $|b_3|$.
This is negative definite because $|b_3|>0$ only, for $b_3\ne0$.
However, the partial derivative $\partial \tilde{\Psi}_{01}^+/\partial |b_3|$ does not vanish at $|b_1|=0$.
Next put $|b_3|=0$, then $\tilde{\Psi}_{01}^+(|b_2|,0) \simeq -|b_2|^2$ for sufficiently small $|b_2|$.
This is also negative definite for $b_2\ne0$.
One can show that $\tilde{\Psi}_{01}^+(|b_2|,-|b_3|) < 0$ for all sufficiently small $|b_2|>0$ and $|b_3|>0$.
Hence for our purposes, a `local maximum' is a {\em constrained} local maximum.
Similarly, the concept of `local minimum' is a constrained local minimum.
\item
Similarly $\tilde{\Psi}_{01}^+(-|b_2|,-|b_3|)$ also has a local maximum at the origin.
\item
However $\tilde{\Psi}_{01}^+(|b_2|,|b_3|)$ and $\tilde{\Psi}_{01}^+(-|b_2|,|b_3|)$ both have saddle points at the origin.
Put $|b_2|=0$, then $\tilde{\Psi}_{01}^+(0,|b_3|) \simeq 4|b_3|$ for sufficiently small $|b_3|$, and is positive for $|b_3|>0$.
Next put $|b_3|=0$, 
then $\tilde{\Psi}_{01}^+(|b_2|,0) \simeq -|b_2|^2$ and $\tilde{\Psi}_{01}^+(-|b_2|,0) \simeq -|b_2|^2$ 
for sufficiently small $|b_2|$, and are both negative for $|b_2|>0$.
This establishes that both $\tilde{\Psi}_{01}^+(|b_2|,|b_3|)$ and $\tilde{\Psi}_{01}^+(-|b_2|,|b_3|)$ 
are of indefinite sign in the vicinity of the origin,
i.e.~they have saddle points at the origin.
\end{itemize}
Returning to the general theory, 
a discriminant $\tilde{\Psi}_{pq}^+(\bm{b},\sigma)$ or $\tilde{\Psi}_{pq}^-(\bm{b},\sigma)$ 
which has a saddle point at the origin
does not contribute to the determination of the boundary of the domain of convergence.
Such a discriminant has a nontrivial level set 
$\tilde{\Psi}_{pq}^+(\bm{b},\sigma)=0$ or $\tilde{\Psi}_{pq}^-(\bm{b},\sigma)=0$ 
which includes the origin,
and as such, the hyperplane cannot form part of a set which 
{\em encloses} any open neighborhood of the origin in $\mathbb{R}_+^{n-1}$.

We must therefore define a set of $\pm$ sign assignments $\Sigma^+_{pq}$ (resp.~$\Sigma^-_{pq}$)
consisting only of those $\bm{\sigma}$ such that
the discriminants $\tilde{\Psi}_{pq}^+(\bm{b},\sigma)$ (resp.~$\tilde{\Psi}_{pq}^-(\bm{b},\sigma)$) 
have a (constrained) local extremum at the origin.
(It is possible that the sets $\Sigma^+_{pq}$ and $\Sigma^-_{pq}$ are identical, but I do not have a proof of this.)
For these values of $\bm{\sigma}$, the origin is a one-element level set
of the equations $\tilde{\Psi}_{pq}^+(\bm{b},\sigma)=0$ or $\tilde{\Psi}_{pq}^-(\bm{b},\sigma)=0$.
We exclude the origin as an unwanted solution.
Recall the example of the quartic eq.~\eqref{eq:quartic_caveat} 
and the discriminant $\tilde{\Psi}^-_{24}(-|a_0|,|a_1|)$.
We then proceed as in the previous section.
To determine the boundary of the domain of convergence, 
we solve for $\bm{b}_{**} \ne \bm{0}_{pq}$ where
\bq
\label{eq:bdy_psi0_zero}
\biggl(\prod_{\bm{\sigma}\in\Sigma^+_{pq}} \tilde{\Psi}_{pq}^+(\bm{b}_{**},\bm{\sigma}) \biggr)
\biggl(\prod_{\bm{\sigma}\in\Sigma^-_{pq}} \tilde{\Psi}_{pq}^-(\bm{b}_{**},\bm{\sigma}) \biggr) = 0 \,.
\eq
The domain of convergence $\mathscr{D}$ is the set connected to the origin, 
bounded by the hyperplanes which satisfy eq.~\eqref{eq:bdy_psi0_zero}.
Hence we require a preliminary calculation to exclude those values of $\bm{\sigma}$ 
for which the discriminants have saddle points at the origin.
As before, if $q-p$ is odd, we can restrict attention only to $\tilde{\Psi}_{pq}^+$ and write the simpler formula
\bq
\label{eq:bdy_psi0_zero_odd}
\prod_{\bm{\sigma}\in\Sigma^+_{pq}} \tilde{\Psi}_{pq}^+(\bm{b}_{**},\bm{\sigma}) = 0 \,.
\eq
As an illustrative example, consider the quartic eq.~\eqref{eq:quartic_caveat}. 
Recall we choose $p=2$ and $q=4$ and we have set $a_2=a_4=1$ so that 
the scaled coefficients are simply $b_0=a_0$ and $b_1=a_1$.
Then
\begin{subequations}
\label{eq:disc_d24_domain}
\begin{align}
\Psi_{24}^+(|a_0|,\pm|a_1|) &= -27|a_1|^4 -4(1-36|a_0|)|a_1|^2 +16|a_0|(1-4|a_0|)^2 \,,
\\
\Psi_{24}^+(-|a_0|,\pm|a_1|) &= -27|a_1|^4 -4(1+36|a_0|)|a_1|^2 +16|a_0|(1+4|a_0|)^2 \,,
\\
\Psi_{24}^-(|a_0|,\pm|a_1|) &= -27|a_1|^4 +4(1+36|a_0|)|a_1|^2 -16|a_0|(1+4|a_0|)^2 \,,
\\
\Psi_{24}^-(-|a_0|,\pm|a_1|) &= -27|a_1|^4 +4(1-36|a_0|)|a_1|^2 +16|a_0|(1-4|a_0|)^2 \,.
\end{align}
\end{subequations}
Hence there are four distinct discriminants, viz.~$\Psi_{24}^+(|a_0|,|a_1|)$,
$\Psi_{24}^+(-|a_0|,|a_1|)$, $\Psi_{24}^-(|a_0|,|a_1|)$ and $\Psi_{24}^-(-|a_0|,|a_1|)$.
The necessary bounds for convergence yield $|a_0|\le\frac14$ and $|a_1|\le \sqrt{4/27}$.
The above expressions are all quadratics in $|a_1|^2$.
Thus to solve for $\Psi^\pm_{24}(\pm|a_0|,\pm|a_1|)=0$, 
we fix a value of $|a_0|$ and solve the resulting quadratic in $|a_1|^2$.
This procedure will not always yield a real positive solution for $|a_1|$;
the discriminants which fail to do so do not contribute to the boundary of the domain of convergence.
The discriminant $\Psi_{24}^+(-|a_0|,|a_1|)$ is such an example.
Setting the other discriminants in eq.~\eqref{eq:disc_d24_domain} to zero yields valid hyperplanes.
The resulting curves in the $(|a_0|,|a_1|)$ parameter space are displayed in Fig.~\ref{fig:foliation_d24}, for
$\Psi^+_{24}(|a_0|,|a_1|)=0$ (dashed), 
$\Psi^-_{24}(|a_0|,|a_1|)=0$ (dotdash) and
$\Psi^-_{24}(-|a_0|,|a_1|)=0$ (solid).
The first two level sets pass through the origin, because the discriminants have saddle points at the origin,
and they do not contribute to the boundary of the domain of convergence.
The domain of convergence is determined solely by the level set $\Psi^-_{24}(-|a_0|,|a_1|)=0$.
However, that level set (the solid curve) bounds {\em two} domains in Fig.~\ref{fig:foliation_d24}.
The domain $\mathscr{D}_{24}$ is given by the shaded area only, because that is the region connected to the origin.
The cross-hatched region is not connected to the origin and is not part of the domain of convergence.
Because $\Psi^-_{24}(-|a_0|,|a_1|)$ has a local minimum at the origin,
the domain of convergence is the component connected to the origin such that 
\bq
\mathcal{D}_{24} = \{ \Psi^-_{24}(-|a_0|,|a_1|) \ge 0 \} \,.
\eq
The caveat about `connectedness to the origin' is essential in this case.
The domain $\mathcal{D}_{24}$ is also not convex.
This is demonstrated in Fig.~\ref{fig:nonconvex_d24}.
The domain $\mathscr{D}$ is the shaded area and is bounded by the solid curve.
The dashed line is the straight line which joins the vertices 
$(\hat{b}_0,0) = (\frac14,0)$ 
and $(0,\hat{b}_1) = (0,\sqrt{4/27})$
to form a right-angled triangle with the origin.
Observe that $\mathcal{D}_{24}$ is not convex,
hence the full domain of convergence $\mathcal{D}_{24}$ is not convex.

\subsection{Normalization of discriminant}
The literature in fact contains multiple normalization conventions for discriminants.
For example for the quadratic $ax^2 +bx +c$, some authors define the discriminant to be 
$\Delta = 4ac-b^2$ and others prefer $\Delta = b^2-4ac$.
Hence the directions of the inequalities in expressions such as eq.~\eqref{eq:pt_cubic_all_me}
could be reversed, depending on the normalization convention employed for the discriminant.
The reader should beware of this important detail.
However, the formulas in both eqs.~\eqref{eq:bdy_psi0_nonzero} and \eqref{eq:bdy_psi0_zero}
are independent of the normalization of the discriminant,
which is an advantage of the above formalism.

\setcounter{equation}{0}
\section{\label{sec:principal_brioschi_quintic} Applications: principal and Brioschi quintics}
The principal and Brioschi forms of the quintic are tetranomials, 
and furnish nontrivial applications of the more sophisticated formalism of Sec.~\ref{sec:conv3}, 
to determine the domains of convergence of their solutions by infinite series.
We treat them in turn.
The principal quintic form is
\bq
\label{eq:principal_quintic}
a_0 + a_1 x + a_2 x^2 + x^5 = 0 \,.
\eq
The discriminant is 
\bq
\begin{split}
\Delta_{\rm prin} = a_5^2 \Bigl(3125 a_0^4 a_5^2 
+2250 a_0^2 a_1 a_2^2 a_5
-1600 a_0 a_1^3 a_2 a_5
+108 a_0 a_2^5 
+256 a_1^5 a_5
-27 a_1^2 a_2^4\Bigr) \,.
\end{split}
\eq
We factor out $a_5^2$ and employ reduced discriminants for the formulas for the 
domains of convergence, expressed in terms of the scaled coefficients $(b_0,b_1,b_2 ,b_5)$
for each choice of $p$ and $q$:
\begin{subequations}
\begin{align}
\mathcal{D}_{01} &= \{ \tilde{\Psi}^+_{01}(|b_2|,-|b_5|) \le 0 \} \,,
\\
\mathcal{D}_{02} &= \{ \tilde{\Psi}^+_{02}(|b_1|,|b_5|) \ge 0 \} \cap \{ \tilde{\Psi}^-_{02}(|b_1|,|b_5|) \le 0 \} \,,
\\
\mathcal{D}_{05} &= \{ \tilde{\Psi}^+_{05}(-|b_1|,-|b_2|) \ge 0 \} \,,
\\
\mathcal{D}_{12} &= \{ \tilde{\Psi}^+_{12}(|b_0|,-|b_5|) \le 0 \} \cap \{ \tilde{\Psi}^+_{12}(-|b_0|,|b_5|) \le 0 \} \,,
\\
\mathcal{D}_{15} &= \{ \tilde{\Psi}^+_{15}(|b_0|,|b_2|) \ge 0 \} \cap \{ \tilde{\Psi}^-_{15}(|b_0|,|b_2|) \le 0 \} \,,
\\
\mathcal{D}_{25} &= \{ \tilde{\Psi}^+_{25}(-|b_0|,|b_1|) \le 0 \} \,.
\end{align}
\end{subequations}
As always, the domains of convergence consist only of the components which are connected to the origin.
Next, the Brioschi normal form of the quintic is \cite{Brioschi}
\bq
\label{eq:brioschi_quintic}
x^5 - 10Cx^3 + 45C^2x -C^2 = 0 \,.
\eq
The coefficients are all functions of a single parameter $C$ and are hence not independent.
There is a real root for all real $C$.
If $C=0$ there is a repeated root of multiplicity five at $x=0$.
We write eq.~\eqref{eq:brioschi_quintic} as
$a_0 +a_1x +a_3x^3 +a_5x^5 = 0$ where
$a_0 = -C^2$, $a_1 = 45C^2$, $a_3 = -10C$ and $a_5=1$.
The discriminant is
\bq
\begin{split}
\Delta_{\rm Br} &= a_5 \Bigl( 3125 a_0^4 a_5^3 +2000 a_0^2 a_1^2 a_3 a_5^2 -900 a_0^2 a_1 a_3^3 a_5
+108 a_0^2 a_3^5 +256 a_1^5 a_5^2 -128 a_1^4 a_3^2 a_5 +16 a_1^3 a_3^4 \Bigr) \,.
\end{split}
\eq
We factor out $a_5$ and employ reduced discriminants for the formulas for the 
domains of convergence, now expressed in terms of the scaled coefficients $(b_0,b_1,b_3 ,b_5)$.
\begin{itemize}
\item
Setting $p=0$ and $q=1$ yields one root.
Put $x = (a_0/a_1)z = -z/45$, then
\bq
0 = 1 +z -\frac{10}{45^3C}\, z^3 +\frac{1}{45^5 C^2}\, z^5 \,.
\eq
The domain of convergence is given by
\bq
\mathcal{D}_{01} = \{ \tilde{\Psi}^+_{01}(-|b_3|,-|b_5|) \ge 0 \} \,.
\eq
This yields the condition
\bq
\label{eq:brioschi_d01}
1 +29376 |C| -36578304 |C|^2 \le 0 \,.
\eq
This is satisfied for 
\bq
\label{eq:brioschi_d01_root}
|C| \ge \frac{17 +13\sqrt2}{32\cdot27\cdot49} 
\simeq 8.358 \cdot 10^{-4} \,.
\eq

\item
Setting $p=1$ and $q=5$ yields four roots.
Put $x = (a_1/a_5)^{1/4}z = (45C^2)^{1/4}z$, then
\bq
0 = -\frac{1}{45^{5/4}C^{1/2}} +z -\frac{10}{45^{1/2}} z^3 +z^5 \,.
\eq
The domain of convergence is given by
\bq
\mathcal{D}_{15} = \{ \tilde{\Psi}^+_{15}(|b_0|,-|b_3|) \ge 0 \} \cap \{ \tilde{\Psi}^-_{15}(|b_0|,-|b_3|) \le 0 \} \,.
\eq
This also yields the condition eq.~\eqref{eq:brioschi_d01}
and hence is also satisfied for 
the bound in eq.~\eqref{eq:brioschi_d01_root}.

\item
Setting $p=0$ and $q=5$ yields five roots.
Put $x = (a_0/a_5)^{1/5}z = -C^{2/5}z$, then
\bq
0 = 1 +45C^{2/5} x -10C^{1/5} z^3 +z^5 \,.
\eq
The domain of convergence is given by
\bq
\mathcal{D}_{05} = \{ \tilde{\Psi}^+_{05}(-|b_1|,-|b_3|) \ge 0 \} \,.
\eq
This yields the condition
\bq
\label{eq:brioschi_d05}
1 -29376 |C| -36578304 |C|^2 \ge 0 \,.
\eq
This is satisfied for 
\bq
\label{eq:brioschi_d05_root}
|C| \le \frac{-17 +13\sqrt2}{32\cdot27\cdot49} 
\simeq 0.327 \cdot 10^{-4} \,.
\eq

\item
Next set $p=0$ and $q=3$.
Put $x = (a_0/a_3)^{1/3}z = (C/10)^{1/3}z$ then
\bq
0 = 1 -\frac{45C^{1/3}}{10^{1/3}} z +z^3 -\frac{1}{10^{5/3}C^{1/3}} z^5 \,.
\eq
The domain of convergence is given by
\bq
\mathcal{D}_{03} = \{ \tilde{\Psi}^+_{03}(-|b_1|,-|b_5|) \ge 0 \} \,.
\eq
This yields the condition
\bq
\label{eq:brioschi_d03}
-(1 - 1728 |C|)^2 \ge 0 \,.
\eq
This is only satisfied by the single value $|C| = 1/1728$,
but $|C| = 1/1728$ lies in a domain not connected to the origin.
Hence this scenario yields no roots.
We see that the stipulation `connected to the origin' is essential.

\item
Next set $p=1$ and $q=3$.
Put $x = (a_1/a_3)^{1/2}z = i(45C/10)^{1/2}z$, then
\bq
0 = \frac{i10^{1/2}}{45^{3/2}C^{1/2}} +z +z^3 +\frac{45}{100} z^5 \,.
\eq
The necessary upper bound for convergence is $|b_5|\le\frac14$.
However, $b_5 = 0.45$, which exceeds the above bound.
Hence the series solution does not converge for any value of $C$, for this scenario.

\item
Next set $p=3$ and $q=5$.
Put $x = (a_3/a_5)^{1/2}z = i(10C)^{1/2}z$, then
\bq
0 = \frac{i}{10^{5/2}C^{1/2}} +\frac{45}{100}z +z^3 +z^5 \,.
\eq
The necessary upper bound for convergence is $|b_1|\le\frac14$.
However, $b_1 = 0.45$, which exceeds the above bound.
Hence the series solution does not converge for any value of $C$ for this scenario.
\end{itemize}
The Brioschi quintic normal form yields some instructive insights.
First, for three of the six possible choices of $p$ and $q$,
the series solutions do not converge for any value of $C$.
Second, observe that the choices
$p=0$, $q=1$ and $p=1$, $q=5$ together yield five roots, but only if
$|C| \gtrsim 8.358 \cdot 10^{-4}$
(see eq.~\eqref{eq:brioschi_d01_root}),
whereas the choice
$p=0$, $q=5$ also yields five roots, but only if
$|C| \lesssim 0.327 \cdot 10^{-4}$
(see eq.~\eqref{eq:brioschi_d05_root}).
Hence there is a {\em gap of values}
for which there is 
{\em no convergent series solution of the Brioschi quintic}, for any choice of $p$ and $q$, given by 
\bq
\label{eq:brioschi_gap}
\frac{-17+13\sqrt2}{32\cdot27\cdot49} 
\le |C| \le
\frac{17+13\sqrt2}{32\cdot27\cdot49} \,.
\eq
The Brioschi quintic demonstrates that the formulas in Sec.~\ref{sec:conv3} 
do {\em not} imply that the domains of convergence
for an algebraic equation of degree $n$ span all values of the coefficients, i.e.~all of $\mathbb{C}^{n+1}$.
The criterion for convergence may not be satisfied for any of the choices of $p$ and $q$.
Nevertheless, it was shown in Sec.~\ref{sec:quintic} 
that convergent series solutions do exist for all values of the coefficients of a general quintic,
via the use of the Bring-Jerrand normal form.
Hence one must examine every algebraic equation on its merits; 
some transformations may work better than others.

\section{\label{sec:conc} Conclusion}
This author was led to the main ideas of this paper
because they are required to prove results in probability and statistics (not reported here).
The papers by numerous authors were cited
and the various notations, definitions, identities 
and nomenclature were collected in a common setting.
Note that although most of the derivations in the literature treat only integer valued parameters,
Theorem \ref{thm:thm_k}
is applicable for arbitrary complex coefficients 
and real (or even complex) exponents.
The early works by 
Lambert \cite{Lambert_1758,Lambert_1770}
and 
Euler \cite{Euler_1779} and
were shown to be Fuss--Catalan series.
An important application of the formalism was the
solution of algebraic equations by infinite series.
This is a heavily studied problem and
contact was made with the works of numerous authors
\cite{Ramanujan_Pt1,
Birkeland_algebraic_1927,
Eagle_trinomial,
Lewis,
McClintock,
Mellin}.
An example was to present convergent
Fuss--Catalan series solutions for all the roots 
the Bring-Jerrard normal form, thence the roots of a general quintic,
for arbitrary values of the quintic coefficients.
Two bounds for the absolute convergence of general Fuss--Catalan series were derived
(necessary but not sufficient and sufficient but not necessary).
For the important special case of 
the solutions of algebraic equations by infinite series, 
a new necessary and sufficient bound for absolute convergence 
was presented in Sec.~\ref{sec:conv3},
correcting and extending earlier work in the field \cite{Passare_Tsikh}.

\section*{\label{sec:ack} Acknowledgements}
I am deeply indebted to numerous individuals who helped me generously with their time and enthusiasm.
However, special mention must go to 
Professors R.~E.~Borcherds and S.~J.~Dilworth,
and to Dr.~P.~M.~Strickland,
without whose unflagging support this work simply would not have seen fruition. 
It is my enormous pleasure to thank them, and also, in alphabetical order,
Professors
I.~M.~Gessel,
H.~W.~Gould,
R.~L.~Graham,
O.~Patashnik,
T.~J.~Ransford,
R.~P.~Stanley and
D.~Zeilberger.
\\
{\bf\color{red} Addendum:} I thank Drs.~F.~de Sousa Coelho and D.~Rubine for pointing out misprints in previous versions of this note.

\appendix
\setcounter{equation}{0}
\section{\label{app:app} Miscellaneous items}
This Appendix lists various items which were not used in the main body of the paper,
essentially for completeness of the exposition.
According to the information in Appendix B of Stanley's text \cite{Stanley_book}
(the Appendix was written by Pak), the name `Catalan numbers'
only came into prominent use after the publication of Riordan's monograph \cite{Riordan}
(in 1968, first edition).
Hence it is understandable if authors such as Gould \cite{Gould_1956,Gould_1957} and Raney \cite{Raney}, 
also earlier authors such as Mellin \cite{Mellin} and Schl{\"a}fli \cite{Schlafli},
did not mention Fuss or Catalan.
Belardinelli's memoir \cite{Belardinelli_memoir}
contains an overview of the solutions of algebraic equations
using hypergeometric series.
His extensive bibliography lists several papers on
functions of several complex variables,
but not papers on combinatorics such as by Raney \cite{Raney}.
There is evidently a diversity of notations and terminology, and duplication of proofs.

The `diversity of notations' leads to an immediate caveat:
different authors employ the same symbols, 
such as $n$, $p$ or $q$, to mean different things
and it is impractical to disambiguate all the notations in the equations below.
The reader is warned to consult the original literature for the precise meanings of all symbols displayed below.  
  
Turning to technical matters,
Mohanty \cite{Mohanty_1966} derived additional convolution identities not mentioned in the main text above.
I list only one, viz.~\cite[eq.~(11)]{Mohanty_1966}, because it subsumes the others as special cases.
In the notation of this paper, \cite[eq.~(11)]{Mohanty_1966} is
\bq
\label{eq:myconv3_k}
\sum_{\bm{j}\in\mathbb{N}^k} (p+\bm{q}\cdot\bm{j})
\mathscr{A}_{\bm{j}}(\bm{b},a)
\mathscr{A}_{\bm{n}-\bm{j}}(\bm{b},c) = 
\frac{p(a+c) +a\bm{q}\cdot\bm{n}}{a+c} \mathscr{A}_{\bm{n}}(\bm{b},a+c) \,.
\eq
All of $a$, $c$, $p$, $\bm{b}$ and $\bm{q}$ are complex valued
and $\bm{b}$, $\bm{j}$, $\bm{n}$ and $\bm{q}$ are $k$-tuples.
Put $\bm{q}=0$ then eq.~\eqref{eq:myconv3_k} yields \cite[eq.~(9)]{Mohanty_1966}, 
which is eq.~\eqref{eq:fc_kconv} above.
Put $p=c+\bm{b}\cdot\bm{n}$ and $\bm{q}=-\bm{b}$ then eq.~\eqref{eq:myconv3_k} yields \cite[eq.~(10)]{Mohanty_1966}. 
Mohanty \cite{Mohanty_1966} actually cited Gould \cite{Gould_1957} for the single-parameter convolution identities;
Mohanty generalized them to multiparameter versions.
Gould \cite{Gould_1956,Gould_1957} 
proved several convolution identities
and suggested that they are all special cases of a single general formula.
The exposition below follows Raney's \cite{Raney} summary of Gould's work.
Define the numbers \cite[eq.~(7.7)]{Raney} 
\bq
\begin{split}
G(\alpha,0;\beta,\gamma) &= 1 \,,
\\
G(\alpha,n;\beta,\gamma) &= \frac{\alpha}{n!}\prod_{m=1}^{n-1} (\alpha+\beta n -\gamma m) \,.
\end{split}
\eq
See also the polynomial $P_k(p)$ by Gould \cite[Sec.~5]{Gould_1957}
and associated comments therein about the work of
Schl{\"a}fli \cite{Schlafli}.
Then \cite[eq.~(7.8)]{Raney} 
\bq
G(\alpha_1+\alpha_2,n;\beta,\gamma) = \sum_{n_1+n_2=n} G(\alpha_1,n_1;\beta,\gamma) G(\alpha_2,n_2;\beta,\gamma) \,.
\eq
Here $\alpha_1,\alpha_2,\beta,\gamma\in\mathbb{A}$ where $\mathbb{A}$ is a commutative ring and $n,n_1,n_2\in\mathbb{N}$.
Gould \cite[Sec.~5]{Gould_1957} 
also proved that the convolution identity derived by
Schl{\"a}fli \cite{Schlafli}, 
in the latter's 1847 paper on Lambert series,
was equivalent to \cite[eq.~(10)]{Gould_1956}. 
See also Riordan \cite{Riordan}
for additional combinatorial identities
and Strehl \cite{Strehl} for an overview of numerous multiparameter identities.
If $\gamma=0$ then 
\bq
G(\alpha,n;\beta,0) = \frac{\alpha}{\alpha+\beta n} \frac{(\alpha+\beta n)^n}{n!} \,.
\eq
If $\gamma\ne0$ then the above is proportional to a Fuss--Catalan number
\bq
\begin{split}
G(\alpha,n;\beta,\gamma) &= 
\frac{(\alpha/\gamma)\gamma^n}{n!}\prod_{m=1}^{n-1} \Bigl(\frac{\alpha+\beta n}{\gamma} -m\Bigr) 
\\
&= \gamma^n A_n(\beta/\gamma, \alpha/\gamma) \,.
\end{split}
\eq
Note that this relation works for $n=0$ also.
A multiparameter generalization might be as follows
\bq
\begin{split}
G(\alpha,\bm{0};\bm{\beta},\gamma) &= 1 \,,
\\
G(\alpha,\bm{n};\bm{\beta},\gamma) &= \frac{\alpha}{n_1!\cdots n_k!}
\prod_{m=1}^{|\bm{n}|-1} (\alpha+\bm{\beta}\cdot\bm{n} -\gamma m) \,.
\end{split}
\eq
Gould also published a later paper \cite{Gould_1974}
with additional formulas,
but its contents are beyond the scope of this paper.
The work of Gould may lead to a more general set of multiparameter identities and generating functions.
The matter will be left to future work.

Kahkeshani \cite{Kahkeshani}
has defined so-called `generalized Catalan numbers' via 
\bq
C(m,n) = \frac{1}{n(m-1)+1} \binom{2n(m-1)}{\underbrace{n,\dots,n}_{m-1},n(m-1)} \,.
\eq
Let us process this as follows.
Set $k=m-1$ and $r=1$ in eq.~\eqref{eq:fc_1def}.
Also set $t_1=\cdots=t_k=n$ and $\mu_1=\cdots=\mu_k=2$, 
so $|\bm{t}|=n(m-1)$ and $\bm{t}\cdot\bm{\mu} = 2n(m-1)$.
Then
\bq
\begin{split}
C(m,n) &= \frac{1}{n(m-1)+1} \frac{1}{(n!)^{m-1}} \prod_{j=0}^{n(m-1)-1} (2n(m-1)-j)
\\
&= \frac{1}{(n!)^{m-1}} \prod_{j=1}^{n(m-1)-1} (2n(m-1)+1-j) 
\\
&= \mathscr{A}_{(n,\dots,n)}((2,\dots,2),1) \,.
\end{split}
\eq
Hence Kahkeshani's definition is a special case of the 
multiparameter Fuss--Catalan numbers defined in this paper.
Note that Chu's \cite{Chu} and 
Kahkeshani's \cite{Kahkeshani}
nomenclature `generalized Catalan numbers' 
should not be confused with each other.
 
We close with a comment on the paper by Aval \cite{Aval},
who defined so-called `multivariate Fuss--Catalan numbers' via \cite[remark 3.2]{Aval} 
\bq
B_p(n,k_1,k_2,\dots,k_{p-1}) = \biggl( \prod_{i=1}^{p-1} \binom{n+k_i-1}{k_i} \biggr)\,\frac{n-\sum_{i=1}^{p-1} k_i}{n} \,.
\eq
Clearly $B_p(\cdot)=1$ for $p=0$ and $p=1$. For $p\ge2$ we have
\bq
\begin{split}
B_p(n,k_1,k_2,\dots,k_{p-1}) &= \frac{n-\sum_{i=1}^{p-1} k_i}{n}
\prod_{i=1}^{p-1} \biggl(\frac{1}{k_1!}\prod_{j=0}^{k_i-1} (n+k_i-1-j)\biggr)
\\
&= n^{p-2}(n-|\bm{k}|) \prod_{i=1}^{p-1} \biggl(\frac{1}{k_1!}\prod_{j=1}^{k_i-1} (n+k_i-j)\biggr)
\\
&= n^{p-2} (n-|\bm{k}|) \prod_{i=1}^{p-1} A_{k_i}(1,n) \,.
\end{split}
\eq
Hence for $p\ge2$, Aval's definition equals a product of $p-1$ 
single-parameter Fuss--Catalan numbers, with a prefactor.
This is different from the multiparameter Fuss--Catalan numbers defined in this paper.

\setcounter{equation}{0}
\section{\label{app:sturmfels} $\mathscr{A}$-hypergeometric series}
Sturmfels \cite{Sturmfels}
published an elegant analysis employing so-called $\mathscr{A}$-hypergeometric series 
to solve for the roots of the general algebraic equation of degree $n$.
A brief comparison with his work is presented here.
His first example is for the quintic.
Let us write the quintic in the form
\bq
x = -\frac{a_0}{a_1} -\frac{a_2}{a_1}x^2 -\frac{a_3}{a_1}x^3 -\frac{a_4}{a_1}x^4 -\frac{a_5}{a_1}x^5 \,.
\eq
This corresponds to $p=0$ and $q=1$ in my terminology,
so $q-p=1$ and the Fuss--Catalan series yields one root, which is 
\bq
\label{eq:fc_quint_sturmfels}
\begin{split}
x_{\rm root} &= -\frac{a_0}{a_1}
\biggl[ \sum_{\bm{t}\in\mathbb{N}^4} A_{\bm{t}}(\bm{\mu}, 1)\, \frac{e^{-i\pi \bm{t}\cdot\bm{\mu}}}{a_0^{t-\bm{t}\cdot\bm{\mu}} a_1^{\bm{t}\cdot\bm{\mu}}} 
\Bigl(\prod_{j\in\mathscr{N}_{npq}} a_j^{t_j}\Bigr) \biggr]
\\
&= -\frac{a_0}{a_1} \biggl[\, 1 
+\frac{a_0a_2}{a_1^2}
-\frac{a_0^2a_3}{a_1^3}
+\frac{a_0^3a_4}{a_1^4}
-\frac{a_0^4a_5}{a_1^5}
+\frac{2a_0^2a_2^2}{a_1^4}
-\frac{5a_0^3a_2a_3}{a_1^5}
+\cdots
\biggr] \,.
\end{split}
\eq
This equals the root $X_{1,-1}$ of Sturmfels \cite{Sturmfels}.
Next let us select $p=0$ and $q=5$ and write
\bq
x^5 = -\frac{a_0}{a_5} -\frac{a_1}{a_5}x -\frac{a_2}{a_5}x^2 -\frac{a_3}{a_5}x^3 -\frac{a_4}{a_5}x^4 \,.
\eq
The series yields five roots.
Following Sturmfels, we define $\xi = e^{i\pi(2\ell+1)/5}$ as a root of $-1$.
Then the roots of the quintic are given by
\bq
\begin{split}
x_\xi &= \xi\, \frac{a_0^{1/5}}{a_5^{1/5}} \,
\sum_{\bm{t}\in\mathbb{N}^4}
A_{\bm{t}}\Bigl(\bm{\mu}, \frac15\Bigr)\, \frac{\xi^{5\bm{t}\cdot\bm{\mu}}}{a_0^{t-\bm{t}\cdot\bm{\mu}} a_5^{\bm{t}\cdot\bm{\mu}}} 
\Bigl(\prod_{j\in\mathscr{N}_{npq}} a_j^{t_j}\Bigr) 
\\
&= \xi\, \frac{a_0^{1/5}}{a_5^{1/5}} 
+\frac15\biggl(
\frac{\xi^2 a_1}{a_0^{3/5}a_5^{2/5}} 
+\frac{\xi^3 a_2}{a_0^{4/5}a_5^{3/5}} 
+\frac{\xi^4 a_3}{a_0^{1/5}a_5^{4/5}} 
-\frac{ a_4}{a_5} \biggr)
+\cdots 
\end{split}
\eq
These are the leading terms of the $\mathscr{A}$-hypergeometric series for the roots $X_{5,\xi}$
of Sturmfels (see \cite{Sturmfels} for details of his notation)
\bq
X_{5,\xi} = \xi\,\biggl[\frac{a_0^{1/5}}{a_5^{1/5}}\biggr]
+\frac15\biggl( 
\xi^2\,\biggl[\frac{a_1}{a_0^{3/5} a_5^{2/5}}\biggr]
+\xi^3\,\biggl[\frac{a_2}{a_0^{2/5} a_5^{3/5}}\biggr]
+\xi^4\,\biggl[\frac{a_3}{a_0^{1/5} a_5^{4/5}}\biggr]
-\biggl[\frac{a_4}{a_5}\biggr] \biggr) \,.
\eq
\begin{itemize}
\item
This illustrates a difference between the use of $\mathscr{A}$-hypergeometric series and Fuss--Catalan series.
In general, for a polynomial of degree $n$, a total of $n$ $\mathscr{A}$-hypergeometric series 
are required to derive solutions for all the roots.
In contrast, a Fuss--Catalan series encapsulates all the roots in one series, cycling through the roots of unity.
There is a single formula for all the terms in a Fuss--Catalan series.

\item
A similar remark applies to the work of 
Birkeland \cite{Birkeland_algebraic_1927}.
In general, a total of $|q-p|^{n-1}$ hypergeometric series 
are required to express the solutions for all the roots
of an algebraic equation of degree $n$.

\item
For the `triangulation of unit length' of the quintic,
Sturmfels obtained expressions for the five roots $X_{j,-1}$, $j=1,\dots,5$
(see \cite{Sturmfels} for details).
The example $X_{1,-1}$ was displayed above.
If all of the coefficients of the quintic are real
then all of the series for the roots $X_{j,-1}$ are real.
However, a quintic with all real coefficients may not have all real roots.
As Sturmfels noted, the $\mathscr{A}$-hypergeometric series have finite radii of convergence.
Sturmfels offered a convergence criterion for the $\mathscr{A}$-hypergeometric series in his Theorem 3.2,
reproduced here for ease of reference.
(Consult \cite{Sturmfels} for definitions and notation).
\begin{theorem*}[\cite{Sturmfels} Theorem 3.2]
\label{thm:sturmfels_thm3.2}
The $n$ series $X_{j,\xi}$ are roots of the general equation of order $n$; that is;
$f(X_{j,\xi})=0$. There exists a constant $M$ such that all $n$ series $X_{j,\xi}$ converge whenever
\bq
\label{eq:sturmfels_thm3.2}
|a_{i_{j-1}}|^{i_j-k}
|a_{i_j}|^{k-i_{j-1}} \ge M |a_k|^{d_j} \qquad 
\textrm{for all $1 \le j\le r$ and $k\not\in \{i_{j-1},i_j\}$}\,.
\eq
\end{theorem*}
The above corrects a misprint in the direction of the inequality in \cite[Thm.~3.2]{Sturmfels}.
(I thank Sturmfels \cite{Sturmfels_privcomm}
for confirming the correct direction of the inequality.)

\item
Sturmfels also stated (last paragraph in \cite[Section 3]{Sturmfels})
``First, no good bound for $M$ seems to be currently known, and, second,
for many concrete instances the inequalities (3.2) 
{\em [this is reproduced as eq.~\eqref{eq:sturmfels_thm3.2} above]}
will not hold for any triangulation.
In this case one has to carry out analytic continuation: \dots''
From eq.~\eqref{eq:conv_nec_j}, we can supply a value for $M$ above.
First define 
\bq
M_{k,i_j,i_{j-1}} = \frac{|k-i_{j-1}|^{k-i_{j-1}} |i_j-k|^{i_j-k}}{|i_j-i_{j-1}|^{i_j-i_{j-1}}} \,.
\eq
Then $M = \min(M_{k,i_j,i_{j-1}})/(n-1)$.
Sturmfels was however correct to note that a convergent series solution might not exist for any triangulation.
We saw some examples earlier in this paper.
\end{itemize}

\setcounter{equation}{0}
\section{Relation of Fuss-Catalan series to hypergeometric series}\label{app:hyp_series}
\subsection{Relation of series}\label{sec:rel_series}
It was stated in Sec.~\ref{sec:Birkeland}
that Birkeland \cite{Birkeland_algebraic_1927} published papers on the solutions of algebraic equations using hypergeometric series.
It was also stated in Sec.~\ref{sec:Birkeland} that it would take us too far afield to discuss hypergeometric series in this paper.
However, the connection is tractable in some circumstances, and we present the connection here.

We begin by stating the definition of the hypergeometric function ${}_mF_n(a_1,\dots,a_m;c_1,\dots,c_n;z)$.
The definition is given in many textbooks.
Here $m$ and $n$ are nonnegative integers and the series coefficients contain ratios of products of rising factorials as follows
\bq
\label{def:hyp_def}
\begin{split}
{}_mF_n(a_1,\dots,a_m;c_1,\dots,c_n;z) &= 1 +\frac{a_1\dots a_m}{c_1\dots c_n}\,\frac{z}{1!}
+\frac{a_1(a_1+1)\dots a_m(a_m+1)}{c_1(c_1+1)\dots c_n(c_n+1)}\,\frac{z^2}{2!}
\\
&\qquad
+\frac{a_1(a_1+1)(a_1+2)\dots a_m(a_m+1)(a_m+2)}{c_1(c_1+1)(c_1+2)\dots c_n(c_n+1)(c_n+2)}\,\frac{z^3}{3!} +\dots
\end{split}
\eq
Note the following.
\begin{enumerate}
\item
If $m=0$, the numerator parameter list is absent and the numerator is set equal to unity.
\item
If $n=0$, the denominator parameter list is absent and the denominator is set equal to unity.
\item
If the series is nonterminating, its domain of convergence is $|z|<1$ if $m=n+1$,
$|z|<\infty$ if $m < n+1$ and $z=0$ if $m > n+1$.
\item
If any of the parameters $a_i$, $i=1,\dots,m$ in the numerator are negative integers, the series terminates and is a polynomial in $z$.
\item
If any of the parameters $c_j$, $j=1,\dots,n$ in the denominator are negative integers, the series is ill-defined.
\item
However, there are caveats, e.g.~if both the numerator and denominator parameter lists contain negative integers.
We refer the reader to textbooks for the details.
We assume below that the series in eq.~\eqref{def:hyp_def} is well-defined.
\end{enumerate}

Let us relate a Fuss-Catalan series to a hypergeometric series.
We treat only the case of scalar $\mu$ and $z$, viz.~$\mathcal{B}(\mu;r;z)$.
From Remark \ref{rem:b_symm_mu_scalar} and eq.~\eqref{eq:b_symm_mu_scalar},
it suffices to treat $\mu\ge0$ and $r=1$ only.
The case $\mu=0$ is trivial, viz.~$\mathcal{B}(0;1;z) = 1+z = {}_1F_0(-1;;-z)$, hence we assume $\mu>0$ below.

For integer $\mu > 0$, the Fuss-Catalan series $\mathcal{B}(\mu;1;z)$ is related to a ${}_\mu F_{\mu-1}$ hypergeometric series as follows:
\begin{equation}
\label{eq:B_hyp_pos}  
\mathcal{B}(\mu;1;z) = {}_\mu F_{\mu-1}\Bigl(\frac{1}{\mu},\dots,\frac{\mu-1}{\mu},\frac{\mu}{\mu};
\frac{2}{\mu-1},\dots,\frac{\mu-1}{\mu-1},\frac{\mu}{\mu-1};
\frac{\mu^\mu}{(\mu-1)^{\mu-1}}\,z\Bigr) \,.
\end{equation}
Note the following.
\begin{enumerate}
\item
The Fuss-Catalan series $\mathcal{B}(\mu;1;a)$ is well-defined even if $\mu$ is positive but not an integer.
However, the hypergeometric series in eq.~\eqref{eq:B_hyp_pos} requires the value of $\mu$ to be a positive integer.  

\item
If $\mu>2$, the numerator and denominator parameter lists in ${}_\mu F_{\mu-1}$ in eq.~\eqref{eq:B_hyp_pos} 
both contain a term which equals unity,
which we cancel to yield a ${}_{\mu-1}F_{\mu-2}$ hypergeometric series as follows
\begin{equation}
\label{eq:B_hyp_pos3}
\mathcal{B}(\mu;1;z) = {}_{\mu-1} F_{\mu-2}\Bigl(\frac{1}{\mu},\dots,\frac{\mu-1}{\mu};
\frac{2}{\mu-1},\dots,\frac{\mu-2}{\mu-1},\frac{\mu}{\mu-1};
\frac{\mu^\mu}{(\mu-1)^{\mu-1}}\,z\Bigr) \,.
\end{equation}
\end{enumerate}
It is actually possible to derive a hypergeometric series if $\mu$ is a negative integer.
As noted, it suffices to set $r=1$.
For $\mu$ a negative integer, the na{\"\i}ve relation is to a ${}_{|\mu|+1}F_{|\mu|}$ hypergeometric series as follows:
\begin{equation}
\label{eq:B_hyp_neg}  
\mathcal{B}(\mu;1;z) =
{}_{|\mu|+1}F_{|\mu|}\biggl(\frac{-1}{|\mu|+1},\frac{0}{|\mu|+1},\frac{1}{|\mu|+1},\dots,\frac{|\mu|-1}{|\mu|+1};\frac{0}{|\mu|},\frac{1}{|\mu|}\dots\frac{|\mu|-1}{|\mu|};
  -\frac{(|\mu|+1)^{|\mu|+1}}{|\mu|^{|\mu|}}z\biggr) \,.
\end{equation}
The vanishing terms in the numerator and denominator parameter lists pose a problem.
One possibility is to set $r=1+\epsilon$ and take the limit $\epsilon\to0$.
This is possible, but not computationally useful.
Instead we play the following trick.
We omit the zero terms in the numerator and denominator parameter lists in the hypergeometric function in eq.~\eqref{eq:B_hyp_neg}  
and consider the following ${}_{|\mu|}F_{|\mu|-1}$ modified hypergeometric function.
For brevity, define $\chi = -(|\mu|+1)^{|\mu|+1}/|\mu|^{|\mu|})z$.
Then, for $\mu$ a negative integer, we obtain
\begin{equation}
\begin{split}
  {}_{|\mu|}F_{|\mu|-1}\biggl(\frac{-1}{|\mu|+1},\frac{1}{|\mu|+1},\dots,\frac{|\mu|-1}{|\mu|+1};
  \frac{1}{|\mu|},\dots,\frac{|\mu|-1}{|\mu|}; \chi\biggr)
  &= 1 +\frac{\frac{-1}{|\mu|+1}\frac{1}{|\mu|+1}\,\frac{|\mu|-1}{|\mu|+1}}{\frac{1}{|\mu|}\dots\frac{|\mu|-1}{|\mu|}}\, \chi +\dots
  \\
  &= 1 + \frac{|\mu|+1}{|\mu|}\Bigl[-1 + \mathcal{B}(\mu;1;z)\Bigr] \,.        
\end{split}
\end{equation}
Hence for $\mu$ a negative integer, the connection is proportional to a constant plus a hypergeometric series.
\begin{equation}
\label{eq:B_hyp_neg1}  
  \mathcal{B}(\mu;1;z) = \frac{|\mu|}{|\mu|+1}\biggl[\, \frac{1}{|\mu|} + {}_{|\mu|}F_{|\mu|-1}\biggl(\frac{-1}{|\mu|+1},\frac{1}{|\mu|+1},\dots,\frac{|\mu|-1}{|\mu|+1};
    \frac{1}{|\mu|},\dots,\frac{|\mu|-1}{|\mu|}; -\frac{(|\mu|+1)^{|\mu|+1}}{|\mu|^{|\mu|}}z\biggr)\biggr] \,.
\end{equation}

If the value of $\mu$ is not an integer, there is no simple relation between a Fuss-Catalan series and a hypergeometric series.
In general, we require a weighted sum of hypergeometric series.
It is perhaps easiest and best to explain the matter by presenting the example of Birkeland's hypergeometric series solutions for the trinomial
and relating them to our Fuss-Catalan series solutions in Sec.~\ref{sec:tri}.

\subsection{Series for selected roots of the trinomial}

\subsubsection{Case $n=1$}
The trinomial equation is $x^{m+1} +ax +b = 0$ in eq.~\eqref{eq:tri}.
Using eq.~\eqref{eq:trinom_soln_2} yields one root, given by
\begin{equation}
x_{\rm root} = -\frac{b}{a}\mathcal{B}\Bigl(m+1; 1; (-1)^{m+1}\frac{b^m}{a^{m+1}}\Bigr) \,.
\end{equation}
Using eq.~\eqref{eq:B_hyp_pos} yields (with $z = (-1)^{m+1}b^m/a^{m+1}$)
\begin{equation}
\mathcal{B}(m+1;1;z) = {}_{m+1} F_{m}\Bigl(\frac{1}{m+1},\dots,\frac{m}{m+1},\frac{m+1}{m+1};
\frac{2}{m},\dots,\frac{m}{m},\frac{m+1}{m};
\frac{(m+1)^{m+1}}{m^m}\,z\Bigr) \,.
\end{equation}
As noted in eq.~\eqref{eq:B_hyp_pos3}, simplification is possible if $m>2$.

\subsubsection{Case $m=1$}
The trinomial equation is $x^{n+1} +ax^n +b = 0$ in eq.~\eqref{eq:tri}.
Using eq.~\eqref{eq:trinom_soln_3} yields one root, given by
\begin{equation}
x^\prime_{\rm root} = -a\,\mathcal{B}\Bigl(-n; 1; (-1)^{n}\frac{b}{a^{n+1}}\Bigr) \,.
\end{equation}
Using eqs.~\eqref{eq:B_hyp_neg1} yields (with $z = (-1)^{n}b/a^{n+1}$)
\begin{equation}
\begin{split}
  \mathcal{B}(-n;1;z) &=
\frac{n}{n+1}\biggl[\, \frac{1}{n} + {}_{n}F_{n-1}\biggl(\frac{-1}{n+1},\frac{1}{n+1},\dots,\frac{n-1}{n+1};
    \frac{1}{n},\dots,\frac{n-1}{n}; -\frac{(n+1)^{n+1}}{n^{n}}z\biggr)\biggr] \,.
\end{split}
\end{equation}

\subsection{Trinomial series solutions for cubic and quintic}
Birkeland treated the trinomial equation $x^n = gx^s +\beta$ (\cite{Birkeland_algebraic_1927}, eq.~(13)).
We expressed the trinomial equation as $x^{m+n} +ax^n +b = 0$ in eq.~\eqref{eq:tri} in Sec.~\ref{sec:tri}.
The relation between the two sets of parameters is $(n,s,g,\beta) \leftrightarrow (m+n,n,-a,-b)$.
Birkeland stipulated that $n>s>0$, which corresponds to our constraint that our $m$ and $n$ are both positive integers.

Birkeland published a solution for the trinomial for arbitrary $n>s>0$.
He also published the solutions for the specific cases of the cubic ($n=3$ and $s=1$) and quintic ($n=5$ and $s=2$),
and those cases are what we treat below.
To avoid ambiguity of notation, we denote Birkeland's roots by $x_\ell$ and ours by $\xi_\ell$, where $\ell=1,\dots,n$.
(Note that $x_\ell \ne \xi_\ell$ in general, but the overall set $\{x_\ell\}$ is the same as the set $\{\xi_\ell\}$.)

We begin with the cubic, hence $\ell=1,2,3$.
Following Birkeland, define $\zeta = (27/4)(\beta^2/g^3)$. 
There are two sets of solutions for the roots, for $|\zeta|\le1$ and $|\zeta|>1$.
For $|\zeta|\le1$, our solutions are given by eqs.~\eqref{eq:trinom_soln_2} and \eqref{eq:trinom_soln_3}.
Using eq.~\eqref{eq:trinom_soln_3} yields two roots
\begin{subequations}
\begin{align}
\xi_1 &=  i\sqrt{a}\, \mathcal{B}\Bigl(-\frac12; \frac12\; -i\frac{b}{a^{3/2}}\Bigr) \,,
\\
\xi_2 &= -i\sqrt{a}\, \mathcal{B}\Bigl(-\frac12; \frac12\; i\frac{b}{a^{3/2}}\Bigr) \,.
\end{align}
\end{subequations}
Using eq.~\eqref{eq:trinom_soln_2} yields one root
\begin{equation}
\xi_3 = -\frac{b}{a}\mathcal{B}\Bigl(3; 1; -\frac{-b^2}{a^3}\Bigr) \,.
\end{equation}
From Birkeland \cite{Birkeland_algebraic_1927}, the roots are
\begin{subequations}
\begin{align}
x_1 &= \sqrt{g} \biggl[ -{}_2F_1\Bigl(-\frac16,\frac16;\frac12;\zeta\Bigr) +\frac13\sqrt{\frac{\zeta}{3}}\,{}_2F_1\Bigl(\frac13,\frac23;\frac32;\zeta\Bigr)\biggr] \,.
\\
x_2 &= \sqrt{g} \biggl[ {}_2F_1\Bigl(-\frac16,\frac16;\frac12;\zeta\Bigr) +\frac13\sqrt{\frac{\zeta}{3}}\,{}_2F_1\Bigl(\frac13,\frac23;\frac32;\zeta\Bigr)\biggr] \,.
\\
x_3 &= -\frac23\sqrt{\frac{g\zeta}{3}}\,{}_2F_1\Bigl(\frac13,\frac23;\frac32;\zeta\Bigr) \,.
\end{align}
\end{subequations}
Observe that each root $x_\ell$ is given by a single Fuss-Catalan series.
The Fuss-Catalan series for $\xi_3$ is given by integer $\mu=3$ and the root $x_3$ is given by a single hypergeometric series.
However, the roots $\xi_1$ and $\xi_2$ are given by Fuss-Catalan series with fractional $\mu=-\frac12$
and the roots $x_1$ and $x_2$ are given by a sum/difference of two hypergeometric series.
For $|\zeta|>1$, our solutions are given by eq.~\eqref{eq:trinom_soln_1}
\begin{equation}
\xi_\ell = e^{i\pi (2\ell-1)/3} b^{1/3} \mathcal{B}\Bigl(\frac13;\,\frac13;\, e^{i\pi\frac{2(2\ell-1)}{3}} \frac{a}{b^{2/3}} \Bigr) \qquad (\ell=1,2,3) \,.
\end{equation}
From Birkeland \cite{Birkeland_algebraic_1927}, the roots are (where $\omega_3$ is a primitive third root of unity)
\begin{equation}
  x_\ell = \sqrt{\frac{g}{3}}\biggl[\,\omega_3^\ell 2^{1/3}\zeta^{1/6}{}_2F_1\Bigl(-\frac16,\frac13;\frac23;\frac{1}{\zeta}\Bigr)
    +\omega_3^{2\ell} 2^{-1/3}\zeta^{-1/6}{}_2F_1\Bigl(\frac16,\frac23;\frac43;\frac{1}{\zeta}\Bigr)\biggr]   \qquad (\ell=1,2,3) \,.
\end{equation}
In this scenario, all the Fuss-Catalan series for $\xi_\ell$ are given by fractional $\mu=\frac13$
and the roots $x_\ell$ are given by weighted sums of two hypergeometric series.

For the quintic, Birkeland treated the equation $x^5 = x^2 + \beta$, i.e.~$g=1$. (Hence $a=-1$ in our notation.)
The roots are indexed by $\ell=1,\dots,5$.
Following Birkeland, we now define $\zeta = -(5^5/(2^23^3))\beta^3$. 
There are two sets of solutions for the roots, for $|\zeta|\le1$ and $|\zeta|>1$.
For $|\zeta|\le1$, our solutions are given by eqs.~\eqref{eq:trinom_soln_2} and \eqref{eq:trinom_soln_3}.
Using eq.~\eqref{eq:trinom_soln_3} yields three roots
\begin{subequations}
\begin{align}
\xi_1 &= e^{i\pi/3} a^{1/3} \mathcal{B}\Bigl(-\frac23;\frac13; e^{-i2\pi/3} \frac{b}{a^{5/3}} \Bigr) \,,
\\
\xi_2 &= -a^{1/3} \mathcal{B}\Bigl(-\frac23;\frac13; \frac{b}{a^{5/3}} \Bigr) \,,
\\
\xi_3 &= e^{-i\pi/3} a^{1/3} \mathcal{B}\Bigl(-\frac23;\frac13; e^{i2\pi/3} \frac{b}{a^{5/3}} \Bigr) \,.
\end{align}
\end{subequations}
Using eq.~\eqref{eq:trinom_soln_2} yields two roots
\begin{subequations}
\begin{align}
\xi_4 &= i\sqrt{\frac{b}{a}}\, \mathcal{B}\Bigl(\frac52;\frac12; i \frac{b^{3/2}}{a^{5/2}} \Bigr) \,,
\\
\xi_5 &= -i\sqrt{\frac{b}{a}}\, \mathcal{B}\Bigl(\frac52;\frac12; -i \frac{b^{3/2}}{a^{5/2}} \Bigr) \,.
\end{align}
\end{subequations}
From Birkeland \cite{Birkeland_algebraic_1927}, we first define the following hypergeometric functions:
\begin{subequations}
\begin{align}
F_0 &= {}_4F_3\Bigl(-\frac{1}{15}, \frac{2}{15}, \frac{8}{15}, \frac{11}{15}; \frac56, \frac13, \frac23, \zeta\Bigr) \,,
\\
F_1 &= {}_4F_3\Bigl(\frac{4}{15}, \frac{7}{15}, \frac{13}{15}, \frac{16}{15}; \frac43, \frac76, \frac23, \zeta\Bigr) \,,
\\
F_2 &= {}_4F_3\Bigl(\frac{3}{5}, \frac{4}{5}, \frac{6}{5}, \frac{7}{5}; \frac53, \frac32, \frac43, \zeta\Bigr) \,,
\\
F_3 &= {}_4F_3\Bigl(\frac{1}{10}, \frac{3}{10}, \frac{7}{10}, \frac{9}{10}; \frac76, \frac12, \frac56, \zeta\Bigr) \,.
\end{align}
\end{subequations}
We also define $\omega_3$ as a primitive third root of unity. Then the roots are
\begin{subequations}
\begin{align}
x_1 &= \omega_3 F_0 +\omega_3^2\frac{\beta}{3} F_1 -\frac{\beta^2}{3} F_2 \,,
\\
x_2 &= \omega_3^2 F_0 +\omega_3\frac{\beta}{3} F_1 -\frac{\beta^2}{3} F_2 \,,
\\
x_3 &= F_0 +\frac{\beta}{3} F_1 -\frac{\beta^2}{3} F_2 \,,
\\
x_4 &= \frac{\beta^2}{2} F_2 -(-\beta)^{1/2} F_3 \,,
\\
x_5 &= \frac{\beta^2}{2} F_2 +(-\beta)^{1/2} F_3 \,.
\end{align}
\end{subequations}
Next, for $|\zeta|>1$, the Fuss-Catalan series for the roots are given by eq.~\eqref{eq:trinom_soln_1} 
\begin{equation}
\xi_\ell = e^{i\pi (2\ell-1)/5} b^{1/5} \mathcal{B}\Bigl(\frac25;\,\frac15;\, -e^{i\pi\frac{2(2\ell-1)}{5}}\frac{1}{b^{3/5}} \Bigr) \qquad (\ell=1,\dots,5) \,.
\end{equation}
Birkeland defined the following functions:
\begin{subequations}
\begin{align}
\psi_0 &= {}_4F_3\Bigl(-\frac{1}{15}, \frac{1}{10}, \frac{3}{5}, \frac{4}{15}; \frac45, \frac25, \frac15, \frac{1}{\zeta}\Bigr) \,,
\\
\psi_1 &= {}_4F_3\Bigl(\frac{2}{15}, \frac{3}{10}, \frac{4}{5}, \frac{7}{15}; \frac65, \frac35, \frac25, \frac{1}{\zeta}\Bigr) \,,
\\
\psi_2 &= {}_4F_3\Bigl(\frac{8}{15}, \frac{7}{10}, \frac{6}{5}, \frac{13}{15}; \frac85, \frac75, \frac45, \frac{1}{\zeta}\Bigr) \,,
\\
\psi_3 &= {}_4F_3\Bigl(\frac{11}{15}, \frac{9}{10}, \frac{7}{5}, \frac{16}{15}; \frac95, \frac85, \frac65, \frac{1}{\zeta}\Bigr) \,.
\end{align}
\end{subequations}
From Birkeland \cite{Birkeland_algebraic_1927}, the roots are (where $\omega_5$ is a primitive fifth root of unity)
\begin{equation}
  x_\ell = \beta^{1/5} \biggl(\omega_5^\ell\psi_0 +\omega_5^{3\ell}\frac{\beta^{-3/5}}{5}\psi_1
  -\omega_5^{2\ell}\frac{\beta^{-9/5}}{125} \psi_2 +\omega_5^{4\ell}\frac{\beta^{-12/5}}{625} \psi_3 \biggr)   \qquad (\ell=1,\dots,5) \,.
\end{equation}
In this scenario, all the Fuss-Catalan series for $\xi_\ell$ are given by fractional $\mu=\frac25$
and the roots $x_\ell$ are given by weighted sums of four hypergeometric series.

Overall, this demonstrates that if $\mu$ is an integer, a Fuss-Catalan series is equivalent to a single hypergeometric series,
but if $\mu$ is fractional, a Fuss-Catalan series is equivalent to a weighted sum of hypergeometric series.
There is no obvious pattern for that sum of hypergeometric series.

\vfill\pagebreak

\renewcommand\thefigure{\arabic{figure}}
\vfill\pagebreak
\begin{figure}[!htb]
\centering
\includegraphics[width=1.0\textwidth]{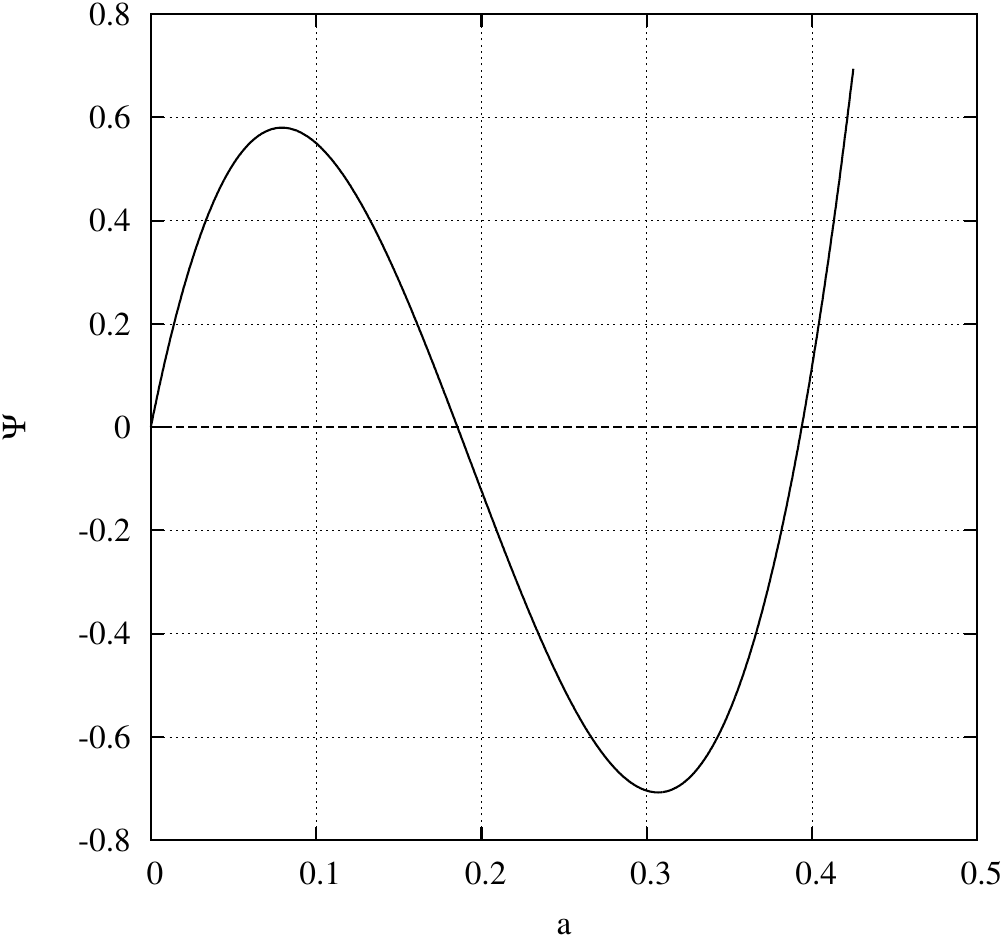}
\caption{\small
\label{fig:ray}
Graph of the discriminant $\Psi^-_{24}(-|a_0|,|a_1|)$ for $(|a_0|,|a_1|)=(a,\frac12 a)$, plotted as a function of $a$.
}
\end{figure}

\vfill\pagebreak
\begin{figure}[!htb]
\centering
\includegraphics[width=1.0\textwidth]{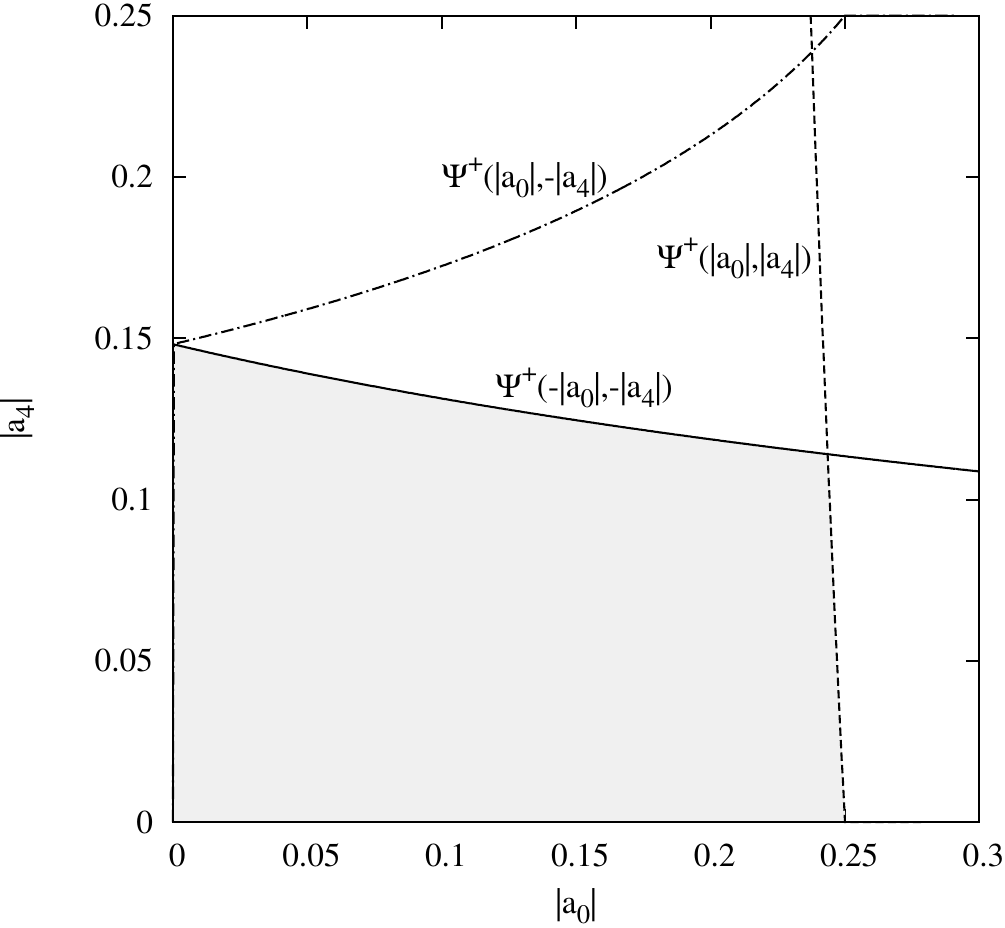}
\caption{\small
\label{fig:foliation_d12}
Graphs of the discriminant level sets 
$\Psi^+_{12}(|a_0|,|a_4|)=0$ (dashed), 
$\Psi^+_{12}(|a_0|,-|a_4|)=0$ (dotdash) and
$\Psi^+_{12}(-|a_0|,-|a_4|)=0$ (solid)
plotted in the $(|a_0|,|a_4|)$ parameter space.
Points which map into the shaded area lie in the domain of convergence. 
}
\end{figure}

\vfill\pagebreak
\begin{figure}[!htb]
\centering
\includegraphics[width=1.0\textwidth]{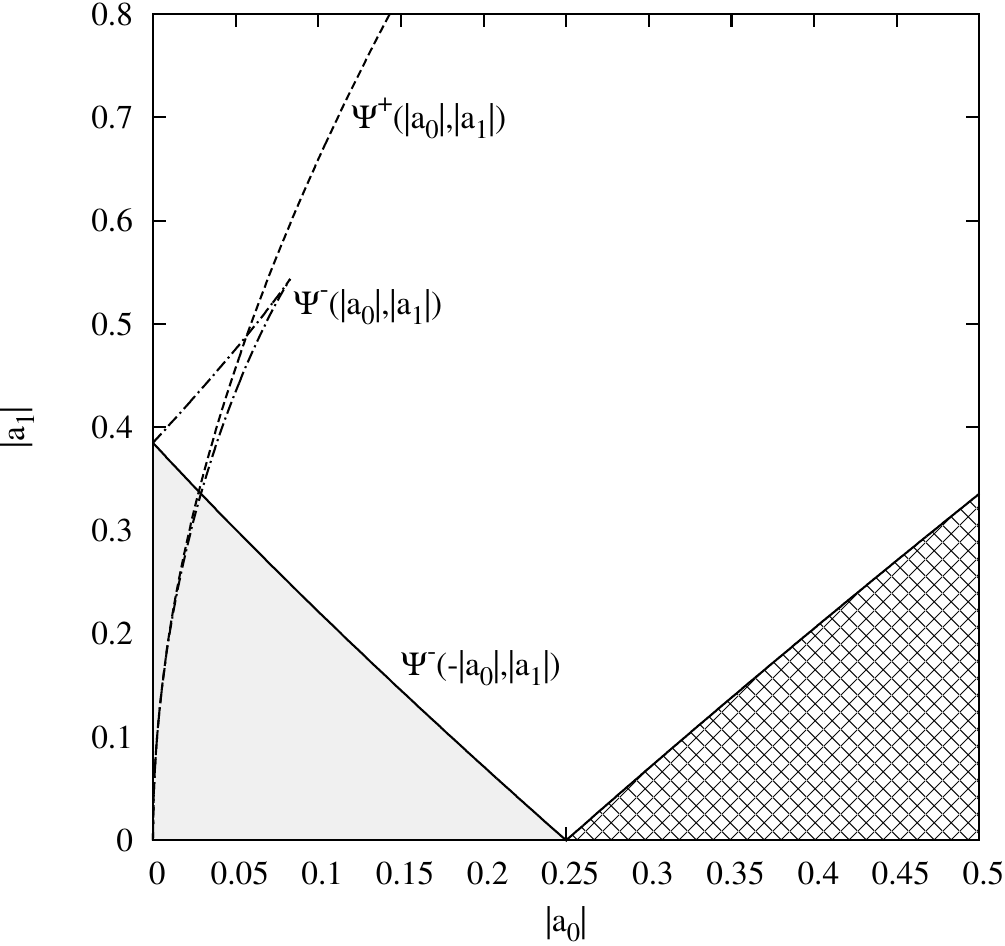}
\caption{\small
\label{fig:foliation_d24}
Graphs of the discriminant level sets 
$\Psi^+_{24}(|a_0|,|a_1|)=0$ (dashed), 
$\Psi^-_{24}(|a_0|,|a_1|)=0$ (dotdash) and
$\Psi^-_{24}(-|a_0|,|a_1|)=0$ (solid)
plotted in the $(|a_0|,|a_1|)$ parameter space.
Points which map into the shaded area lie in the domain of convergence. 
Points which map into the cross-hatched area are not in the domain of convergence. 
}
\end{figure}

\vfill\pagebreak
\begin{figure}[!htb]
\centering
\includegraphics[width=1.0\textwidth]{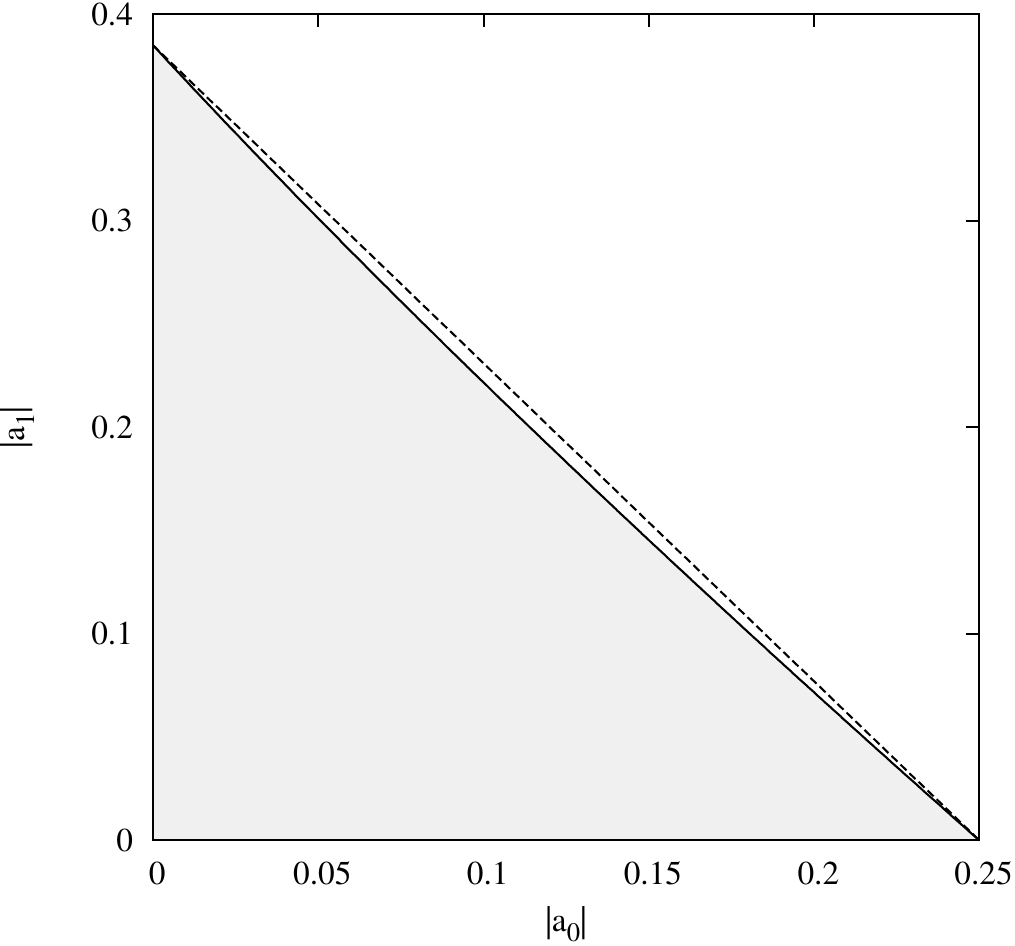}
\caption{\small
\label{fig:nonconvex_d24}
Magnified view of Fig.~\ref{fig:foliation_d24}.
The solid curve is the level set $\Psi^-_{24}(-|a_0|,|a_1|)=0$. 
The dashed line is a straight line which demarcates a right-angled triangle in the parameter space $(|a_0|,|a_1|)$.
}
\end{figure}

\end{document}